\documentclass[a4paper]{amsart}

\pdfoutput=1

\usepackage{lmodern}
\usepackage{bbm}
\usepackage{amsmath,amssymb,amsthm,url,scalerel}
\usepackage[utf8]{inputenc}
\usepackage[T1]{fontenc}
\usepackage{libertine}
\usepackage[libertine,cmintegrals,cmbraces]{newtxmath}
\usepackage{verbatim}
\usepackage[shortlabels]{enumitem}
\usepackage{stmaryrd}
\usepackage{mathtools}
\usepackage{microtype}
\usepackage{multicol}
\usepackage{xcolor}
\definecolor{darkred}{RGB}{139,0,0}
\definecolor{darkblue}{RGB}{0,0,139}
\definecolor{darkgreen}{RGB}{0,100,0}
\usepackage{hyperref}
\hypersetup{
  colorlinks   = true, 
  urlcolor     = darkred, 
  linkcolor    = darkred, 
  citecolor   = darkgreen}
\usepackage{tikz}
\usepackage{tikz-cd}
\usepackage{sseq}
\usepackage[enableskew]{youngtab}
\usepackage{ytableau}
\usepackage[capitalise]{cleveref}
\usepackage{nicefrac}
\usepackage{faktor}
\usepackage{dutchcal}

\AtBeginDocument{%
   \def\MR#1{}
}
\setenumerate[1]{label=(\roman*),} 
\setitemize[1]{label=\raisebox{0.25ex}{\tiny$\bullet$}}

\DeclareMathAlphabet{\mathpzc}{OT1}{pzc}{m}{it}

\newcommand{\Psh}{\mathrm{Psh}}
\newcommand{\Fun}{\operatorname{Fun}}
\newcommand{\Conf}{\mathrm{Conf}}
\newcommand{\Grp}{\mathrm{Grp}}
\newcommand{\Pro}{\mathrm{Pro}}
\newcommand{\fin}{\mathrm{fin}}
\newcommand{\Spin}{\mathrm{Spin}}

\newcommand{\Manf}{\mathrm{Manf}}

\newcommand{\numset}[1]{\mathbb{#1}}
\newcommand{\N}{\numset{N}}

\DeclareMathOperator{\Hom}{Hom} 

\DeclareMathOperator{\op}{op}

\makeatletter
\newcommand{\colim@}[2]{%
  \vtop{\m@th\ialign{##\cr
    \hfil$#1\operator@font colim$\hfil\cr
    \noalign{\nointerlineskip\kern1.5\ex@}#2\cr
    \noalign{\nointerlineskip\kern-\ex@}\cr}}%
}
\newcommand{\colim}{%
  \mathop{\mathpalette\colim@{\rightarrowfill@\textstyle}}\nmlimits@
}
\makeatother

\newcommand*{\defeq}{\mathrel{\vcenter{\baselineskip0.5ex \lineskiplimit0pt
                     \hbox{\scriptsize.}\hbox{\scriptsize.}}}%
                     =}

\newcommand{\Disk}{\operatorname{Disc}}

\newcommand{\Map}{\operatorname{Map}}
\newcommand{\Fr}{\operatorname{Fr}}
\newcommand{\Alg}{\operatorname{Alg}_{E_1}}
\newcommand{\Mat}{\operatorname{Mat}}
\newcommand{\hoAut}{\operatorname{hoAut}}
\newcommand{\coker}{\operatorname{coker}}

\newcommand{\fib}{\mathrm{fib}}

\newcommand{\Cat}{\mathrm{Cat}}
\newcommand{\Aut}{\operatorname{Aut}}

\newcommand*\circled[1]{\tikz[baseline=(char.base)]{
            \node[shape=circle,draw,inner sep=0.5pt] (char) {#1};}}

\tikzset{
math to/.tip={Glyph[glyph math command=rightarrow]},
loop/.tip={Glyph[glyph math command=looparrowleft, swap]},
weird/.tip={Glyph[glyph math command=Rrightarrow, glyph length=1.5ex]},
pi/.tip={Glyph[glyph math command=pi, glyph length=1.5ex, glyph axis=0pt]},
}

\DeclareMathOperator{\Int}{Int} 

\DeclareMathOperator{\Homeo}{Homeo}

\DeclareMathOperator{\Diff}{Diff}
\DeclareMathOperator{\Emb}{Emb}

\newtheorem{theorem}{Theorem}[section]

\newtheorem{corollary}[theorem]{Corollary}

\newtheorem{proposition}[theorem]{Proposition}
\newtheorem{lemma}[theorem]{Lemma}

\newtheorem{thmx}{Theorem}

\theoremstyle{definition}

\newtheorem{definition}[theorem]{Definition}
\newtheorem{example}[theorem]{Example}
\newtheorem{question}[theorem]{Question}

\newtheorem{construction}[theorem]{Construction}

\theoremstyle{remark}
\newtheorem{remark}[theorem]{Remark}

\begin{document}
\author{Fadi Mezher}
\address{Department of Mathematical Sciences, University of Copenhagen, 2100 Copenhagen, Denmark.}
\email{fm@math.ku.dk}


\title{Residual finiteness of some automorphism groups of high dimensional manifolds}

\begin{abstract}
    We show that for a smooth, closed 2-connected manifold $M$ of dimension $d \geq 6$, the topological mapping class group $\pi_0 \Homeo(M)$ is residually finite, in contrast to the situation for the smooth mapping class group $\pi_0 \Diff(M)$. Combined with a result of Sullivan, this implies that $\pi_0 \Homeo(M)$ is an arithmetic group. The proof uses embedding calculus, and is of independent interest: we show that the $T_k$-mapping class group, $\pi_0 T_k \Diff(M)$, is residually finite, for all $k \in \mathbb{N}$. The statement on the topological mapping class group is then deduced from the Weiss fibre sequence, convergence of the embedding calculus tower and smoothing theory.
\end{abstract}
\maketitle

\section{Introduction}
In his seminal work \cite{Sullivan-infinitesimal}, Sullivan shows, among other things, a structural result on mapping class groups of smooth manifolds. Combining surgery theory and rational homotopy theory, Sullivan shows that, for $M$ a closed, simply connected manifold of dimension $d \geq 5$, $\pi_0 \Diff(M)$, the group of isotopy classes of diffeomorphisms of $M$, is \emph{commensurable up to finite kernel} to an arithmetic group (see §\ref{section:arithmeticity} for more details on this equivalence relation).

The equivalence relation of commensurability up to finite kernel differs from the classical notion of commensurability, and studying this difference is initiated in \cite{Manuel-Oscar-arithmetic}. Among groups commensurable up to finite kernel to an arithmetic group, being arithmetic is equivalent to another classical group theoretic notion, that of \emph{residual finiteness}. A group $G$ is said to be residually finite if its finite residual $\mathrm{fr}(G)$, the intersection of all finite index normal subgroups of $G$, is the trivial group; see \cref{def:resfinite} for further details. Building on an example of Deligne of a $\mathbb{Z}$-extension of the symplectic group $\mathrm{Sp}_{2g}(\mathbb{Z})$ which is not residually finite, Krannich and Randal-Williams show that not all smooth mapping class groups of high dimensional manifolds are residually finite in \cite{Manuel-Oscar-arithmetic}. They consider the following example: let $W_g^n \defeq \#_g (S^n \times S^n)$ where we assume $g \geq 5$, and fix an embedding $D^{2n} \hookrightarrow W_g^n$. Extension by the identity yields a group homomorphism
$$\pi_0 \Diff_\partial (D^{2n}) \to \pi_0 \Diff(W_g^n)$$
where the domain may be identified with the group of exotic spheres $\Theta_{2n+1}$. Whenever $n = 5\pmod{8}$, Krannich and Randal-Williams show that the subgroup $\mathrm{bP}_{2n+2}$ of exotic spheres bounding a parallelisable compact manifold embeds into $\mathrm{fr}(\pi_0 \Diff(W_g^n))$. For these values of $n$, $\mathrm{bP}_{2n+2}$ is known to be non-trivial \cite{Kervaire-Milnor}; this thus yields a family of examples of closed smooth manifolds with non-residually finite smooth mapping class groups.
\subsection*{Main results}
Given the smooth nature of the above counterexamples, one may wonder what happens in the category of topological manifolds. The following three theorems constitute three of the main results of this work
\begin{thmx}\label{thmA}
    Let $M$ be a smoothable, closed 2-connected topological manifold of dimension $d \geq 6$. Then, $\pi_0 \Homeo(M)$ and $\pi_0 \Homeo^+ (M)$ are residually finite groups. For $W$ a smoothable, 2-connected compact manifold of dimension $d \geq 5$, such that $\partial W \neq \varnothing$, then $\pi_0 \Homeo_\partial (W)$ is residually finite.
\end{thmx}
This is shown in the text as \cref{thm:homeo-res-finite} for the closed case, and \cref{thm:relbdryresfinite} for the boundary case. As a corollary, we obtain the following (in the text \cref{thm:homeo-arithmetic})
\begin{thmx}\label{thmB}
    Let $M$ be a smoothable, closed 2-connected topological manifold of dimension $d \geq 6$. Then, $\pi_0 \Homeo(M)$ and $\pi_0 \Homeo^+ (M)$ are arithmetic groups.
\end{thmx}
We furthermore show that the failure of residual finiteness of smooth mapping class groups is entirely due to exotic spheres, as occurring in Krannich and Randal-Williams' example. Indeed, we show the following (as \cref{thm:fin-res-smooth} in the text)
\begin{thmx}\label{thm:new}
    Let $M$ be a smooth, closed 2-connected manifold of dimension $d \geq 6$, and fix an embedded disc $D^d \subset M$. Then,
    \begin{enumerate}[(i)]
        \item $\mathrm{fr}(\pi_0 \Diff^+ (M)) \subseteq \mathrm{Im}(\pi_0 \Diff_\partial(D^d) \to \pi_0 \Diff^+(M))$, where the map $\pi_0 \Diff_\partial (D^d) \to \pi_0 \Diff^+(M)$ is given by extension by the identity;
        \item $\mathrm{fr}(\pi_0 \Diff(M)) \subseteq \mathrm{Im}(\pi_0 \Diff_\partial(D^d) \to \pi_0 \Diff(M))$, where the map $\pi_0 \Diff_\partial (D^d) \to \pi_0 \Diff(M)$ is given by extension by the identity
    \end{enumerate}
\end{thmx}
\subsection*{Outline of proof} The proof of the above statements is broken into various key steps, which we outline here. We begin by pointing out that \cref{thmA} implies both \cref{thmB} and \cref{thm:new}: see §\ref{section:arithmeticity} for the first implication, and the proof of \cref{thm:fin-res-smooth} for the second. We therefore focus on presenting an outline of the proof of \cref{thmA}.
\subsubsection*{Weiss fibre sequences} The first key tool is the Weiss fibre sequence, as developed in \cite{Sander-finiteness,Weiss-Dalian}. For a smooth manifold $M$ of dimension $d$, let $M^\circ$ denote the manifold with boundary $M \setminus \mathrm{int}(D^d)$, where $D^d \hookrightarrow M$ is a fixed embedding of a $d$-dimensional disc. We then have a commutative diagram

\begin{center}
    \begin{tikzcd}
        B\Diff_{\partial}(D^d) \arrow{r} \arrow{d} & B\Diff_\partial(M^\circ) \arrow{r} \arrow{d} & B\Emb_{\partial/2}^{\cong}(M^\circ,M^\circ) \arrow{d}{\Psi}
        \\
        \ast \simeq B\Homeo_\partial(D^d) \arrow{r} & B\Homeo_{\partial}(M^\circ) \arrow{r}[swap]{\simeq} & B\Emb_{\partial/2}^{\mathrm{Top},\cong}(M^\circ, M^\circ)
    \end{tikzcd}
\end{center}
where both rows are fibre sequences, the smooth and topological \emph{Weiss fibre sequences} (§\ref{section:Weiss-fibre-sequence}), and where $\Emb_{\partial/2}^{\cong}(M^\circ,M^\circ)$ denotes the space of smooth embeddings that are isotopic to a diffeomorphism, and that restrict to the identity on a neighbourhood around the lower hemisphere $D^{d-1}_{-} \subset S^{d-1}\cong \partial M^{\circ}$; $\Emb_{\partial/2}^{\mathrm{Top},\cong}$ is its topological analogue. By the Alexander trick, $B\Homeo_{\partial}(D^d)$ is contractible, so that $B\Homeo_{\partial}(M^\circ) \simeq B\Emb_{\partial/2}^{\mathrm{Top},\cong}(M^\circ, M^\circ)$. Following \cite{Sander-finiteness,Weiss-Dalian}, we study the space of self-embeddings relative half the boundary using embedding calculus.
\bigskip

\subsubsection*{Embedding calculus} We consider the $\infty$-category $\Manf_d$ consisting of $d$-dimensional smooth manifolds with empty boundary as objects, and spaces of embeddings as 1-morphisms. We can consider the full subcategory of $\Manf_d$, $\Disk_d$, on those objects of the form $\sqcup_\ell \mathbb{R}^d$, which further decomposes into full subcategories $\{\Disk_d^{\leq k}\}_{k \in \mathbb{N}}$, where for a given $k \in \mathbb{N}$, the objects of $\Disk_d^{\leq k}$ are precisely $\sqcup_{\ell} \mathbb{R}^d$, for $\ell \leq k$. We denote the inclusions $\Disk_d^{\leq k} \hookrightarrow \Manf_d$ by $\iota_k$, for $k \in \mathbb{N} \cup \{\infty\}$. Thus, we define a tower of functors via
$$\iota_k^\ast E\colon \Manf_d \xrightarrow[]{h} \Psh(\Manf_d) \xrightarrow[]{\iota_k^\ast} \Psh(\Disk_d^{\leq k})$$
with the convention that $\iota_\infty^\ast E_M$ is simply denoted by $E_M$. Concretely, given a smooth $d$-manifold $M$, $E_M$ is the disc-presheaf $\sqcup_\ell \mathbb{R}^d \mapsto \Emb(\sqcup_\ell \mathbb{R}^d, M)$. For each $k \in \mathbb{N}\cup\{\infty\}$, and smooth manifolds $M,N \in \Manf_d$, we define
$$T_k \Emb(M,N) \defeq \Map_{\Psh(\Disk_d^{\leq k})}(E_M,E_N)$$
and
$$T_k\Diff(M) \defeq \Map_{\Psh(\Disk_d^{\leq k})}^{\simeq}(E_M,E_M)$$
On mapping spaces, the above tower of functors yields the \emph{embedding tower}. This consists of a tower $\{T_k \Emb(M,N)\}_{k \in \mathbb{N}}$, receiving a map from $\Emb(M,N)$. We say the embedding tower \emph{converges} if the map
$$\Emb(M,N) \to T_\infty \Emb(M,N) \simeq \lim_{k} T_k \Emb(M,N)$$
is an equivalence. By the celebrated result of Goodwillie-Weiss \cite[Corollary 2.5]{Goodwillie-Weiss}, this holds in the case where $h\dim(M) \leq d-3$, where $h\dim$ denotes the handle dimension of $M$.
\\

A setup for embedding calculus relative boundary can be formulated, and the analogous convergence result holds; see, for instance, \cite[Chapter 5]{Goodwillie-Weiss} or \cite[§3.3.2]{Sander-finiteness}. We shall make use of this setup as follows. Given $M$ a closed, smooth 2-connected manifold of dimension $d \geq 6$, a standard Morse theoretic argument shows that $M^\circ \defeq M \setminus \mathrm{int}(D^d)$ admits a handle decomposition of handle dimension at most $d-3$, relative half the boundary; thus, the embedding calculus tower for $\Emb_{\partial/2}(M^\circ,M^\circ) \simeq \Emb_{\partial}(M^\ast,M^\ast)$, where $M^\ast \defeq M^\circ \setminus \partial/2$, converges. Consequently, one can infer information on $\Emb_{\partial/2}^{\cong}(M^\circ , M^\circ)$ from embedding calculus. As convergence of embedding calculus is not known in the topological category, we restrict ourselves to the case of smoothable manifolds and use smoothing theory to infer information in the topological setting. As the class of residually finite groups is closed under taking limits, the convergence of embedding calculus allows us to infer that $\pi_0 \Emb_{\partial/2}^{\cong}(M^\circ,M^\circ)$ is residually finite, by showing that $\pi_0 T_k \Emb_{\partial/2}^{\cong}(M^\circ,M^\circ)$ is residually finite. The following is the technical heart of this paper (in the text as \cref{thm:spin-T_k-residual-finite,thm:main-result-resfinite,prop:resfinite-boundary}).
\begin{thmx} \label{thmC}
    For $M$ a smooth, closed 2-connected manifold of dimension $d \geq 5$, and for any $k \in \mathbb{N}$, the groups $\pi_0 T_k \Diff(M)$ and $\pi_0 T_k \Emb_{\partial/2}^{\cong}(M^\circ,M^\circ)$ are residually finite.
\end{thmx}
\bigskip

\subsubsection*{Profinite homotopy theory} Theorem \ref{thmC} above builds on a known result, observed by Serre in \cite[p. 108]{Serre-arithmetic}, as a corollary of Sullivan's work on profinite completion of spaces, namely \cite[thm 3.2]{Sullivan-genetics}; we give a careful and modern exposition of its proof, that is adaptable to embedding calculus.

\begin{thmx}[Serre-Sullivan]
    Let $X$ be a finite, simply connected CW-complex. Then, $\pi_0 \hoAut(X)$ is a residually finite group.
\end{thmx}

As alluded to in the above, this follows from \emph{profinite homotopy theory}. Roughly speaking, profinite homotopy theory attempts to lift the group profinite completion functor to spaces. We begin by briefly recalling group profinite completion. We consider the functor $\widehat{(-)}\colon \Grp \to \Pro(\Grp^{\fin})$, where $\Grp$ denotes the 1-category of groups and group morphisms, and $\Grp^{\fin}$ the full subcategory thereof, consisting of finite groups; the functor is defined by sending a group $G$ to the functor $F \mapsto \Hom(G,F) \in \Fun(\Grp^\fin,\Grp)^{\op}$. The latter is classically denoted by $\widehat{G}$, and coincides with the cofiltered system of finite quotients of $G$. The above functor admits a right adjoint $\Mat^g$, the \emph{materialisation functor}, sending a pro-object of finite groups $F$ to its limit in groups. The composite $\Mat^g \circ \, \widehat{(-)} \colon \Grp \to \Grp$ we denote by $\Phi^g$. The unit of the adjunction $\widehat{(-)} \dashv \Mat^g$ yields, for every group $G$, a map $\eta_G \colon G \to \Phi^g G$, called the \emph{profinite completion map}. More concretely, $\Phi^g G$ is the limit over the diagram $\{G/N_i\}$ over $\{N_i\}$, the cofiltered system of finite index subgroups of $G$. Thus, $\ker(\eta_G \colon G \to \Phi^g)$ coincides with the finite residual $\mathrm{fr}(G)$; consequently, a group $G$ is residually finite if and only if the map $\eta_G \colon G \to \Phi^g G$ is injective, which is the key connection between residually finite groups and group profinite completion.
\\

As is the case for groups, one can also approximate spaces by certain spaces with a finiteness condition. We say a space $X$ is \emph{$\pi$-finite} if
\begin{itemize}
    \item $\pi_0 X$ is a finite set;
    \item There exists some $N \in \mathbb{N}$ such that for all $x \in X$, and all $k \geq N$, $\pi_k(X,x)=0$;
    \item $\pi_n (X,x)$ is finite for all $n\in \mathbb{N}$ and $x \in X$.
\end{itemize}
We denote $\mathcal{S}_\pi \subset \mathcal{S}$ for the full subcategory of the $\infty$-category of spaces consisiting of $\pi$-finite spaces. We define a functor $\widehat{(-)} \colon \mathcal{S} \to \Pro(\mathcal{S}_\pi)$, $X \mapsto (F \mapsto \Map_{\mathcal{S}}(X,F))$. This functor similarly admits a right adjoint $\Mat\colon \Pro(\mathcal{S}_\pi) \to \mathcal{S}$, obtained by applying the limit in $\mathcal{S}$ of an object $X \in \Pro(\mathcal{S}_\pi)$. We denote the composite $\Mat \circ \,\widehat{(-)}$ by $\Phi^s$, which we call the \emph{finite completion functor} following Sullivan. The unit of the adjunction $\widehat{(-)} \dashv \Mat$ yields, for each $X$, a natural map $\eta_X \colon X \to \Phi^s X$. Residual finiteness of $\pi_0 \hoAut(X)$ then follows from \cite[thm 3.2]{Sullivan-genetics}, stating that the map
$$\Map(Y,Z) \to \Map(Y,\Phi^s Z)$$
is $\pi_0$-injective, for nice enough spaces $Y$ and $Z$. The key input of Sullivan's proof of the above statement, which we study in some detail, is that the functor $\Phi^s$ comes close to preserving finite limits when restricted to nice enough spaces, namely componentwise nilpotent spaces of finite type (\cref{thm:finite-limit-finite-completion} and \cref{cor:finite-limits-finite-completion}).
\begin{thmx}
    Let $\mathcal{F} \colon \mathcal{D} \to \mathcal{S}^{\mathrm{nil,ft}}$ be a functor from a finite $\infty$-category to the category of componentwise nilpotent spaces of finite type. Then, the canonical coassembly morphism
    $$\mathrm{coass} \colon \Phi^s\lim_{\mathcal{D}}\mathcal{F} \to \lim_{\mathcal{D}}\Phi^s \mathcal{F}$$
    induces an isomorphism on $\pi_n$, for all $n \geq 1$, based at points coming from $\lim_{\mathcal{D}} \mathcal{F}$ via the maps of the following commutative diagram
    \begin{center}
        \begin{tikzcd}
            & \lim_{\mathcal{D}}\mathcal{F} \arrow{dl}[swap]{\eta} \arrow{dr}{\lim \eta}
            \\
            \Phi^{s}\lim_{\mathcal{D}}\mathcal{F} \arrow{rr}[swap]{\mathrm{coass}} && \lim_{\mathcal{D}}\Phi^s \mathcal{F}
        \end{tikzcd}
    \end{center}
    where all limits are taken in $\mathcal{S}$.
\end{thmx}

We now fix $M$ a closed, 2-connected smooth manifold, and study residual finiteness of $\pi_0 T_k \Diff(M)$. Given $\iota_k^\ast E_M \in \Psh(\Disk_d^{\leq k})$, we may consider the finite and profinite completions pointwise. As in the setting of homotopy automorphisms, residual finiteness follows from the $\pi_0$-injectivity of the map 
$$\Map_{\Psh(\Disk_d^{\leq k})}(\iota_k^\ast E_M,\iota_k^\ast E_M) \to \Map_{\Psh(\Disk_d^{\leq k})}(\iota_k^\ast E_M , \iota_k^\ast \Phi^s E_M)$$
given by composition with the map $\iota_k^\ast E_M \to \iota_k^\ast \Phi^s E_M$. This is shown in the text as \cref{thm:injective-pi_0}, and occupies the technical heart of the paper.
\bigskip
\subsubsection*{Smoothing theory} Convergence of the embedding tower thus implies that $\pi_0 \Emb_{\partial/2}^{\cong}(M^\circ, M^\circ)$ is residually finite. In order to deduce the same for the topological analogue, we use smoothing theory as follows: the fibre of the map 
$$ B\Emb_{\partial/2}^{\cong}(M^\circ,M^\circ) \to B\Emb_{\partial/2}^{\cong,\mathrm{Top}}(M^\circ,M^\circ) \simeq B\Homeo_{\partial}(M^\circ)$$
can be described as a collection of components of the space of sections $\Gamma_{\partial/2}(\xi_M)$ of a bundle over $M^\circ$ with fibres $\mathrm{Top}(d)/O(d)$, relative some fixed section over $\partial/2$. One can show that $\Gamma_{\partial/2}(\xi_M)$ has finitely many components, and that on each of these components, the fundamental group is a finite group. The long exact sequence on homotopy groups associated to the above fibre sequence thus yields an exact sequence
$$\pi_1 \Gamma_{\partial/2}(\xi_M) \to \pi_0 \Emb_{\partial/2}^{\cong}(M^\circ,M^\circ) \to \pi_0 \Homeo_{\partial}(M^\circ) \to \pi_0 \Gamma_{\partial/2}(\xi_M)$$
where exactness on the last term is in the category of pointed sets. It now follows from elementary group theoretic arguments (\cref{lemma:quotient-resfinite,lemma:finite-index-resfin}) that $\pi_0 \Homeo_{\partial}(M^\circ)$ is residually finite. Finally, to glue back in the deleted disc, we use parametrized isotopy extension again to obtain a fibre sequence
$$\Homeo_{\partial}(M^\circ) \to \Homeo(M) \to \Emb^{\mathrm{Top}}(D,M)$$
where the base space of the fibration is equivalent to the topological frame bundle of $M$, which in particular has $\pi_0 \Fr^\mathrm{Top} (M) \cong \mathbb{Z}/2\mathbb{Z}$, and $\pi_1 \Fr^\mathrm{Top} (M) \cong \mathbb{Z}/2\mathbb{Z}$; again, applications of \cref{lemma:quotient-resfinite,lemma:finite-index-resfin} implies $\pi_0 \Homeo(M)$ is indeed residually finite.

\subsection*{Remark on notation} It is common to use the notation $\widehat{X}$ to denote profinite completion of an object $X$, be it a set, group or space. In this work, the notation $\widehat{X}$ is only reserved for $\Pro$-objects. To make clear the distinction between the different types of objects in use, the endofunctor on our favorite category obtained after materialisation of the profinite completion is denoted $\Phi$, with a decoration to emphasize which object we are looking at. The choice of $\Phi$ is to remind us that this functor is called \emph{finite completion}, following Sullivan. We will mainly be interested in three functors:
\begin{itemize}
    \item $\Phi^g \colon \Grp \to \Grp$;
    \item $\Phi^{tg}\colon \Grp \to \mathrm{Top}\Grp$;
    \item $\Phi^s\colon \mathcal{S} \to \mathcal{S}$
\end{itemize}
the first of which is the group finite completion, whereas the second also remembers the canonical topological group structure obtained on the finite completion $\Phi^g G$ of a group $G$, and the last is the finite completion functor on spaces.
\subsection*{Conventions} Throughout this paper, we use the language of $\infty$-categories. However, for the more geometrically minded reader, $\infty$-categories, while being sometimes more convenient, can be substituted by more classical homotopy theoretic language without much trouble. The reader is encouraged to think of classical homotopy theory for the idea and intuition, and $\infty$-categories for the convenience. We will throughout attempt to point out classical analogues of some stated $\infty$-categorical facts or constructions.
\\

Furthermore, throughout the work, the sets of smooth embeddings and diffeomorphisms of smooth manifolds are endowed with the $C^\infty$-topology, while sets of topological embeddings and homeomorphisms are endowed with the compact-open topology.

\subsection*{Acknowledgements} I would like to thank my PhD supervisor Søren Galatius for suggesting the study of residual finiteness through embedding calculus, many encouraging and engaging discussions, and for careful readings of previous versions of this paper. I would also like to thank Thomas Blom, Robert Burklund, Manuel Krannich, Alexander Kupers, Jan Steinebrunner and Tim With Berland for many helpful discussions related to this work. I would also like to thank Manuel Krannich, Az\'elie Picot and Nathalie Wahl for reading previous drafts of this manuscript. I was supported by the Danish National Research Foundation through the Copenhagen Center for Geometry and Topology (DNRF151).

\tableofcontents

\section{Homotopy automorphisms}
In this section, we lay out a careful exposition for a proof of the following theorem, with a view towards later generalizations. This theorem was observed by Serre in \cite[p. 108]{Serre-arithmetic}, building on work of Sullivan, namely \cite[thm 3.2]{Sullivan-genetics}.

\begin{theorem}[Serre-Sullivan]\label{thm:Serre-Sullivan}
    For $X$ a simply connected, finite CW complex, the group $\pi_0 \hoAut(X)$ is residually finite.
\end{theorem}

As a build-up for the proof, we review some constructions on groups and spaces that are crucial for the argument: the classical theory of profinite groups and group profinite completion, Sullivan's theory of profinite completions of spaces, along with a functorial setup for some classical constructions in homotopy theory.

\subsection{Profinite completion of groups}

One approach to studying groups is by approximating them by their finite quotients. For a group $G$, consider the cofiltered system $\{G/N\}$, where $N$ ranges over all finite index, normal subgroups of $G$. Equivalently, one may look at all group morphisms $G \to F$, where $F$ is a finite group. The categorical notion that captures this type of object is called a \emph{Pro-object}; namely, if $\mathrm{Grp}^{\mathrm{fin}}$ denotes the 1-category of finite groups and group morphisms, a Pro-object in $\mathrm{Grp}^{\mathrm{fin}}$ is then an object in $\Fun(\mathrm{Grp}^{\mathrm{fin}},\mathrm{Set})^{\op}$, that preserves finite limits. The \emph{(group) profinite completion} functor is the functor $\mathrm{Grp} \to \Pro(\mathrm{Grp}^\mathrm{fin})$, $G \mapsto (F \mapsto \Hom(G,F))$. The \emph{group materialisation functor} is defined to be the functor $\Mat^g \colon \Pro(\mathrm{Grp}^{\mathrm{fin}}) \to \mathrm{Grp}$, sending a pro-object of finite groups to its limit in groups, and we have an adjunction $\widehat{(-)} \dashv \Mat^g$. The composite
$$\Grp \to \Pro(\Grp^\fin) \to \Grp$$
is denoted by $\Phi^{g}$. The unit of the above adjunction gives a natural transformation $\mathrm{id} \Rightarrow \Phi^g$, and for a given group $G$, we denote by $\eta_G$ the associated map $\eta_G \colon G \to \Phi^g G$. Note that for a group $G$, the group $\Phi^g G$ carries a canonical structure of a topological group, given by the subgroup topology it inherits from $\Pi (G/N)$, where the product is taken over the same indexing set as above, namely over all finite quotients of $G$, and where each factor is endowed with the discrete topology. When endowed with this topology, we obtain a functor $\Grp \to \mathrm{Top}\Grp$, the category of topological groups and continuous group homomorphisms; this functor is the \emph{finite completion} functor of groups, which we shall denote $\Phi^{tg}$ (classically, $\Phi^{tg}$ is referred to as profinite completion, and denoted by $\widehat{(-)}$; we call it finite completion to stay consistent with Sullivan's terminology later). Such a topological group has a fairly explicit description in terms of its topology: it is compact, Hausdorff and totally disconnected. We thus define
\begin{definition}
    A topological group $G$ is said to be \emph{profinite} if it is a compact, Hausdorff totally disconnected group.
\end{definition}
Let $\mathrm{ProfGrp}$ be the full subcategory of $\mathrm{Top}\Grp$ consisting of profinite topological groups. The functor $\Phi^{tg}$ does indeed land in the above full subcategory, as can be seen through the following classical result
\begin{theorem}
    A topological group is profinite if and only if it can be written as a cofiltered limit of finite groups (with its canonical topology).
\end{theorem}
We also obtain an adjunction $\Phi^{tg} \dashv \mathrm{fgt}$ between the group finite completion functor and the forgetful functor. This adjunction thus allows us to write the following universal property

\begin{theorem}[Universal property of group finite completion]\label{thm:group-finite-completion-universal-ppty}
    Let $G$ be a group, and let $H$ be a profinite topological group. Given a group homomorphism $f\colon G \to H$, there exists a unique continuous group homomorphism $\Phi^{tg} G \to H$ making the following diagram
    \begin{center}
        \begin{tikzcd}
            G \arrow{r}\arrow{dr}[swap]{f} &  \Phi^{tg}G \arrow[dashed]{d}
            \\
             & H
        \end{tikzcd}
    \end{center}
    commute.
\end{theorem}

\bigskip

\subsubsection{Residual finiteness}
The functor $\Phi^g$ can furthermore be used to define \emph{residual finiteness}, the main group theoretic notion under investigation in this study.
\begin{definition}\label{def:resfinite}
    A group $G$ is said to be \emph{residually finite} if the canonical group homomorphism
    $$G \to \Phi^gG$$
    is injective; equivalently, if the intersection of all finite index, normal subgroups of $G$ is the trivial subgroup. In other words, given any $g \in G$ such that $g \neq 1$, there exists some finite group $F_g$ and a group homomorphism $\varphi_g \colon G \to F_g$ such that $\varphi_g (g) \neq 1 \in F_g$; we will be informally referring to this property by saying that $g \in G$ can be \emph{detected by finite groups}.
\end{definition}
For $G$ a group, we denote $\mathrm{fr}(G)$ for the \emph{finite residual} of $G$, namely, the intersection of all finite index, normal subgroups of $G$. From the definitions, it follows that $\mathrm{fr}(G)=\ker(G \to \Phi^g G)$, and that $G$ is residually finite if and only if $\mathrm{fr}(G)$ is the trivial subgroup.
\begin{example}
    The following classes of groups are residually finite
    \begin{itemize}
        \item Finite groups;
        \item Free groups;
        \item Finitely generated nilpotent groups;
        \item Finitely generated subgroups of general linear groups;
    \end{itemize}
\end{example}
The class of residually finite groups is closed under taking subgroups and inverse limits. However, it is not closed under extensions: 
\begin{example}
    The following example is due to Deligne \cite{Deligne}. We consider the Lie group $\mathrm{Sp}_{2g}(\mathbb{R})$. We recall that $\pi_1 \mathrm{Sp}_{2g}(\mathbb{R}) \cong \mathbb{Z}$; let $\widetilde{\mathrm{Sp}}_{2g}(\mathbb{Z})$ denote the pullback of the universal cover of $\mathrm{Sp}_{2g}(\mathbb{R})$ along the inclusion map $\mathrm{Sp}_{2g}(\mathbb{Z}) \to \mathrm{Sp}_{2g}(\mathbb{R})$. We thus obtain a central extension
    $$1 \to \mathbb{Z} \to \widetilde{\mathrm{Sp}}_{2g}(\mathbb{Z}) \to \mathrm{Sp}_{2g}(\mathbb{Z}) \to 1$$
    For $g \geq 2$, Deligne shows that the finite residual of $\widetilde{\mathrm{Sp}}_{2g}(\mathbb{Z})$ is $2\mathbb{Z}$.
\end{example}
The above example was used in \cite{Manuel-Oscar-arithmetic}, to show that $\pi_0 \Diff(\#_g (S^n \times S^n))$ is not residually finite for certain $n \in \mathbb{N}$. Combining their result with one of the main results of this paper (\cref{thm:homeo-res-finite}), we obtain a similar example:
\begin{example}
    Let $n = 5\pmod{8}$ and $g \geq 5$, so that $\pi_0 \Diff(W_g^n)$ is not residually finite, where $W_g^n \defeq \#_g (S^n \times S^n)$. One of the main results of this work, which is shown much later, namely \cref{thm:homeo-res-finite}, implies that $\pi_0 \Homeo(W_g^n)$ is residually finite. By smoothing theory, we can describe the fibre $F$ of the forgetful map 
    $$B\Diff(W_g^n) \to B\Homeo(W_g^n)$$
    as a collection of components of a space of sections $\Gamma(\xi)$ of a bundle $\xi$ over $W_g^n$ with fibres $\mathrm{Top}(d)/O(d)$ as in \cite[Essay V. §3,4]{Kirby-Siebenmann}, which in particular has finite $\pi_1$ on each of its components. Thus, we obtain an exact sequence
    $$\pi_1 F \to \pi_0 \Diff(W_g^n) \to \pi_0 \Homeo(W_g^n)$$
    and restricting to the image and kernel, we obtain an extension of a residually finite group by a finite group, which is not residually finite.
\end{example}

The following two elementary lemmas concern certain properties of that class of groups, that prove to be necessary in some arguments later.

\begin{lemma}\label{lemma:quotient-resfinite}
    Let $G$ be a residually finite group, and let $N$ be a finite normal subgroup. Then, $G/N$ is again residually finite.
\end{lemma}
\begin{proof}
    Let $x \in G/N$ be a non-zero element; we work to find a normal finite index subgroup of $G/N$ that does not contain $x$. Consider the set $\{\widetilde x_i\}_{i \in \{1,\cdots,n\}}$ of lifts of $x$ via the canonical morphism $p:G \to G/N$, where $n$ is the order of the finite group $N$. Since $G$ is residually finite, it follows that for each $i$, there exists $H_i$ normal, finite index subgroup of $G$ that does not contain $\widetilde x_i$. Consequently, $H \defeq \cap_i H_i$ is again a finite index normal subgroup with the further property that $\widetilde x_i \notin H$, for all $i$. Then, $p(H)$ is a normal subgroup of $G/N$ which does not contain $x$. To see that $p(H)$ is a finite index subgroup, we observe that 
    $$\faktor{\left(\faktor{G}{N}\right)}{p(H)}=\faktor{\left(\faktor{G}{N}\right)}{\left(\faktor{NH}{N}\right)} \cong \faktor{G}{NH}$$
    where the latter is finite as it receives a surjection from the finite group $G/H$.
\end{proof}

We next see that residual finiteness can be already detected at the level of finite index subgroups.

\begin{lemma}\label{lemma:finite-index-resfin}
    Let $G$ be a group, and let $H$ be a finite index subgroup of $G$ that is residually finite. Then, $G$ is also residually finite.
\end{lemma}
\begin{proof}
    As normality is not transitive for subgroups, we use the equivalent formulation of residual finiteness: we have to check that the intersection of all finite index subgroups of $G$ is the trivial group. Let $g\neq0 \in G$. If $g \notin H$, then $H$ is a finite index subgroup of $G$ not containing $g$. Otherwise, $g \in H$, and $g\neq 0$. Since $H$ was itself residually finite, there exists some finite index subgroup $K$ of $H$ such that $g \notin K$, and as $H$ is finite index in $G$, then so too is $K$. Consequently, $K$ is a finite index subgroup of $G$ not containing $g$, and thus $G$ is residually finite.
\end{proof}

\subsubsection{Exactness properties of profinite completions of groups}
In this section, we study the exactness of the profinite completion functor of groups, on a certain subclass of groups. The result is stated as \cite[thm 2.2]{Hilton-Roitberg}, which they credit to Bousfield and Kan \cite{Bousfield-Kan}, a proof of which can be also found in \cite[§5.3]{Schneebeli}. Let $\mathrm{Nil}\Grp^{fg}$ be the full subcategory of $\Grp$ consisting of finitely generated, nilpotent groups.
\begin{theorem}\label{thm:Hilton-Roitberg}
    The functors $\Phi^g$ and $\Phi^{tg}$ are exact when restricted to $\mathrm{Nil}\Grp^{fg}$.
\end{theorem}
\begin{remark}
Exactness in the above theorem means that the functor sends short exact sequences to short exact sequences. For the category of compact Hausdorff topological groups, a short exact sequence is a triple $(G_1,G_2,G_3)$ of topological groups together with consecutive continuous morphisms
$$1 \to G_1 \xrightarrow[]{f_1} G_2 \xrightarrow[]{f_2} G_3 \to 1$$
that is an exact sequence on the underlying groups. It was shown in the above references that $\Phi^{tg}$ sends a short exact sequence of finitely generated nilpotent groups to a short exact sequence of compact Hausdorff topological groups. Among the topological groups that we are interested in, namely topologically finitely generated profinite groups, the two notions of exactness agree: given a triple $(K,G,H)$ of topologically finitely generated profinite groups, and maps $K \to G \to H$, exactness as underlying discrete groups is equivalent to exactness in the category of compact Hausdorff groups. This can be seen as a corollary of \cite[thm 1.1]{Nikolov-Segal}. However, we strive to make our presentation independent of \emph{loc. cit.}, and focus on the functor $\Phi^g$.
\end{remark}

\bigskip

For the purposes of later application, we show that \cref{thm:Hilton-Roitberg} implies that arbitrary long exact sequences of finitely generated nilpotent groups are indeed sent to long exact sequences by the group profinite completion functor. Such considerations would be automatic should the category $\mathrm{Nil}\Grp^{fg}$ be an abelian category, which it is not, and consequently further arguments are in order.

\begin{lemma}\label{lemma:long-exact}
    Consider a long exact sequence
    $$\cdots \xrightarrow[]{\varphi_2} G_2 \xrightarrow{\varphi_1} G_1 \xrightarrow[]{\varphi_0} G_0$$
    of finitely generated nilpotent groups. Then, the sequence
    $$\cdots \xrightarrow[]{\Phi^g\varphi_2} \Phi^g G_2 \xrightarrow{\Phi^g\varphi_1}\Phi^g G_1 \xrightarrow[]{\Phi^g\varphi_0} \Phi^g G_0$$
    is again exact.
\end{lemma}
\begin{proof}
    We begin by showing that $\Phi^g$ preserves injections on the category of finitely generated, nilpotent groups. This is clear when restricted to finitely generated abelian groups. Fix now finitely generated nilpotent groups $H$ and $G$, and an injective group morphism $1 \to H \xrightarrow[]{\varphi} G$. As $G$ is a finitely generated nilpotent group, it can be obtained by a finite iteration of central extensions, starting from finitely generated abelian groups; in other words, there exists a finite collection $\{A_0, A_0'\}\cup\{A_i\}_{1 \leq i\leq n}$ of finitely generated abelian groups and finitely generated groups $\{G_i\}_{0 \leq i \leq n}$ that fit into central extensions
    $$1 \to A_0' \to G_0 \to A_0 \to 1$$
    and for all $1 \leq i \leq n-1$, central extensions
    $$1 \to A_i \to G_{i+1} \to G_i \to 1$$
    and finally, such that $G_n \cong G$. We proceed by induction. We say that each $G_i$ is $(i+1)$-centrally extended from abelian groups: i.e. $G_0$ is once centrally extended from abelian groups, $G_1$ is twice centrally extended from abelian groups, and so on. We show that if $\Phi^g$ preserves injections when mapping into any group that is $i$-centrally extended from abelian groups, then it also does for mapping into $(i+1)$-centrally extended from abelian groups. To see this, we consider the following diagram
    \begin{center}
        \begin{tikzcd}
            1 \arrow{r} & \ker(f \circ \varphi) \arrow{r}\arrow[hook]{d} & H \arrow{r}\arrow{d}{\varphi} & \mathrm{Im}(f \circ \varphi) \arrow[hook]{d} \arrow{r} & 1
            \\
            1 \arrow{r} & A_i \arrow{r}[swap]{g} & G_{i+1} \arrow{r}[swap]{f} & G_i \arrow{r} & 1
        \end{tikzcd}
    \end{center}
    where $A_i$ is finitely generated abelian, $G_i$ is $i$-centrally extended from abelian groups and $G_{i+1}$ is consequently $(i+1)$-centrally extended from abelian groups. The two outermost vertical maps are injective by definition; thus, the group $\ker(f \circ \varphi)$ is finitely generated abelian. By \cref{thm:Hilton-Roitberg}, it follows that we obtain a morphism of exact sequences after applying $\Phi^g$. The outermost maps remain injective, by the induction hypothesis for the rihtmost vertical morphism, and since $\Phi^g$ preserves injections on finitely generated abelian groups for the leftmost morphism. As a consequence, the middle vertical map $\Phi^g  \varphi\colon \Phi^g H \to \Phi^g G_{i+1}$ is injective; consequently, the claim follows by induction.
    \\

    We now turn our attention to the following question. Given finitely generated nilpotent groups $G$, $H$ and a morphism $\varphi\colon G \to H$, how does $\Phi^g\left(\mathrm{Im}(\varphi)\right)$ compare to $\mathrm{Im}(\Phi^g \varphi)$, and $\Phi^g \ker(\varphi)$ to $\ker(\Phi^g \varphi)$? We show that the natural comparison maps are indeed isomorphisms in the setting of finitely generated nilpotent groups. We begin by noting that the image of the continuous group morphism $\Phi^{tg}\varphi\colon \Phi^{tg}G \to \Phi^{tg}H$ is a closed subgroup of the profinite topological group $\Phi^{tg}H$, and its kernel a closed subgroup of $\Phi^{tg}G_1$; thus, both the kernel and image carry a profinite topological group structure, and by the universal property of profinite completions, we obtain the following commutative diagram
    \begin{center}
        \begin{tikzcd}
            1 \arrow{r} &\Phi^g \ker(\varphi) \arrow{r} \arrow{d} & \Phi^g G \arrow[equal]{d}\arrow{r} & \Phi^g \mathrm{Im}(\varphi) \arrow{d} \arrow{r} & 1
            \\
            1 \arrow{r} & \ker \Phi^g \varphi \arrow{r} & \Phi^g G \arrow{r} & \mathrm{Im}(\Phi^g \varphi) \arrow{r} & 1
        \end{tikzcd}
    \end{center}
    We observe that the lower row in the diagram is an exact sequence by definition. As for the upper row, we note that subgroups of finitely generated nilpotent groups are again finitely generated nilpotent, and exactness follows from \cref{thm:Hilton-Roitberg}. As a consequence, it follows that $\Phi^g \mathrm{Im}(\varphi) \to \mathrm{Im}(\Phi^g \varphi)$ is surjective. We argue that it is injective; once done, it also follows that $\Phi^g \ker(\varphi) \cong \ker(\Phi^g \varphi)$. We have a commutative triangle
    \begin{center}
        \begin{tikzcd}
           \Phi^g  \mathrm{Im}(\varphi) \arrow{d} \arrow{r} & \mathrm{Im}(\Phi^g \varphi) \arrow[hook]{ld}
           \\
           \Phi^g H
        \end{tikzcd}
    \end{center}
    Since $\Phi^g$ preserves injections on finitely generated nilpotent groups, it follows that $\Phi^g \mathrm{Im}(\varphi) \to \Phi^g H$ is again injective, and thus it follows that $\Phi^g\mathrm{Im}(\varphi) \to \mathrm{Im}(\Phi^g \varphi)$ is injective, hence an isomorphism.
    \\

    Consider now a long exact sequence
    $$\cdots \to G_2 \xrightarrow{\varphi_1} G_1 \xrightarrow{\varphi_0} G_0$$
    of finitely generated, nilpotent groups. To show that the sequence obtained by applying $\Phi^g$ to the above is exact, we use the above observation that $\Phi^g\mathrm{Im}(\varphi_n) \cong \mathrm{Im}(\Phi^g \varphi_n)$, and $\Phi^g \ker(\varphi) \cong \ker(\Phi^g \varphi)$, for all $n \geq 0$. Thus, we may extend the above long exact sequence to the right, as follows
    $$\cdots G_2 \xrightarrow[]{\varphi_1} G_1 \xrightarrow[]{\varphi_0} \mathrm{Im}(\varphi_0) \to 1 \to 1 \cdots$$
    and we would like to show that applying $\Phi^g$ yields a long exact sequence. The claim follows from the following, more general statement. Let 
    $$\cdots \to H_{-2} \xrightarrow[]{\varphi_{-1}} H_{-1} \xrightarrow[]{\varphi_0} H_0 \xrightarrow[]{\varphi_1} H_1 \xrightarrow[]{\varphi_2} H_2 \to \cdots$$
    be a long exact sequence of finitely generated, nilpotent groups, that is infinite in both directions. The claim is that \cref{thm:Hilton-Roitberg} implies that such two-sided long exact sequences are sent to long exact sequences. To see this, consider the node $H_{k-1} \xrightarrow[]{\varphi_k} H_k \xrightarrow[]{\varphi_{k+1}} H_{k+1}$. We consider the diagram
    \begin{center}
        \begin{tikzcd}
            \ker \varphi_k \arrow{dr} & & & & \mathrm{Im}(\varphi_{k+1}) \arrow{dr}
            \\
            & H_{k-1} \arrow{rr}{\varphi_k} \arrow{dr} & & H_k \arrow{ur} \arrow{rr}{\varphi_{k+1}} & & H_{k+1} \arrow{dr}
            \\
            & & \mathrm{Im}\varphi_k \arrow{ur} & & & & \faktor{H_{k+1}}{\mathrm{Im}(\varphi_{k+1})}
        \end{tikzcd}
    \end{center}
    where all the diagonals are short exact sequences. We observe that the quotient $H_{k+1}/\mathrm{Im}(\varphi_{k+1})$ is again a group, since $\mathrm{Im}(\varphi_{k+1})$ is a normal subgroup of $H_{k+1}$, as being equal to the kernel of the successive map; this was the key reason to extend the exact sequence to the right. Furthermore, all groups involved in the above diagram are finitely generated, nilpotent. Applying $\Phi^g$ to the above diagram, and using \cref{thm:Hilton-Roitberg}, it follows that $\Phi^g$ preserves two-sided long exact sequences, and the claim follows.
\end{proof}

\bigskip

\subsection{Profinite completion of spaces}
\subsubsection{Definitions and main properties}
In \cite{Sullivan-genetics}, Sullivan considers an analogous situation of profinite completion as above, but this time in the setting of spaces. The idea is to approximate a space $X$ by certain spaces with a finiteness assumption, namely

\begin{definition}[$\pi$-finite spaces]
    A space $X$ is said to be $\pi$-finite if
    \begin{itemize}
        \item it has a finite set of path components;
        \item there exists $k \in \mathbb{N}$, such that for all $n\geq k$ and all $x \in X$, $\pi_n (X,x)$ is trivial;
        \item all homotopy groups are finite groups.
    \end{itemize}
\end{definition}
Roughly speaking, Sullivan considers the profinite completion of $X$ as the limit of $\pi$-finite spaces $F$ equipped with all possible maps from $X$; we briefly recall this construction, following \cite[Appendix E]{SAG}. The following is by no means a complete exposition, and we refer to Lurie's Appendix E for more details. Let $\mathcal{S}_\pi \subset \mathcal{S}$ be the full subcategory of spaces consisting of $\pi$-finite spaces, and consider the category $\mathrm{Pro}(\mathcal{S}_\pi)$ of pro-objects in $\mathcal{S}_{\pi}$, i.e. the full subcategory of $\Fun(\mathcal{S}_\pi,\mathcal{S})^{\op}$ of functors preserving finite limits; objects in that category are called \emph{profinite spaces}, and we emphasize that such an object is \emph{not} a space, but merely a formal limit of spaces. Given a space $X \in \mathcal{S}$, we may associate to it the functor $\mathcal{S}_{\pi} \to \mathcal{S}$, $F \mapsto \Map(X,F)$. This association defines a functor, which we call the \emph{pro-$\pi$-finite completion}, and denote by $\widehat{(-)}\colon \mathcal{S} \to \Pro(\mathcal{S}_\pi)$. For a space $X$, the pro-object $\widehat{X}$ can be identified with the diagram $\{X_\alpha\}$ indexed by maps $X \to X_\alpha$, where $X_\alpha$ is a $\pi$-finite space. The \emph{materialisation} functor is then the functor $\mathrm{Mat}\colon \Pro(\mathcal{S}_\pi) \to \mathcal{S}$, sending a pro-object $\{X_\alpha\}$ to its limit in spaces. The above defines an adjoint pair $\widehat{(-)} \dashv \mathrm{Mat}$. The composite 
$$\mathcal{S} \xrightarrow[]{\widehat{(-)}} \Pro(\mathcal{S}_\pi) \xrightarrow[]{\mathrm{Mat}}\mathcal{S}$$
is denoted by $\Phi^s$, and called \emph{finite completion}, following Sullivan. The unit of the adjunction gives rise a to natural transformation $\mathrm{id} \Rightarrow \Phi^s$; for a space $X$, we denote by $\eta_X$ the associated map $\eta_X \colon X \to \Phi^s X$.
\\

The first relation between profinite completion of spaces and that of groups is given as follows. Following Lurie, given a profinite space $X$ and a point $x \in X$, we define $\pi_n(X,x) \defeq \pi_n (\mathrm{Mat}(X),x)$. By \cite[Corollary E.5.2.4, Remark E.5.2.5]{SAG}, the observation is then that $\pi_n (X,x)$ carries a canonical structure of a profinite group. Thus, if we start with a space $X$ and choose a basepoint, let $\eta_X\colon X \to \Phi^s X$ be the canonical map given by the counit of the adjunction, $\mathrm{id} \Rightarrow \Phi^s$. Then, $\pi_n (\Phi^sX,\eta_X(x))$ carries a canonical structure of profinite group; in particular, via the universal property of group profinite completion, we obtain a commutative diagram
\begin{center}
    \begin{tikzcd}
        \pi_n (X,x) \arrow{r} \arrow{dr}  &   \pi_n (\Phi^s X,\eta_X (x))
        \\
        & \Phi^g \pi_n (X,x) \arrow[dashed]{u}[swap]{\exists!}
    \end{tikzcd}
\end{center}
It is natural to then ask: for which spaces $X$ is the comparison map $\Phi^g \pi_n (X,x) \to \pi_n (\Phi^s X,\eta_X (x))$ an isomorphism? This is the content of \cite[thm 3.1]{Sullivan-genetics}. 
\begin{definition}
    Let $\mathcal{C}$ be a subclass of the class of groups. A connected space $X$ is said to be $\pi_1$-$\mathcal{C}$ of finite type, if, for all choices of basepoints $x \in X$, $\pi_1 (X,x)$ is a $\mathcal{C}$-group, and $\pi_n (X,x)$ is a finitely generated group for all $n \in \mathbb{N}$. An arbitrary space $X$ is said to be componentwise $\pi_1$-$\mathcal{C}$ of finite type, if each of its connected components is $\pi_1$-$\mathcal{C}$ of finite type.
\end{definition}
Let $\mathcal{S}^{\pi_1\text{-gd,ft}}$ be the full subcategory of spaces on those objects which are componentwise $\pi_1$-$\mathcal{C}_{gd}$ of finite type, where $\mathcal{C}_{gd}$ is the class of "good" groups in the sense of Sullivan, namely all groups commensurable to a solvable group, with all subgroups finitely generated. We emphasize that we require no further condition on the action of $\pi_1$ on the higher homotopy groups. 
\begin{theorem}[Sullivan]\label{thm:Sullivan3.1}
    Let $X \in \mathcal{S}^{\pi_1\text{-gd,ft}}$ and fix $x \in X$. Then, the comparison map $\Phi^g \pi_n (X,x) \to \pi_n (\Phi^s X,\eta_X(x))$ is an isomorphism, for all $n \geq 1$.
\end{theorem}
In other words, the map $\eta_X\colon X \to \Phi^s X$ exhibits the underlying group of the profinite completion of the homotopy groups of $X$ as the homotopy groups of the space $\Phi^s X$, whenever $X$ is a reasonable enough space. This class of spaces includes, among others, componentwise nilpotent spaces of finite type; these will play a crucial role in what follows.

\begin{remark}\label{remar:Sullivan3.1Morphisms}
    Given spaces $X,Y \in \mathcal{S}^{\pi_1\text{gd,ft}}$, and a map $f \colon X \to Y$, the natural transformation $\mathrm{id} \Rightarrow \Phi^s$ yields a commutative square
    \begin{center}
    \begin{tikzcd}
        X \arrow{r}{\eta_X} \arrow{d}[swap]{f} & \Phi^s X \arrow{d}{\Phi^s f}
        \\
        Y \arrow{r}[swap]{\eta_Y} & \Phi^s Y
    \end{tikzcd}
    \end{center}
    For all basepoints $x \in X$ and $y \defeq f(x) \in Y$, \cref{thm:Sullivan3.1} implies that for all $n \geq 1$, the square of groups obtained by applying $\pi_n$ is isomorphic, in the category of squares of groups, to the square
    \begin{center}
        \begin{tikzcd}
            \pi_n(X,x) \arrow{d}[swap]{\pi_n(f)} \arrow{r} & \Phi^g \pi_n(X,x) \cong \pi_n(\Phi^s X,\eta_X(x)) \arrow{d}{\pi_n(\Phi^s f)}
            \\
            \pi_n(Y,y) \arrow{r} & \Phi^g \pi_n(Y,y) \cong \pi_n(\Phi^s Y, \eta_Y(y))
        \end{tikzcd}
    \end{center}
    A careful reading of Sullivan's proof shows that the group morphism $\pi_n(\Phi^s f) \colon \pi_n(\Phi^s X,\eta_X(x)) \to \pi_n(\Phi^s Y,\eta_Y(y))$ agrees with the morphism $\Phi^g (\pi_n f) \colon \Phi^g \pi_n(X,x) \to \Phi^g\pi_n(Y,y)$. This could also be seen as a corollary of \cite[thm 1.1]{Nikolov-Segal}.
\end{remark}

\subsubsection{Finite completion and finite limits}
In the following, we study how close the functor $\Phi^s$ comes to preserving finite limits. Observe that $\Phi^s$ depends on the functor $\Map(-,X)$, so that preservation of any kinds of limits will not follow from a formal argument. Recall that preserving finite limits is equivalent to preserving the terminal object and pullbacks; we study how $\Phi^s$ behaves with pullbacks. Let $X \rightarrow Z \leftarrow Y$ be a cospan in $\mathcal{S}^{\text{nil,ft}}$, the category of finite type nilpotent spaces, and let $P$ be its pullback. We can apply $\Phi^s$ to the cospan and obtain $\Phi^s X \rightarrow \Phi^s Z \leftarrow \Phi^s Y$, and let $P'$ denote its pullback. We have a commutative diagram
\begin{center}
    \begin{tikzcd}
        & P \arrow{dr}\arrow{dl}
        \\
        \Phi^s P \arrow{rr}  &&  P'
    \end{tikzcd}
\end{center}
The functor preserving finite limits is then equivalent to the map $\Phi^s P \to P'$ being an equivalence, for all such $P$ and $P'$. We show that the functor comes close to preserving finite limits, when restricted to $\mathcal{S}^{\text{nil,ft}}$. The aim of this section is to show the following

\begin{theorem}\label{thm:finite-limit-finite-completion}
    Consider a cospan $X \rightarrow Z \leftarrow Y$ of componentwise nilpotent, finite type spaces, and let $P$ denote the pullback. Let $P'$ denote the pullback of the cospan $\Phi^s X \rightarrow \Phi^s Z \leftarrow \Phi^s Y$, so that we have a commutative triangle
    \begin{center}
    \begin{tikzcd}
        & P \arrow{dr}\arrow{dl}
        \\
        \Phi^s P \arrow{rr}  &&  P'
    \end{tikzcd}
\end{center}
Then, for all basepoints coming from $P$ via the above diagram, the map $\Phi^s P \to P'$ induces an isomorphism on $\pi_n$, $n\geq 1$.
\end{theorem}

We begin by showing a special case of the above, for fibre sequences.

\begin{lemma}\label{lemma:fibration-profinite-completion}
    Let $F \to E \to B$ be a fibre sequence, where $E$ and $B$ are connected, nilpotent spaces of finite type. Denote by $F'$ the fibre of the map $\Phi^s E \to \Phi^s B$. Then, the natural map $\Phi^s F \to F'$ induces an isomorphism on $\pi_n$ for all $n \geq 1$, with basepoints coming from $F$.
\end{lemma}
\begin{proof}
    We begin by observing that $F$ is itself a componentwise nilpotent space of finite type. We have a commutative diagram
    \begin{center}
        \begin{tikzcd}
            \Phi^s F \arrow{r} \arrow{d} &  \Phi^s E \arrow[equal]{d} \arrow{r} &  \Phi^s B \arrow[equal]{d}
            \\
            F' \arrow{r} & \Phi^s E \arrow{r} & \Phi^s B 
        \end{tikzcd}
    \end{center}
    Since $F$ is a nilpotent space of finite type, it follows that $\pi_n (\Phi^s F, \eta_F(x)) \cong \Phi^g \pi_n(F,x)$, where $x \in F$, and this exhibits the induced map on $\pi_n$ of the map $\eta_F\colon F \to \Phi^s F$ as the (group) profinite completion of $\pi_n(F,x)$; the same holds for the other two spaces, $E$ and $B$. By \cref{lemma:long-exact}, it follows that applying $\Phi^g$ to the long exact sequence of homotopy groups of the fibre sequence $F \to E \to B$ yields again a long exact sequence of homotopy groups; thus, even though the top row of the above diagram is not a fibre sequence, we do have a long exact sequence of homotopy groups based at points coming from $F$, $E$ and $B$ respectively. We thus obtain a map of long exact sequences
    \small\begin{center}
        \begin{tikzcd}
            \cdots \arrow{r} & \pi_{n+1}(\Phi^s B, \eta_B(b)) \arrow{r} \arrow[equal]{d} & \pi_n (\Phi^s F,\eta_F(f)) \arrow{r} \arrow{d} & \pi_n(\Phi^s E,\eta_E(e)) \arrow{r} \arrow[equal]{d} & \pi_n(\Phi^s B, \eta_B(b)) \arrow[equal]{d} \arrow{r} & \cdots
            \\
            \cdots \arrow{r} & \pi_{n+1}(\Phi^s B, \eta_B(b)) \arrow{r} & \pi_n (F',\eta'(f)) \arrow{r} & \pi_n(\Phi^s E,\eta_E(e)) \arrow{r} & \pi_n(\Phi^s B, \eta_B(b)) \arrow{r} & \cdots
        \end{tikzcd}
    \end{center}
    where $\eta'\colon F \to F'$ is the natural map between the pullbacks. The claim now follows from the five lemma.
\end{proof}
\bigskip
With the above lemma at hand, we can prove \cref{thm:finite-limit-finite-completion}.
\begin{proof}[Proof of \cref{thm:finite-limit-finite-completion}]

The strategy is to now deduce the general claim about pullbacks from \cref{lemma:fibration-profinite-completion}. Consider
\begin{center}
    \begin{tikzcd}
        P \arrow{d}\arrow{r} & X \arrow{d}
        \\
        Y \arrow{r} & Z
    \end{tikzcd}
\end{center}
where $X$, $Y$ and $Z$ are connected, nilpotent spaces of finite type. Let $P'$ be the pullback of the associated cospan obtained after applying $\Phi^s$. We aim to show that $\Phi^s P \to P'$ induces isomorphisms on homotopy groups, based at points from $P$, using \cref{lemma:fibration-profinite-completion}. Denote by $\{Z_n \}_{n\in \mathbb{N}}$ the principal refinement of the Postnikov tower of $Z$. In particular, we note that $Z_1 \simeq K(A,1)$, for some finitely generated abelian group. Furthermore, there is a sequence $K_n$ of Eilenberg-Maclane spaces associated to a finitely generated abelian group $A_n$ (we omit declaring the degree of the Eilenberg-Maclane space, for clarity), fitting into pullback squares

\begin{center}
    \begin{tikzcd}
        Z_{n} \arrow{d} \arrow{r} & \ast \arrow{d}
        \\
        Z_{n-1} \arrow{r} & \Omega K_n
    \end{tikzcd}
\end{center}

For each $n \in \mathbb{N}$, we consider the pullback $P_n \coloneqq X \times_{Z_n}Y$ of the cospan 
$$X \rightarrow Z_n \leftarrow Y$$
obtained by composition with the canonical map $Z \to Z_n$, and observe that since the pullback functor $\mathrm{cospan}(\mathcal{S}) \to \mathcal{S}$ preserves limits, it follows that $P \simeq \lim_n P_n$. Furthermore, we observe that we have a pullback square
\begin{center}
    \begin{tikzcd}
        P_{n+1}\arrow{d} \arrow{r} & \ast \arrow{d}
        \\
        P_n \arrow{r} & K_n
    \end{tikzcd}
\end{center}
for each $n\in \mathbb{N}$. In particular, for each $k \in \mathbb{N}$, $\pi_k P$ can be read at a finite stage of the inverse system $\{P_n\}$, namely there is a $N_k \in \mathbb{N}$ such that $P \to P_{N_k}$ induces an isomorphism on all homotopy groups below degree $k$. As $P$ is a componetnwise nilpotent space of finite type, \cref{lemma:Principal-Postnikov-profinite} implies that $\{\Phi^s P_n\}_{n \in \mathbb{N}}$ is again the principal refinement of the Postnikov system of $\Phi^s P$, on the components coming from the image of the map $\eta_P \colon P \to \Phi^s P$, \cref{lemma:Principal-Postnikov-profinite}. Consequently, the same analysis as above applies, and we obtain an inverse system $\{P_n'\}_{n \in\mathbb{N}}$ via
$$P_n' \coloneqq \Phi^{s}X \times_{\Phi^s Z_n} \Phi^s Y$$
and observe that $P' \to \lim_n P_n'$ is again an equivalence. We similarly have pullback squares
\begin{center}
    \begin{tikzcd}
        P_{n+1}' \arrow{r} \arrow{d} & \ast \arrow{d}
        \\
        P_n' \arrow{r} & \Phi^s K_n
    \end{tikzcd}
\end{center}
for all $n \in \mathbb{N}$, where we recall that $\Phi^s K_n$ is equivalent to the Eilnberg-Maclane space obtained by applying $\Phi^g$ to the non-trivial homotopy group of $K_n$.
\\

Fix $k \in \mathbb{N}$, and let $N_k \in \mathbb{N}$ be the integer after which both the inverse systems $\{P_n\}$ and $\{P_n'\}$ induce isomorphisms on $\pi_k$. We now consider the following commutative diagram
\begin{center}
    \begin{tikzcd}
        \Phi^s P \arrow{r} \arrow{ddd} & \Phi^s (X \times_{Z_{N_k}}Y) \simeq \Phi^s(\ast \times_{K_{N_k}}(X \times_{Z_{N_{k}-1}}Y)) \arrow{d}
        \\
          & \ast \times_{\Phi^s K_{N_k}}(\Phi^s(\ast \times_{K_{N_k-1}}(X \times_{Z_{N_{k-2}}}Y)) \arrow{d}
          \\
          &  \vdots \arrow{d}
          \\
          P' \arrow{r} &  \ast  \times_{\Phi^s K_{N_k}}(\ast  \times_{\Phi^s K_{N_k-1}}( \ast \cdots \ast  \times_{\Phi^s K_{1}}(\Phi^s X \times \Phi^s Y)))
    \end{tikzcd}
\end{center}
Without further reference, all homotopy groups under consideration are based at points coming from the image of $P$ via the canonical maps. By assumption, the lower horizontal map induces isomorphisms on all $\pi_n$, for $n\leq k$. Using \cref{lemma:fibration-profinite-completion}, the same holds for the top horizontal map, and that furthermore, all the rightmost vertical maps are isomorphisms on all homotopy groups. Thus, the left vertical map induces an isomorphism on $\pi_n$, for all $n \leq k$. This being true for all $k$, the claim follows.
\end{proof}

From \cref{thm:finite-limit-finite-completion}, the claim about finite limit follows. Recall that a finite limit is a limit over a category which is the nerve of a finite simplicial set, i.e. consisting of finitely many degeneracies. Given a pushout for the indexing category
\begin{center}
    \begin{tikzcd}
        S \arrow{r} \arrow{d} & D \arrow{d}
        \\
        X \arrow{r} & Y
    \end{tikzcd}
\end{center}
the limit of the functor over this pushout is precisely the pullback of the evaluations. By finite induction, we obtain the following

\begin{corollary}\label{cor:finite-limits-finite-completion}
    Let $F\colon \mathcal{D} \to \mathcal{S}^{\text{nil,ft}}$ be a functor from a finite diagram of componentwise nilpotent spaces of finite type. Then, the canonical map
    $$\Phi^s(\lim_\mathcal{D}F) \to \lim_\mathcal{D} \Phi^s F$$
    induces an isomorphism on $\pi_n$ for $n \geq 1$, based at points coming from $\lim_{\mathcal{D}}F$ (via the canonical maps), where all the limits are taken in $\mathcal{S}$.
\end{corollary}

\bigskip

Rather informally speaking, the above is a categorical reformulation of the following classical argument, that is used by Sullivan in \cite[p. 28]{Sullivan-genetics}. Let $X$ be a finite CW-complex, and let $Y$ be a nilpotent, finite type space. We would like to show that $\pi_n \Map(X,\Phi^s Y)$, based at the a map $\eta_Y \circ f \colon X \to \Phi^s Y$, obtained as the composition
$$\eta_Y \circ f \colon X \xrightarrow[]{f} Y \xrightarrow[]{\eta_Y} \Phi^s Y$$
for some map $f \colon X \to Y$, can be canonically identified with $\pi_n \Phi^s \Map(X,Y)$, based at $\eta_{\Map(X,Y)}(f)$. This is shown by induction on the finite cells of $X$, as follows. We consider the situation where $X= X' \cup e_n$, for some $n$-cell. Then, we have a fibration
$$F \to \Map(X,Y) \xrightarrow[]{\mathrm{res}}\Map(X',Y)$$
where the fibre $F \coloneqq \fib_{f|X'}(\mathrm{res})$ can be identified with $\Omega^n Y$, after choice of a nullhomotopy of $f|\partial e_n$. Thus, we obtain an exact sequence
$$\cdots \to\pi_2 Y^{X'} \to \pi_{n+1} Y \to \pi_1 Y^X \to \pi_1 Y^{X'} \to \pi_n Y$$
where the basepoints of the mapping spaces are taken to be $f$ and its restriction to $X'$, correspondingly. The induction hypothesis is then that $\pi_1 Y^{X'}$ is a finitely generated, nilpotent group, and that its higher homotopy groups are finitely generated, abelian groups; that is, $Y^{X'}$ is a nilpotent space of finite type. It can further be shown that the extension above on $\pi_1$ is central; thus, it follows that $\pi_1 Y^X$ is again a finitely generated nilpotent group (as being centrally extended from a finitely generated nilpotent group by a finitely generated abelian group), and that all the higher homotopy groups are finitely generated. Exactness of the group profinite completion functor on that class of groups yields the identification $\pi_n \Map(X,\Phi^s Y) \cong \Phi^g \pi_n\Map(X,Y)$ on the basepoints specified above.

\subsubsection{Diagrams of finite completions}
The following theorem is crucial for the proof of residual finiteness of the automorphism groups under consideration in this work.
\begin{theorem}\label{thm:profinite}
    Let $\mathcal{C}$ be an arbitrary $\infty$-category, and consider a space valued functor $\mathcal{F}\colon \mathcal{C} \to \mathcal{S}_{>1}^{\mathrm{fin}}$ landing in the full subcategory of 1-connected finite spaces (i.e. simply connected finite CW complexes), and denote by $\widehat{\mathcal{F}}$ the composition of $\mathcal{F}$ with the profinite completion functor. Then, the group
    $$\pi_0 \Map^{\simeq}_{\Fun(\mathcal{C},\Pro(\mathcal{S}_\pi))}(\widehat{\mathcal{F}}, \widehat{\mathcal{F}})$$
    can be promoted to a profinite topological group.
\end{theorem}

The proof of the above theorem is of categorical nature, and consequently we delay it to \cref{appendix}. Let us first highlight how the above theorem is crucial for the later sections. The following corrollary is an immediate application of the universal property of profinite completion of groups (\cref{thm:group-finite-completion-universal-ppty}).
\begin{corollary}\label{cor:profinite}
    Let $\mathcal{F}$ be as in \cref{thm:profinite}. Then, we have the following commutative diagram
    \begin{center}
        \begin{tikzcd}
            \pi_0 \Map^{\simeq}_{\Fun(\mathcal{C},\mathcal{S})}(\mathcal{F},\mathcal{F}) \arrow{r}\arrow{d} \arrow{dr} &  \Phi^g\pi_0 \Map^{\simeq}_{\Fun(\mathcal{C},\mathcal{S})}(\mathcal{F},\mathcal{F})\arrow{d}{\exists!}
            \\
            \pi_0 \Map^{\simeq}_{\Fun(\mathcal{C},\mathcal{S})}(\Phi^s \circ \mathcal{F},\Phi^s \circ \mathcal{F})  &  \pi_0 \Map^{\simeq}_{\mathrm{Fun}(\mathcal{C},\mathrm{Pro}(\mathcal{S}_\pi))}(\widehat{\mathcal{F}}, \widehat{\mathcal{F}} ) \arrow{l}
        \end{tikzcd}
    \end{center}
\end{corollary}

\cref{cor:profinite} serves as a blueprint to showing residual finiteness: in order to show that the top horizontal map is injective (which is equivalent to the top left group being residually finite), it suffices to show that the left vertical map is injective. The latter is amenable to more homotopy theoretic methods. In fact, injectivity of that map in the setting of embedding calculus follows from obstruction theory, and this is the topic for the later sections.
\\

\subsection{Postnikov decomposition}\label{sect:Functorial-Postnikov-square}
The main idea of \emph{Postnikov theory} is to study a space $X$ via a particular sequence of truncations: for a given $n\in\mathbb{N}$, one can define a space $\tau_{\leq n}X$ receiving a map from $X$, satisfying the following two properties
\begin{enumerate}[(i)]
    \item The map $X \to \tau_{\leq n}X$ induces an isomorphism on $\pi_i$, for $0\leq i \leq n$
    \item $\pi_i (\tau_{\leq n}X,x) =0$ for all $i>n$ and $x \in X$.
\end{enumerate}

The collection $\{\tau_{\leq n}X\}_{n \in \N}$ forms an inverse system approximating $X$, in the sense that the induced map $X \to \lim_n \tau_{\leq n}X$ is an equivalence. A key property of the Postnikov decomposition of a space $X$ is that $\tau_{\leq n+1}X$ is built out of $\tau_{\leq n}X$ and some cohomological datum, the \emph{Postnikov $k$-invariant}. We now study how to do this construction functorially in $X$. Classically, this cohomology class is constructed by transgression on the Leray-Serre spectral sequence for the fibration $K(\pi_{n+1}X,n+1) \to \tau_{\leq n+1}X \to \tau_{\leq n}X$. An $\infty$-categorical setup can be found in \cite[§2]{Piotr}, which we follow.

\subsubsection{Non-principal case}
A crucial ingredient is the \emph{Grothendieck construction}: for a space $X$, there is an equivalence of $\infty$-categories
$$\mathcal{S}_{/X} \xrightarrow[]{\simeq}\Fun(X,\mathcal{S})$$
which informally sends a map $f:Y \to X$ to the functor $x \mapsto \fib_{x}(f)$. Furthermore, for $n \geq 1$, let $\mathcal{EM}_n^{\simeq} \subset \mathcal{S}^{\simeq}$ be the subcategory consisting of spaces that have the homotopy type of a $K(G,n)$ where $G$ is a group (abelian when $n\geq 2$). Let $\pi_n (X,x) \colon \tau_{\leq 1}X \to \Grp^{\cong}$ be the classifying map for the $\pi_1$-action on $\pi_n$, where $\Grp^{\cong}$ is the 1-category of groups and group isomorphisms. We then define $K_{n}(X)$ as the pullback in $\mathcal{S}$ (where we identify the $\infty$-category of spaces with the $\infty$-category of $\infty$-groupoids)

\begin{center}
    \begin{tikzcd}
        K_n (X) \arrow{r} \arrow{d}  &  \mathcal{EM}_{n+1}^{\simeq}\arrow{d}{\pi_{n+1}}
        \\
        \tau_{\leq 1}X \arrow{r}[swap]{\pi_{n+1} (X,-)} \arrow[dashed, bend left]{u}{\mathcal{s}_n}  &  \mathcal{Ab}^{\cong} \arrow[dashed, bend left]{u}{B^{n+1}}
    \end{tikzcd}
\end{center}
where $\mathcal{Ab}^{\cong}$ is the full subcategory of $\Grp^{\cong}$ of abelian groups, and where we note that the right vertical map is well-defined, since an Eilenberg-Maclane space of type $K(G,n+1)$ is simply connected. The map admits a section given by the delooping functor $B^{n+1}\colon \mathcal{Ab}^{\cong} \to \mathcal{EM}_{n+1}^{\simeq}$, which thus gives a section $\mathcal{s}_n \colon \tau_{\leq 1}X \to K_n (X)$. The map $\tau_{\leq n}X \to \mathcal{EM}_{n+1}^{\simeq}$, obtained as the Grothendieck construction on the fibration $\tau_{\leq n+1}X \to \tau_{\leq n}X$, and the truncation map $\tau_{\leq n}X \to \tau_{\leq 1}X$ commute after passing to $\mathcal{Ab}^{\cong}$, and thus give rise to a map $\kappa_n : \tau_{\leq n}X \to K_{n}(X)$. The key observation is then the following

\begin{lemma}
    For $X \in \mathcal{S}$ and $n\geq1$, we have a pullback square
    \begin{center}
        \begin{tikzcd}
            \tau_{\leq n+1}X \arrow{d}[swap]{p_{n+1}} \arrow{r}  &  \tau_{\leq 1}X \arrow{d}{\mathcal{s}_n}
            \\
            \tau_{\leq n}X \arrow{r}[swap]{\kappa_n} & K_n (X)
        \end{tikzcd}
    \end{center}
\end{lemma}
\begin{proof}
    We first observe that the fibre of the functor $\pi_{n+1}\colon\mathcal{EM}_{n+1}^{\simeq} \to \mathcal{Ab}^{\cong}$ at $\pi_{n+1}(X,x) \in\mathcal{Ab}^{\cong}$ is $BK(\pi_{n+1}(X,x),n+1)\simeq K(\pi_{n+1}X,n+2)$. Now, as $\mathcal{s}_n \colon \tau_{\leq 1}X \to K_n (X)$ is a section to the map $K_n (X) \to \tau_{\leq 1}X$, the homotopy type of the fibre at any basepoint (since $K_n (X)$ is connected) is therefore $\Omega K(\pi_{n+1}(X,x),n+2) \simeq K(\pi_{n+1}(X,x),n+1)$. We thus obtain the commutative diagram
    \begin{center}
        \begin{tikzcd}
            K(\pi_{n+1}(X,x),n+1) \arrow{d} &  K(\pi_{n+1}(X,x),n+1)\arrow{d}
            \\
            \tau_{\leq n+1}X \arrow{r} \arrow{d} & \tau_{\leq 1}X \arrow{d}{\mathcal{s}_n}
            \\
            \tau_{\leq n}X \arrow{r}[swap]{\kappa_n}  &  K_n (X)
        \end{tikzcd}
    \end{center}
    which induces an equivalence on the fibres, as the induced map is an isomorphism on $\pi_{n+1}$ by construction.
\end{proof}

It follows readily that the above construction is functorial in equivalences, and we obtain a functor $\mathcal{S}^\simeq \to \Fun(\Delta^1 \times \Delta^1,\mathcal{S})$, i.e. a functor into the category of commutative squares.

\subsubsection{Principal case}
The passage from $\tau_{\leq n}X$ to $\tau_{\leq n+1}X$ involved a cohomology class
$$[\kappa_n] \in H^{n+2}(\tau_{\leq n}X; \underline{\pi_{n+1}X})$$
where the local coefficient system is induced from the action of $\pi_1 X$ on $\pi_n X$, and such a group is usually unwieldy for computations. However, when the action of $\pi_1$ is trivial, it becomes significantly easier to grasp these cohomology groups. We first begin in the case of a 1-connected space $X$; in this setting, one has $\tau_{\leq 1}X \simeq \ast$, and thus the map $\tau_{\leq n+1}X \to \tau_{\leq n}X$ is a principal fibration fitting into the following pullback diagram
\begin{center}
    \begin{tikzcd} 
        \tau_{\leq n+1}X \arrow{d} \arrow{r}  & \ast \arrow{d}
        \\
        \tau_{\leq n}X \arrow{r}[swap]{\kappa_n} & K(\pi_{n+1}X,n+2)
    \end{tikzcd}
\end{center}
Similarly, if $\pi_1 (X,x)$ is non-trivial but acts trivially on $\pi_n (X,x)$ for all $x \in X$, one obtains a decomposition $K_n (X) \simeq K(\pi_{n+1}(X,x),n+2)\times \tau_{\leq 1}X$ and a trivialization of the fibration $\tau_{\leq 1}X \to K_n (X)$, and we recover the same pullback diagram for the simply connected case by gluing the two pullback diagrams
\begin{center}
    \begin{tikzcd}
        \tau_{\leq n+1}X  \arrow{r}\arrow{d}  &  \tau_{\leq 1}X  \arrow{d}{\ast \times \mathrm{id}} \arrow{r} & \ast \arrow{d}
        \\
        \tau_{\leq n}X \arrow{r}[swap]{\kappa_n}  & K(\pi_{n+1}X,n+2)\times \tau_{\leq 1}X  \arrow{r}[swap]{\mathrm{proj}_1} &   K(\pi_{n+1}X,n+2)
    \end{tikzcd}
\end{center}

This defines a functor from the category of simple spaces to the category of towers consisting of principal fibrations (we note that the $k$-invariants form additional data).

\subsubsection{Nilpotent case}
Among spaces whose Postnikov tower is not principal, one particularly stands out for being close enough: \emph{nilpotent spaces}. We give a brief reminder of this notion. Fix a group $\pi$ and a $\mathbb{Z}\pi$-module $A$, i.e. an abelian group endowed with an action of $\pi$.

\begin{definition}[Lower central series of action]
    The lower central series of a $\mathbb{Z}\pi$-module $A$ is a sequence of submodules
    $$\cdots \subset \Gamma_{\pi}^{r}(A) \subset \Gamma_{\pi}^{r-1}(A) \subset \cdots \subset \Gamma_{\pi}^{2}(A) \subset A$$
    where $\Gamma_{\pi}^{2}(A)$ is generated by elements of the form $g\cdot x-x$, with $g \in \pi$ and $x \in A$, and where $\Gamma_{\pi}^{r}(A)$ is inductively defined by
    $$\Gamma_{\pi}^{r}(A) \defeq \Gamma_{\pi}^{2}(\Gamma_{\pi}^{r-1}(A))$$
\end{definition}

The action of $\pi$ on $A$ is \emph{nilpotent} if its lower central series is eventually trivial, i.e. if there exists some $r \in \mathbb{N}$ so that $\Gamma_{\pi}^r (A)$ is the trivial module.

\begin{definition}[Nilpotent space]
    A space $X$ is nilpotent if for all $x \in X$, the group $\pi_1(X,x)$ is a nilpotent group, and if the action of $\pi_1 (X,x)$ on $\pi_n (X,x)$ is nilpotent for all $n \geq 2$. 
\end{definition}
Let $X$ be a nilpotent space. Then, for all $n \in \mathbb{N}$, the maps 
$$\tau_{\leq n}X \to \tau_{\leq n-1}X$$
occuring in the Postnikov tower can be rewritten as a composition of finitely many maps, each of which is a principal fibration. In fact, the converse is true: a space $X$ is nilpotent if and only if each map $\tau_{\leq n}X \to \tau_{\leq n-1} X$ can be decomposed as above. We call the tower thus obtained a \emph{principal refinement} of the Postnikov tower.
\begin{remark}
    The above construction can be made functorial in the $\infty$-category of spaces; however, this will not be needed for our purposes, and we will be treating the principal refinement following classical homotopy theory.
\end{remark}

We now record a lemma highlighting the behaviour of (refined) Postnikov towers under finite completions, for nilpotent spaces of finite type; such a technical result will prove useful in the proofs of the following sections.

\begin{lemma}\label{lemma:Principal-Postnikov-profinite}
    Let $X$ be a connected, nilpotent space of finite type, and let $\{X_n\}_{n \in \mathbb{N}}$ be a refinement of its Postnikov tower. Then, $\{\Phi^s X_n\}$ is a refined Postnikov tower for $\Phi^s X$, where the fibre of the map $X_{n+1} \to X_n$ is an Eilenberg-Maclane space, obtained from the fibre of $X_{n+1} \to X_n$ by group profinite completion.
\end{lemma}
\begin{proof}
    This follows immediately from \cref{lemma:fibration-profinite-completion}.
\end{proof}

\bigskip

\subsubsection{Moore-Postnikov decomposition}\label{Moore-Postnikov-functor}
Given connected spaces $X,Y$ and a map $f\colon X \to Y$, the \emph{Moore-Postnikov decomposition of $f$} is a sequence of spaces $\{Z_n\}_{n \in \mathbb{N}}$ along with maps $g_n: X \to Z_n$, $h_n:Z_n \to Y$ and $p_n :Z_{n+1} \to Z_n$ fitting into the following commutative diagram
\begin{center}
    \begin{tikzcd}
     X \arrow[bend right]{dd}[swap]{f} \arrow{d}{g_n} \arrow[bend left]{dr}{g_{n+1}}
     \\
     Z_n \arrow{d}{h_n} & Z_{n+1}\arrow[bend left]{dl}{h_{n+1}} \arrow{l}{p_n}
     \\
     Y
    \end{tikzcd}
\end{center}
uniquely characterized by the following properties
\begin{enumerate}[(i)]
    \item $g_n$ is $n$-connected: it induces an isomorphism on homotopy groups $\pi_k$ for $k<n$, and a surjection for $k=n$;
    \item $h_n$ is $n$-truncated (sometimes also referred to as $n$-co-connected in geometric contexts): it induces an isomorphism on $\pi_k$ for $k> n$, and an injection on $k=n$;
\end{enumerate}

Given a morphism between two maps $f: X \to Y$, $f':X' \to Y'$ in $\mathrm{Ar}(\mathcal{S})$, namely a commutative square
\begin{center}
    \begin{tikzcd}
        X \arrow{d}[swap]{f} \arrow{r} & X' \arrow{d}{f'}
        \\
        Y \arrow{r} & Y'
    \end{tikzcd}
\end{center}
we obtain maps $Z_{n} \to Z_{n}'$ between their Moore-Postnikov truncations, fitting into the above square. This suggests that the Moore-Postnikov decomposition defines a system of functors
$$\mathrm{MP}_n:\mathrm{Ar}(\mathcal{S}) \to \mathcal{S}$$
which to an object in the arrow category of spaces, namely a map $X \to Y$, associates its Moore-Postnikov truncation in degree $n$. To see that this describes a functor, we observe that by \cite[Example 5.2.8.16]{HTT} that (i), (ii) in the above describe a factorisation system.
\subsection{Homotopy mapping class groups}
We are now ready to show \cref{thm:Serre-Sullivan} of Serre and Sullivan. Let $X$ be a simply connected finite CW-complex. Consider the following diagram
\begin{center}
    \begin{tikzcd}
        \pi_0 \hoAut(X) \arrow{r} \arrow{dr} \arrow{d} & \Phi^g \pi_0 \hoAut(X) \arrow[dashed]{d}{\exists !}
        \\
        \pi_0 \hoAut(\Phi^sX) & \pi_0 \Map^{\simeq}_{\mathrm{Pro}(\mathcal{S}_\pi)}(\widehat X, \widehat X) \arrow{l}
    \end{tikzcd}
\end{center}

The existence (and uniqueness) of the rightmost vertical morphism follows from the universal property of profinite completion of groups, and \cref{thm:profinite}. The purpose of the above diagram is that injectivity of $\pi_0 \hoAut(X) \to \Phi^g \pi_0 \hoAut(X)$ follows from injectivity of $\pi_0 \hoAut(X) \to \pi_0 \hoAut(\Phi^sX)$; the latter follows from a more general theorem of Sullivan, namely \cite[thm 3.2]{Sullivan-genetics} (\cref{thm:Sullivan3.2}), stating that for reasonable spaces $X,Y$, the map $\Map(X,Y) \to \Map(X,\Phi^s Y)$, obtained by composition with the map $\eta_Y \colon Y \to \Phi^s Y$, is $\pi_0$-injective; we rewrite the proof in what follows, as we will be using a similar strategy in later proofs. 

\begin{theorem}[Sullivan] \label{thm:Sullivan3.2}
    Let $Y$ be a finite CW complex, and let $B$ be a nilpotent space of finite type. Then, the map $\Map(Y,B) \to \Map(Y,\Phi^sB)$, given by composition with the finite completion $B \to \Phi^sB$, induces an injective map on $\pi_0$.
\end{theorem}

\begin{proof}
    Let $\{B_n\}_{n \in \mathbb{N}}$ be a principal refinement of the Postnikov decomposition of $B$, and observe that $\{ \Phi^sB_n \}_{n \in \mathbb{N}}$ yields a principal Postnikov decomposition of $\Phi^sB$, by \cref{lemma:Principal-Postnikov-profinite}. As $Y$ is a finite CW complex (hence its cohomology is eventually trivial), any map $f \colon Y \to B$ is uniquely determined by the composition $f_{N} \colon Y \to B \to \tau_{\leq N}B$ for some $N \in \mathbb{N}$, in the sense that for all $k \geq N$, one has unique lifts up to homotopy:
    \begin{center}
        \begin{tikzcd}
            & \tau_{\leq k+1}B \arrow{d}
            \\
            Y \arrow{r}{f_k} \arrow{ur}{\exists !} & \tau_{\leq k}B 
        \end{tikzcd}
    \end{center}
    Consequently, given $f \not \simeq g \colon Y \to B$, we may consider $f$ and $g$ as maps to $\tau_{\leq N}B$ for some large $N$; the proof of injectivity then follows by working inductively on the Postnikov tower. Consider the following situation:

    \begin{center}
        \begin{tikzcd}
              &  \Omega K \arrow{d}
            \\
            & B_{n+1} \arrow{d} \arrow{r} & \ast \arrow{d}
            \\
            Y \arrow{r}[swap]{\varphi} \arrow[bend left]{ur}[swap]{g} \arrow[bend right]{ur}{f} \arrow[dashed, bend left]{uur}{\Delta(f,g)}&  B_n \arrow{r}[swap]{\kappa_n}  & K_n\defeq K(\pi_{n+1},n+2)
        \end{tikzcd}
    \end{center}
    We first describe the above diagram. The right hand square is a pullback square. We fix a map $\varphi:Y \to B_n$, and take two lifts $f,g$ of $\varphi$ over $B_{n+1}$. Each of the previous two lifts yields a nullhomotopy of the composition $\kappa_n \circ \varphi$, which is equivalent to the data of two maps $N(f),N(g)\colon CY \to K$ from the cone of $Y$. Gluing the above two maps from the cone, we obtain a map from the suspension $\Sigma Y \to K$, which by adjunction yields the map $\Delta(f,g) \colon Y \to \Omega K$ (which is rightly thought of as the difference between $f$ and $g$). The cohomology class represented by this map detects when $f$ is homotopic to $g$ as lifts: $f \simeq g$ as lifts if and only if $[\Delta(f,g)]=0$. However, one can also obtain information on when $f$ and $g$ are homotopic, but not necessarily through lifts, from the cohomology class $[\Delta(f,g)]$. Indeed, we obtain the following criterion:
    $$f \simeq g \iff [\Delta(f,g)] \in \operatorname{Im}\left(\pi_0 \Omega_{\varphi}\Map(Y,B_n) \xrightarrow[]{\Omega\kappa_n} \pi_0\Omega_{\kappa_n \circ \varphi}\Map(Y,K)\right)$$
    Additionally, $\pi_0 \Omega_{\varphi}\Map(Y,B_n)$ and $\pi_0\Omega_{\kappa_n \circ \varphi}\Map(Y,K)$ carry group structures by concatenation of loops, and the map $\Omega\kappa_n$ induces a group homomorphism between them. By the Brown representability theorem, we also see that $\pi_0\Omega_{\kappa_n \circ \varphi}\Map(Y,K) \cong H^{n+1}(Y,\pi_{n+1})$.
    \\

    We can repeat the same diagram after finite completion $\Phi^s B$ on the target, and composition with the natural map $\eta_B \colon B \to \Phi^s B$. We obtain the same diagram above, where $B_i$ is replaced by $\Phi^s B_i$, and $K$ by $\Phi^s K$, and where \cref{lemma:Principal-Postnikov-profinite} provides the required pullback square. Our aim is now to show that if $f \not\simeq g$, then $\eta_B \circ f \not\simeq \eta_B \circ g$.
    \\

    Consider the following diagram

    \begin{equation}\label{diagram}
        \begin{tikzcd}
            G\defeq\pi_0 \Omega_{\varphi}\Map(Y,B_n) \arrow{d}[swap]{\Omega\kappa_n} \arrow{r} &  G'\defeq \pi_0 \Omega_{\eta_B \circ \varphi}\Map(Y,\Phi^s B_n) \arrow{d}{\Omega\eta_B \circ\kappa_n}
            \\
            H^{n+1}(Y,\pi_{n+1}) \arrow{r} \arrow{d} &  H^{n+1}(Y,\Phi^g\pi_{n+1})  \arrow{d}
            \\
            \coker\Omega\kappa_n \arrow{r} &  \coker\Omega\eta_B \circ \kappa_n
        \end{tikzcd}
    \end{equation}
    Note that $f \not\simeq g$ is equivalent to $\overline{[\Delta(f,g)]} \neq 0$ in $\coker \Omega\kappa_n$ (and similarly for $\eta_B \circ f, \eta_B \circ g$), and that $\Delta(f,g)$ maps to $\Delta(\eta_B \circ f,\eta_B \circ g)$ via the middle horizontal morphism of the diagram. Consequently, our aim reduces to showing that the group homomorphism $\coker\Omega\kappa_n \to \coker\Omega\eta_B\kappa_n$ is injective. To this end, we consider the following
    \begin{itemize}
        \item \textit{Claim}: the upper square in diagram \ref{diagram} is isomorphic, in the category $\mathrm{Sq}(\Grp)$ of squares in groups, to the square
        \begin{center}
            \begin{tikzcd}
                G \arrow{r}{\eta_G} \arrow{d}[swap]{\Omega \kappa_n} & \Phi^g G \arrow{d}{\Phi^g (\Omega \kappa_n)}
                \\
                H^{n+1}(Y,\pi_{n+1}) \arrow{r}[swap]{\eta_{H^{n+1}}} & \Phi^g H^{n+1}(Y,\pi_{n+1})
            \end{tikzcd}
        \end{center}
    \end{itemize}
    We delay the proof of the claim to the next paragraph, and show how this already implies that the map $\coker \Omega\kappa_n \to \coker \Omega\Phi^s\kappa_n$ is injective, as desired. We begin by observing that we have an isomorphism in the arrow category in groups $\mathrm{Ar}(\Grp)$ between the arrow $(\coker \Omega \kappa_n \to \coker \Omega \eta_B \circ \kappa_n)$, and the arrow $(\coker \Omega \kappa_n \to \coker \Phi^g(\Omega \kappa_n))$. The profinite completion functor (remembering the topology) $\Phi^{tg}\colon \Grp \to \mathrm{Prof}\Grp$ is left adjoint to the forgetful functor $\mathrm{Prof}\Grp \to \Grp$, and is therefore right exact. Furthermore, as profinite groups are compact, Hausdorff totally disconnected topological groups, the forgetful functor $\mathrm{Prof}\Grp \to \Grp$ is exact, so that the functor $\Phi^g\colon \Grp \to \Grp$ remains right exact. Consequently, the latter arrow is itself isomorphic to the arrow $\coker\Omega\kappa_n \xrightarrow[]{\eta} \Phi^g \coker \Omega \kappa_n$, where $\eta$ is the finite completion morphism. Now, as $\coker \Omega \kappa_n$ is a finitely generated abelian group (as being the quotient of the cohomology group, itself a finitely generated abelian group), it is in particular residually finite, whence injectivity of all the three previous arrows. It therefore follows that if $f \not \simeq g$, then $\eta \circ f \not \simeq \eta \circ g$, and the proof follows by finite induction.
    \\

    We now prove the claim. We begin by observing that the upper square diagram \ref{diagram} is obtained by applying $\pi_1$ to the following square
    \begin{center}
        \begin{tikzcd}
            \Map(Y,B_n) \arrow{r}{\eta_B \circ -} \arrow{d}[swap]{\kappa_n \circ -} & \Map(Y, \Phi^s B_n) \arrow{d}{\Phi^s \kappa_n \circ -}
            \\
            \Map(Y,K_n) \arrow{r}[swap]{\eta_K \circ -} & \Map(Y,\Phi^s K_n)
        \end{tikzcd}
    \end{center}
    where the basepoint at the top left corner is chosen to be $\varphi \in \Map(Y,B_n)$ (and all other basepoints are taken to be the images of $\varphi$ by the corresponding maps). Given spaces $X,Z$, we write the mapping space $\Map(X,Z)$ as the following limit
    $$\Map(X,Z) \simeq \lim_{X}Z$$
    and note that in the case $X$ is a finite CW complex, the above is a finite limit. Rephrased in these terms, the above comutative squares spells out to the following
    \begin{center}
        \begin{tikzcd}
            \lim_Y B_n \arrow{r}{\lim \eta_B} \arrow{d}[swap]{\lim \kappa_n} & \lim_Y \Phi^s B_n \arrow{d}{\lim \Phi^s \kappa_n}
            \\
            \lim_Y K_n \arrow{r}[swap]{\lim \eta_K} & \lim_Y \Phi^s K_n
        \end{tikzcd}
    \end{center}
    Coassembly yields a natural transformation $\mathrm{coass} \colon \Phi^s \lim_Y \Rightarrow \lim_Y \Phi^s$, which consequently yields the following commutative diagram
\[\begin{tikzcd}
	&& {\Phi^s \lim_Y B_n} \\
	{\lim_Y B_n} &&&& {\lim_Y \Phi^s B_n} \\
	\\
	&& {\Phi^s \lim_Y K_n} \\
	{\lim_Y K_n} &&&& {\lim_Y \Phi^s K_n}
	\arrow["{\mathrm{coass}}", from=1-3, to=2-5]
	\arrow["{\Phi^s \lim\kappa_n}", dotted, from=1-3, to=4-3]
	\arrow["\eta", from=2-1, to=1-3]
	\arrow["{\lim\eta_B}"{description}, from=2-1, to=2-5]
	\arrow["{\lim \kappa_n }"', from=2-1, to=5-1]
	\arrow["{\lim \Phi^s\kappa_n}", from=2-5, to=5-5]
	\arrow["{\mathrm{coass}}"{description}, from=4-3, to=5-5]
	\arrow["\eta"', from=5-1, to=4-3]
	\arrow["{\lim \eta_K}"', from=5-1, to=5-5]
\end{tikzcd}\]
By \cref{thm:finite-limit-finite-completion,cor:finite-limits-finite-completion}, both the coassembly maps in the above diagram induce isomorphisms on $\pi_n$, for all $n \geq 1$, and in particular on $\pi_1$, at the basepoint corresponding to the image of $\varphi \in \Map(Y,B_n)$; this follows since $B_n$ and $K_n$ are nilpotent spaces of finite types. \cref{thm:Sullivan3.1} identifies the arrow in groups $\left(\pi_1(\lim_Y B_n,\varphi) \to \pi_1 (\Phi^s \lim_Y B_n,\eta \varphi)\right)$ as $\eta \colon \pi_1 (\lim_Y B_n, \varphi) \to \Phi^g \pi_1 (\lim_Y B_n, \varphi)$, the group finite completion map (and the analogous statement for $K_n$), while \cref{remar:Sullivan3.1Morphisms} identifies the induced morphism $\pi_1(\Phi^s \lim \kappa_n)$ as $\Phi^g \pi_1(\lim \kappa_n)$. This consequently shows the claim, and the proof follows by finite induction.
\end{proof}

\begin{remark}
    Although not relevant for our purposes, the above theorem can, without much trouble, be seen as a special case of a more general statement about section spaces of bundles. Indeed, given a bundle $\xi\colon E \to B$ over a finite CW complex $B$, with fibre $F$ a componentwise nilpotent space of finite type, the map on section spaces
    $$\Gamma(\xi) \to \Gamma(\Phi^s_{/B}\xi)$$
    can be shown to be $\pi_0$-injective, where $\Phi^s_{/B}$ is the fibrewise finite completion: via the Grothendieck construction, $\xi$ can be identified with a functor in $\Fun(B,\mathcal{S})$, where $B$ is viewed as an $\infty$-groupoid, and $\Phi^s_{/B}$ is given by composing the above functor with the functor $\Phi^s\colon\mathcal{S} \to \mathcal{S}$. The above theorem is then the special case when $\xi$ is trivialisable.
\end{remark}

\bigskip

\section{$T_k$-mapping class groups}\label{section:T_k-mapping-class-groups}
In this chapter, we generalize Sullivan's result to the setting of mapping spaces in categories of presheaves on discs. Given a smooth $d$-dimensional manifold $M$, and $k \in \mathbb{N} \cup \{\infty\}$, we define 
$$T_k \Diff(M) \defeq \Map_{\Psh(\Disk_{d}^{\leq k})}^{\simeq}(E_M,E_M)$$
where $E_M$ is the disc-presheaf $\sqcup_{\ell} \mathbb{R}^d \mapsto \Emb(\sqcup_{\ell}\mathbb{R}^d,M)$. The goal of this section is to use techniques similar to those of the previous section in order to show that the group $\pi_0 T_k \Diff(M)$ is residually finite, when $k < \infty$, and $M$ is closed and 2-connected.

\subsection{Embedding calculus with tangential structures}
Let $\Manf_d$ be the topologically enriched category of smooth $d$-dimensional manifolds (not necessarily compact) with empty boundary, and smooth embeddings, endowed with the $C^\infty$-topology; via the coherent nerve, we view this category as an $\infty$-category. As a full subcategory, we can consider $\Disk_d \subset \Manf_d$, where objects are $\sqcup_\ell \mathbb{R}^d$, for all $\ell \in \mathbb{N}$; we denote by $\iota_\infty$ the inclusion $\Disk_d \hookrightarrow \Manf_d$. This category has an obvious filtration: for each $k \in \mathbb{N}$ we can further consider the full subcategory $\Disk_d^{\leq k}$ consisting of objects of the form $\sqcup_k \mathbb{R}^d$, where $\ell \leq k$; we denote the inclusions $\Disk_d^{\leq k} \hookrightarrow \Disk_d$ by $\iota_k$. We define a functor $E \colon \Manf_d \to \Psh(\Disk_d)$ as
$$E \colon \Manf_d \xrightarrow[]{h} \Psh(\Manf_d) \xrightarrow[]{\iota_{\infty}^{\ast}} \Psh(\Disk_d) $$
where $h$ is the Yoneda embedding. Concretely, given $M \in \Manf_d$, $E_M$ is the disc-presheaf $\sqcup_\ell \mathbb{R}^d \mapsto \Emb(\sqcup_\ell \mathbb{R}^d,M)$. We can further compose $E$ by the restriction functor $\iota_k^\ast$, and obtain a functor $\Manf_d \to \Psh(\Disk_d^{\leq k})$. We thus obtain a tower of functors
\begin{center}
    \begin{tikzcd}
         \Psh(\Disk_d) \arrow{r} & \cdots \arrow{r} & \Psh(\Disk_d^{\leq k}) \arrow{r} & \cdots \arrow{r} & \Psh(\Disk_d^{\leq 1})
        \\
        &&&& \Manf_d \arrow{u}[swap]{\iota_1^\ast E} \arrow{ull}{\iota_k^\ast E} \arrow[bend left]{ullll}{E}
    \end{tikzcd}
\end{center}
Given a smooth $d$-manifold $M$, and a smooth manifold $N$ of dimension $\geq d$, we define
$$T_k \Emb(M,N) \defeq \Map_{\Psh(\Disk_d^{\leq k})}(\iota_k^\ast E_M, \iota_k^\ast E^d_N)$$
for all $k \in \mathbb{N} \cup \{\infty\}$, where $E_N^d$ is the disc-preseheaf sending $\sqcup_\ell \mathbb{R}^d \mapsto \Emb(\sqcup_\ell \mathbb{R}^d,N)$. On mapping spaces, the above tower of functors gives a tower of spaces, the \emph{embedding calculus tower}, $\{T_k \Emb(M,N)\}_{k \in \mathbb{N}}$, receiving a map from $\Emb(M,N)$. We say the tower \emph{converges} if the map
$$\Emb(M,N) \to T_\infty \Emb(M,N)$$
is an equivalence. The following theorem stands as a fundamental theorem of embedding calculus
\begin{theorem}[Goodwillie-Weiss]
    Let $M$, $N$ be two smooth manifolds of dimension $d$. If $h\dim M \leq d-3$, then $\Emb(M,N) \to T_\infty \Emb(M,N)$ is an equivalence, where $h\dim M$ is the handle dimension of $M$.
\end{theorem}
The advantage of the above theorem is to give more homotopy theoretic methods to study spaces of embeddings. In particular, given a fixed embedding $f \colon M \hookrightarrow N$, we can describe the fibre of the map 
$$T_{k+1} \Emb(M,N) \to T_k \Emb(M,N)$$
as a certain space of sections of a bundle over the unordered configuration space $U\Conf_k(M)$ of $M$ relative a fixed section on the fat diagonal, where the fibres are built out of cubes of configuration spaces of $N$ (see, for instance, the final paragraph of \cite{Weiss}, or \cite[thm 4.9]{Manuel-Sander}). This fibre is denoted $L_k \Emb(M,N)_f$, and called the \emph{layer of the tower}.
\\
Given a smooth manifold $M$ of dimension $d$, and $k \in \mathbb{N} \cup \{\infty\}$, the main object under investigation in this work is
$$T_k \Diff(M) \defeq \Map_{\Psh(\Disk_d^{\leq k})}^{\simeq}(\iota_k^\ast E_M, \iota_k^\ast E_M)$$
We refer to the group of components of the above space the \emph{$T_k$-mapping class group} of $M$.
\\

It is possible to set up embedding calculus to include manifolds with boundary, and boundary conditions. We shall return to this once needed later.
\\

We now study a method of imposing tangential structures on manifolds, and applying embedding calculus, where now all embeddings are required to preserve those tangential structures. Fix a dimension $d \in \mathbb{N}$, and let $\Theta \in \Psh(BO(d))$ be a space with a Borel action of $O(d)$. Denote by $\xi_\Theta \colon B_\Theta \to BO(d)$ the associated Borel construction through the Grothendieck equivalence 
$$\Psh(BO(d)) \simeq \faktor{\mathcal{S}}{BO(d)}$$
We define $\mathrm{Manf}_d^\Theta$ as the following pullback in $\infty$-categories
\begin{center}
    \begin{tikzcd}
        \mathrm{Manf}_d^{\Theta} \arrow{d} \arrow{r} &  \faktor{\mathcal{S}}{B_\Theta} \arrow{d}{\xi_\Theta \circ -}
        \\
        \mathrm{Manf}_d \arrow{r}  &   \faktor{\mathcal{S}}{BO(d)}
    \end{tikzcd}
\end{center}
where the lower horizontal functor sends a manifold $M$ to the map $M \to BO(d)$ classifying the tangent bundle of $M$.

We may further restrict to discs. Given any tangential structure $\Theta$ and $k \in \mathbb{N}$, we define $\Disk_d^{\Theta,\leq k}$ as the following pullback in $\infty$-categories
\begin{center}
    \begin{tikzcd}
        \Disk_d^{\Theta,\leq k} \arrow{d}\arrow{r} & \faktor{\mathcal{S}}{B_\Theta} \arrow{d}{ \xi_\Theta \circ -}
        \\
        \Disk_d^{\leq k} \arrow{r} & \faktor{\mathcal{S}}{BO(d)}
    \end{tikzcd}
\end{center}
By the universal property of presheaf categories \cite[5.1.5.6]{HTT}, we obtain a colimit preserving functor  $\Psh(\Disk_d^{\leq k}) \to \mathcal{S}/BO(d)$, which sits in the following commuting diagram
\begin{center}
    \begin{tikzcd}
        \Disk_d^{\Theta,\leq k} \arrow{r}\arrow{d} & \faktor{\mathcal{S}}{B_\Theta} \arrow{d}
        \\
        \Psh(\Disk_d^{\leq k}) \arrow{r} &  \faktor{\mathcal{S}}{BO(d)}
    \end{tikzcd}
\end{center}
and which in turn yields a functor $\Disk_d^{\Theta,\leq k} \to \mathcal{C}$, where $\mathcal{C}$ is the pullback of the cospan
$$\Psh(\Disk_d^{\leq k}) \rightarrow \faktor{\mathcal{S}}{BO(d)} \leftarrow \faktor{\mathcal{S}}{B_\Theta}$$
Note that $\mathcal{C}$ again admits all colimits. Using once more the universal property of presheaf categories \cite[5.1.5.6]{HTT}, we obtain a colimit preserving functor $\Psh(\Disk_d^{\Theta,\leq k}) \to \mathcal{C}$, which thus yields a commutative diagram
\begin{center}
    \begin{tikzcd}
        \Psh(\Disk_d^{\Theta,\leq k}) \arrow{r}\arrow{d} &  \faktor{\mathcal{S}}{B_\Theta} \arrow{d}
        \\
        \Psh(\Disk_d^{\leq k}) \arrow{r} & \faktor{\mathcal{S}}{BO(d)}
    \end{tikzcd}
\end{center}
The left vertical functor is obtained by left Kan extension
\begin{center}
    \begin{tikzcd}
        \Disk_d^{\Theta,\leq k} \arrow{r}\arrow[hook]{d} &  \Disk_d^{\leq k} \arrow[hook]{r} & \Psh(\Disk_d^{\leq k})
        \\
        \Psh(\Disk_d^{\Theta,\leq k}) \arrow[dashed]{urr}
    \end{tikzcd}
\end{center}
We may further consider the full subcategory of $\Psh(\Disk_d^{\leq k})$ given by elements $E_M$ for $M$ a $d$-dimensional manifold with empty boundary, which we denote by $T_k \Manf_d$. We may also consider the full subcategory $T_k \Manf_d^{\Theta}$ given by pullback
\begin{center}
    \begin{tikzcd}
        T_k \Manf_d^{\Theta} \arrow{r}\arrow{d} & \Psh(\Disk_d^{\Theta,\leq k}) \arrow{d}
        \\
        T_k \Manf_d \arrow[hook]{r} & \Psh(\Disk_d^{\leq k})
    \end{tikzcd}
\end{center}
For a $d$-dimensional smooth manifold $M$ together with a fixed tangential $\Theta$-structure $\ell$ on it, we define
$$T_k \Diff^{\Theta}(M,\ell) \defeq \Map_{\Psh(\Disk_d^{\Theta,\leq k})}^{\simeq}(E_M^{\Theta},E_M^{\Theta})$$

The following result, due to Krannich and Kupers \cite[thm 4.15]{Manuel-Sander}, gives a smoothing theory description on the level of tangential structures.

\begin{theorem}[Krannich-Kupers]\label{thm:embedding-calculus-with-tangential-structure}
    Let $\Theta \in \Psh(BO(d))$. For each $k\in\mathbb{N}\cup\{\infty\}$, every square in the following commutative diagram is a pullback square of $\infty$-categories
    \begin{center}
        \begin{tikzcd}
            \Disk_d^{\Theta,\leq k} \arrow{r}\arrow{d} &  \mathrm{Manf}_d^{\Theta} \arrow{r}\arrow{d}  &   \Psh(\Disk_d^{\Theta,\leq k}) \arrow{r}\arrow{d} & \faktor{\mathcal{S}}{B_\Theta} \arrow{d}
            \\
            \Disk_d^{\leq k} \arrow[hook]{r}  &  \mathrm{Manf}_d \arrow{r}[swap]{\iota_k^\ast E}  &   \Psh(\Disk_d^{\leq k}) \arrow{r} & \faktor{\mathcal{S}}{BO(d)}
        \end{tikzcd}
    \end{center}
\end{theorem}

As a corollary, we obtain the following
\begin{corollary}\label{thm:pullback-moduli-spaces}
    Let $\Theta \in \Psh(BO(d))$ be a tangential structure, and fix a smooth, $d$-dimensional manifold with empty boundary, together with a fixed $\Theta$-stucture $\ell$ on it. The fibre of the map 
    $$BT_k\Diff^\Theta(M,\ell) \to BT_k\Diff(M) $$
    is a collection of components of the space of tangential $\Theta$-structures on $M$, i.e. the space
    $$\Map^{O(d)}(\Fr(M),\Theta)$$
\end{corollary}
\bigskip

\subsection{Spin $T_k$-mapping class groups}  
Consider a closed, smooth $2$-connected $d$-manifold, together with a choice $\mathfrak{s}$ of a spin structure on it, and define $E_{M}^{\Spin}\in\Psh(\Disk_{d}^{\Spin})$ by $\sqcup_{k} \mathbb{R}^d \mapsto \Emb^{\Spin}(\sqcup_k \mathbb{R}^d, M)$. The aim of this section is to show the following.

\begin{theorem}\label{thm:spin-T_k-residual-finite}
    Let $(M,\mathfrak{s})$ be a closed, smooth $2$-connected $d$-manifold, together with a choice of a spin structure on it. Then, for all $k \in \mathbb{N}$, 
    $$\pi_0 T_k \Diff^{\Spin}(M) \defeq \pi_0 \Map_{\Psh(\Disk_d^{\Spin,\leq k})}^{\simeq}(\iota_{k}^{\ast}E_{M}^{\Spin}, \iota_{k}^{\ast}E_{M}^{\Spin})$$
    is a residually finite group.
\end{theorem}

We begin with some notation. Given a space valued diagram $X \in \Fun(\mathcal{C},\mathcal{S})$, we define its \emph{profinite completion} $\widehat X$ to be performed pointwise, namely as the following composition
$$\widehat X\colon\mathcal{C} \xrightarrow[]{X} \mathcal{S} \to \operatorname{Pro}(\mathcal{S}_\pi)$$
We may further compose with the materialisation functor $\Mat: \operatorname{Pro}(\mathcal{S}_\pi) \to \mathcal{S}$, and keeping with our conventions, we denote the composition as $\Phi^s X \in \Fun(\mathcal{C},\mathcal{S})$. Similarly, we denote by $\tau_{\leq n}X$ the composition
$$\tau_{\leq n}X\colon \mathcal{C} \xrightarrow[]{X} \mathcal{S} \xrightarrow{\tau_{\leq n}}\mathcal{S}$$
In analogy to the similar situation in spaces, we call $\{ \tau_{\leq n}X\}_{n \in \mathbb{N}}$ the \emph{Postnikov decomposition} of $X$. A natural transformation $\alpha:X \to Y$ is said to be \emph{principal} if it is a principal fibration pointwise. A Postnikov decomposition of $X$ is then said to be \emph{principal} if all transformations $\tau_{\leq n+1}X \to \tau_{\leq n}X$ are principal.
\\

\begin{lemma}\label{lemma:pointwisefinite}
    For all $k \in \mathbb{N}$, and for $(M,\mathfrak{s})$ as above, $\Emb^{\Spin}(\sqcup_k \mathbb{R}^d, M)$ has the homotopy type of a simply connected, finite CW complex.
\end{lemma}

\begin{proof}
    Recall that $\Emb^{\Spin}(\sqcup_k \mathbb{R}^d, M)$ sits in the following homotopy pullback square
    \begin{center}
        \begin{tikzcd}
            \Emb^{\Spin}(\sqcup_k \mathbb{R}^d,M) \arrow{r} \arrow{d} & \Fr^{\Spin}(M)^k \arrow{d}
            \\
            \Conf_k (M) \arrow[hook]{r}  &  M^k
        \end{tikzcd}
    \end{center}
    The rightmost vertical map is the cartesian product of the spin frame bundle of $M$, which exists as $M$ is assumed to be 2-connected and thus admits a spin structure. Furthermore, we observe that the vertical fibre $\prod_{k} \Spin(d)$ in the above is homotopy equivalent to a finite CW complex; indeed, $\Spin(d)$ is the double cover of $SO(d)$ (universal for $d \geq 3$), a compact Lie group. Additionally, note that $M\setminus\sqcup_{k}\ast$ is again a 2-connected manifold, for all $k$. Induction on the Fadell-Neuwirth fibre bundle $M\setminus\sqcup_{k-1}\ast \to \Conf_k (M) \to \Conf_{k-1}(M)$ thus yields the result: one first observes that $\pi_1 \Conf_k (M)$ and $\pi_2 \Conf_k (M)$ are trivial for all $k\in\mathbb{N}$, so that the long exact sequence on homotopy groups applied to the fibration $\prod_k \Spin(d) \to \Emb^{\Spin}(\sqcup_k \mathbb{R}^d,M) \to \Conf_k (M)$ implies that the space in question is simply connected. That it has the homotopy type of a finite CW complex also follows from the fact that it sits as the total space of a fiber bundle whose fibre and base space have the homotopy type of finite CW complexes. Indeed, $\Conf_k (M)$ is homotopy equivalent to a finite CW complex for all $k$, since it is homeomorphic to the interior of a compact topological manifold with boundary via the Fulton-MacPherson compactification.
\end{proof}

The above analysis thus yields the following.

\begin{corollary} \label{Principal_Postnikov_Spin}
    The presheaf $E_{M}^{\Spin}\in \Psh(\Disk_d ^{\Spin})$ admits a principal Postnikov decomposition. In particular, so do all restrictions $\iota_{k}^{\ast}E_M \in \Psh(\Disk_{d}^{\Spin,\leq k})$, for $k\in\mathbb{N}$.
\end{corollary}

Using the above corollary, we show the following proposition, which is a generalization of Sullivan's theorem (\cref{thm:Sullivan3.2}) to the case of certain diagrams of spaces.

\begin{theorem}\label{thm:injective-pi_0}
    Let $M$ be a closed, smooth $2$-connected $d$-manifold, where $d \geq 4$. Then, for every $k \in \mathbb{N}$, the composition map
    $$\Map_{\Psh(\Disk^{\Spin,\leq k}_{d})}(\iota^{\ast}_{k}E_M^{\Spin},\iota^{\ast}_{k}E_{M}^{\Spin}) \to \Map_{\Psh(\Disk^{\Spin,\leq k}_{d})}(\iota^{\ast}_{k}E_M^{\Spin},\iota^{\ast}_{k}\Phi^sE_{M}^{\Spin})$$
    induces an injection on $\pi_0$.
\end{theorem}

We present a proof for $k=1$, and delay the case $k\geq2$ for the following section, as some more technology will have to go into a proof in that setting. We point out that the following proof is precisely an equivariant version of the proof of \cref{thm:Sullivan3.2}, and therefore we will be following the aforementioned proof quite closely.

\begin{proof}[Proof for $k=1$]
    Observe that $\iota_1 ^\ast E_M^{\Spin} \simeq \Fr^{\Spin}(M)$, the spin frame bundle of $M$. We furthermore observe that $\Psh(\Disk_d^{\Spin,\leq1}) \simeq \Psh(B\Spin(d))$, the category of spaces carrying an action by the Lie group $\Spin(d)$. The strategy to show injectivity is to work up the Postnikov ladder: given a $\Spin(d)$-equivariant map $\varphi \colon \Fr^\Spin (M) \to \tau_{\leq n}\Fr^\Spin (M)$, we consider two non-homotopic lifts to $\tau_{\leq n+1}\Fr^\Spin (M)$. The key claim is that if we compose with the finite completion on the target, the two lifts remain non-homotopic. As $\Fr^\Spin(M)$ is a finite $\Spin(d)$-CW complex, cohomological boundedness implies that for some $N \in \mathbb{N}$, there are unique lifts to $\tau_{\leq \ell}\Fr^\Spin (M)$, for $\ell \geq N$; thus, the claim follows by finite induction from the above key claim. Consider the situation of the following diagram in $\Psh(B\Spin(d))$
    \begin{center}
        \begin{tikzcd}
                &   \tau_{\leq n+1}\Fr^{\Spin}(M) \arrow{d}\arrow{r}   &    \ast \arrow{d}
            \\
            \Fr^{\Spin}(M) \arrow{r}[swap]{\varphi} \arrow[bend left, dashed]{ur}{f} \arrow[dashed]{ur}[swap]{g} &  \tau_{\leq n}\Fr^{\Spin}(M) \arrow{r}[swap]{\kappa_n}   &  K\defeq K(\pi_{n+1}\Fr^{\Spin}(M),n+2)
        \end{tikzcd}
    \end{center}
    where the square is a pullback square in $\Psh(B\Spin(d))$, and where the action of $\Spin(d)$ on the Eilenberg-Maclane space in the lower right corner is trivial. We fix a map $\varphi:\Fr^{\Spin}(M) \to \tau_{\leq n}\Fr^{\Spin}(M)$, and consider two lifts $f$ and $g$. The goal is to study how differen $f$ and $g$ can be. As the right hand side square is a pullback square, each of these two lifts is equivalent to a nullhomotopy of $\kappa_n \circ \varphi$. Using the Grothendieck construction $\Psh(B\Spin(d)) \simeq \mathcal{S}/B\Spin(d)$, and since the action of $\Spin(d)$ on the Eilenberg-Maclane space is trivial, we obtain  equivalences
    \begin{align*}
        \Map^{\Spin(d)}(\Fr^{\Spin}(M),K) &\simeq \Map_{\mathcal{S}/B\Spin(d)}\left(\Fr^{\Spin}(M)/\!/\Spin(d),K \times B\Spin(d)\right)\\
        &\simeq \Map(M,K)
    \end{align*}
    Consequently, it follows that a nulhomotopy of $\kappa_n \circ \varphi$ through $\Spin(d)$-equivariant maps, is equivalent to a nullhomotopy of the corresponding map between $M$ and $K$ in $\mathcal{S}$. For the lifts $f$, $g$, we get maps $N(f),N(g):CM \to K$ from the cone of $M$, which can be glued to a map $\Sigma M \to K$, resulting in a map $\Delta(f,g) \colon M \to \Omega K$, whose corresponding cohomology class $[\Delta(f,g)] \in H^{n+1}(M;\pi_{n+1}\Fr^{\Spin}(M))$ characterizes when $f$ and $g$ are homotopic through lifts:
    $$f \simeq g \text{ as lifts} \iff [\Delta(f,g)]=0 \in H^{n+1}(M;\pi_{n+1}\Fr^{\Spin}(M))$$
    The cohomology class $[\Delta(f,g)]$ can similarly characterize when $f$ and $g$ are homotopic:
    $$f \simeq g \iff \Delta(f,g)\in \mathrm{Im}\left(G \xrightarrow[]{\Omega\kappa_n} H^{n+1}(M;\pi_{n+1}\Fr^{\Spin}(M))\right)$$
    where $G \defeq \pi_0\Omega_{\varphi}\Map^{\Spin}(\Fr^{\Spin}(M),\tau_{\leq n}\Fr^{\Spin}(M))$. We now apply finite completion on the target, and observe that we obtain the same pullback square above, along with the lifting problem for the map $$\eta\circ\varphi:\Fr^{\Spin}(M) \xrightarrow[]{\varphi} \tau_{\leq n}\Fr^{\Spin}(M) \xrightarrow[]{\eta} \Phi^s\tau_{\leq n}\Fr^{\Spin}(M)$$
    Since $\pi_{n+1}\Fr^{\Spin}(M)$ is a finitely generated group, we know that 
    $$\Phi^sK(\pi_{n+1}\Fr^{\Spin}(M),n+2) \simeq K\left(\Phi^g\pi_{n+1}\Fr^{\Spin}(M),n+2\right)$$
    Given two lifts $f,g$ prior to finite completion, we obtain two corresponding lifts of $\eta\circ\varphi$ upon profinite completion, namely $\eta\circ f,\eta\circ g$ respectively. Composition with finite completion 
    $$H^{n+1}(M;\pi_{n+1}\Fr^{\Spin}(M))$$
    maps the obstruction $\Delta(f,g)$ to the corresponding obstruction $\Delta(\eta\circ f,\eta\circ g)$. We further denote $G'$ to be the group
    $$G'\defeq \pi_0 \Omega_{\eta\circ\varphi}\Map^{\Spin}(\Fr^{\Spin}(M),\Phi^s\tau_{\leq n}\Fr^{\Spin}(M))$$
    We now consider the diagram
    \begin{center}
        \begin{tikzcd}
            G  \arrow{r} \arrow{d}[swap]{\Omega\kappa_n}  &   G' \arrow{d}{\Omega\eta\circ\kappa_n}
            \\
            H^{n+1}(M;\pi_{n+1}\Fr^{\Spin}(M)) \arrow{r} \arrow[two heads]{d}  &  H^{n+1}(M;\Phi^g\pi_{n+1}\Fr^{\Spin}(M)) \arrow[two heads]{d}
            \\
            \coker{\Omega\kappa_n} \arrow{r}  &  \coker(\Omega\eta\circ\kappa_n)
        \end{tikzcd}
    \end{center}
    Similarly to the proof of \cref{thm:Sullivan3.2}, the argument boils down to showing the lower horizontal map is injective. As $\coker \Omega\kappa_n$ is a finitely generated abelian group - which we recall are residually finite - it would suffice to show that $\coker\Omega\Phi^s\kappa_n$ is isomorphic to the underlying group of the profinite completion of $\coker\Omega\kappa_n$, and that the map is given by profinite completion. This follows if we can show that the top square is isomorphic, in the category $\mathrm{Sq}(\Grp)$ of squares in groups, to the square
    \begin{center}
        \begin{tikzcd}
            G \arrow{r}{\eta} \arrow{d}[swap]{\Omega \kappa_n}  & \Phi^g G \arrow{d}{\Phi^g \Omega(\eta \circ \kappa_n)}
            \\
            H^{n+1}(M;\pi_{n+1}\Fr^\Spin (M)) \arrow{r}[swap]{\eta} & \Phi^g H^{n+1}(M;\pi_{n+1}\Fr^\Spin (M))
        \end{tikzcd}
    \end{center}
    obtained applying the natural transformation $\mathrm{id} \Rightarrow \Phi^g$ to the arrow
    $$G \xrightarrow[]{\Omega \kappa_n}H^{n+1}(M;\pi_{n+1}\Fr^\Spin(M))$$ we refer to the proof of \cref{thm:Sullivan3.2} for further details.
    \\
    
    We show how this works for the arrow $G \to G'$ first; the statement for the square works in the same way, as showcased in the proof of \cref{thm:Sullivan3.2}. By the Grothendieck equivalence $\Psh(B\Spin(d)) \simeq \mathcal{S}_{/B\Spin(d)}$, we note that $\Fr^\Spin (M)$ is mapped to $M \to B\Spin(d)$, classifying the spin tangent bundle. As $M \simeq \colim_M \ast$ in $\mathcal{S}$, this gives a description of $M \to B\Spin(d)$ as $\colim_M (\ast \to B\Spin(d))$, since colimits in overcategories are computed in the underlying category. Going back via the Grothendieck equivalence, we obtain
    $$\Fr^\Spin(M) \simeq \colim_M \Spin(d)$$
    where the diagram is given by sending $p \in M$ to $\Spin(T_p M)$. The above equivalence is the categorical analogue of the fact that $\Fr^\Spin (M)$ admits a free, finite $\Spin(d)$-CW-structure, obtained by pulling back a finite CW-structure on $M$ against the principal $\Spin(d)$-bundle $\Fr^\Spin(M) \to M$. We thus obtain the equivalences
    \begin{align*}
        \Map^\Spin (\Fr^\Spin(M),\Phi^s \tau_{\leq n}\Fr^\Spin (M)) & \simeq \Map^\Spin (\colim_M \Spin(d),\Phi^s \tau_{\leq n}\Fr^\Spin (M))
        \\
        &\simeq \lim_M \Map^\Spin(\Spin(d),\Phi^s \tau_{\leq n}\Fr^\Spin (M))
        \\
        &\simeq \lim_M \Map_\mathcal{S}(\ast,\Phi^s \tau_{\leq n}\Fr^\Spin (M))
        \\
        &\simeq \lim_M \Phi^s \tau_{\leq n}\Fr^\Spin (M) 
    \end{align*}
    We observe that $\tau_{\leq n}\Fr^\Spin(M)$ is a simply connected space of finite type, and that the above limit is a finite limit, as $M$ is a closed manifold; thus, \cref{cor:finite-limits-finite-completion} applies, and supplies an isomorphism on $\pi_1$
    $$\Phi^s\lim_M  \tau_{\leq n}\Fr^\Spin (M) \to \lim_M \Phi^s \tau_{\leq n}\Fr^\Spin (M) $$
    on choices of basepoints coming from the images of $\varphi \in \lim_M \tau_{\leq n} \Fr^\Spin (M)$ via the two canonical maps. \cref{thm:Sullivan3.1} now implies that $G' \cong \Phi^g G$, and that the map $G \to G' \xrightarrow[]{\cong} \Phi^g G$ agrees with the finite completion map. The remainder of the proof carries through as for the proof of \cref{thm:Sullivan3.2}.
\end{proof}

\subsection{Restricted Postnikov truncations}
We now aim for a proof of \cref{thm:injective-pi_0} for $k\geq 2$. In order to do this, we wish to construct a \emph{restricted Postnikov decomposition} for $X \in \Psh(\Disk_d^{\leq k})$, namely an inverse system $\{\widetilde{\tau_{\leq n}} X\}_{n \in \mathbb{N}}$, along with morphisms $X \to \widetilde{\tau_{\leq n}}X$ satisfying the following
\begin{enumerate}[(i)]
    \item The induced map
    $$X \to \lim_{n}\widetilde{\tau_{\leq n}}X$$
    is an equivalence
    \item For $\ell < k$, and for all $n \in \mathbb{N}$, the map
    $$X(\sqcup_{\ell}\mathbb{R}^d) \xrightarrow[]{\simeq}\widetilde{\tau_{\leq n}}X(\sqcup_{\ell}\mathbb{R}^d)$$
    is an equivalence
    \item The tower evaluated on $\sqcup_{k}\mathbb{R}^d$ is given by a sequence of fibrations with fibre Eilenberg-MacLane spaces.
\end{enumerate}

\begin{construction}\label{Construction:restricted-Postnikov}
    Any $X \in \Psh(\Disk_d^{\leq k})$ comes equipped with a canonical morphism $X \to (\iota_{k-1})_\ast\iota_{k-1}^{\ast}X$. From $X$, we define a functor
$$X/T_{k-1}X \colon (\Disk_d^{\leq k})^{\op} \to \mathrm{Ar}(\mathcal{S})$$
by $D \mapsto \left(X(D) \to (\iota_{k-1})_{\ast}\iota_{k-1}^{\ast}X(D)\right)$. For each $n \in \{-1,0\}\cup\mathbb{N}$, we define $\widetilde{\tau_{\leq n}}X$ to be the composition
$$\widetilde{\tau_{\leq n}}X \colon (\Disk_d^{\leq k})^{\op} \xrightarrow[]{X/T_{k-1} X}\mathrm{Ar}(\mathcal{S}) \xrightarrow[]{\mathrm{MP}_n}\mathcal{S}$$
where $\mathrm{MP}_n$ is the $n$-th stage Moore-Postnikov functor of §\ref{Moore-Postnikov-functor}.
\end{construction}
The following is immediate from the above construction.
\begin{lemma}
    The inverse system $\widetilde{\tau_{\leq n}}X$ is a restricted Postnikov decomposition of the presheaf $X \in \Psh(\Disk_d^{\leq k})$, i.e. satisfies (i), (ii) and (iii) from above.
\end{lemma}
\begin{remark}

The above construction can be seen as a certain localization onto a full-subcategory of $\Psh(\Disk_{d}^{\leq k})$: for each $n\in\{-1,0\}\cup\mathbb{N}$, we say $X$ is \emph{$n$-truncated over $\Disk_d^{\leq k-1}$} if the map 
$X(\sqcup_{k}\mathbb{R}^d) \to (\iota_{k-1})_{\ast}\iota_{k-1}^{\ast}X(\sqcup_k \mathbb{R}^d)$ is $n$-truncated. $\mathcal{C}_n$ is then defined to be the full subcategory of $\Psh(\Disk_{d}^{\leq k})$ of those objects $X \in \Psh(\Disk_d^{\leq k})$ which are $n$-truncated over $\Disk_{d}^{\leq k}$. Observe that for any $X \in \Psh(\Disk_{d}^{\leq k})$, $\widetilde{\tau_{\leq n}}X \in \mathcal{C}_n$. We thus define a direct system of full subcategories $\mathcal{C}_n$ of $\Psh(\Disk_d^{\leq k})$ such that $\mathcal{C}_{-1} \simeq \Psh(\Disk_{d}^{\leq k-1})$, $\colim_{n} \mathcal{C}_n \simeq \Psh(\Disk_d^{\leq k})$, and the map $\mathcal{C}_{-1} \to \colim_n \mathcal{C}_n$ agrees with the inclusion functor $\Psh(\Disk_d^{\leq 1}) \hookrightarrow \Psh(\Disk_d^{\leq k})$. For each $n$, the inclusion $\mathcal{c}_n \colon \mathcal{C}_n \hookrightarrow \Psh(\Disk_d^{\leq k})$ admits a left adjoint $(\mathcal{c}_n)_{\ast} \colon \Psh(\Disk_d^{\leq k}) \to \mathcal{C}_{n}$, and it can be shown that $(\mathcal{c}_n)_\ast X \simeq \widetilde{\tau}_{\leq n}X$.
\end{remark}

\bigskip
Recollement/Reedy extensions, following \cite[Theorem 4.9 (i)]{Manuel-Sander} and \cite{Jan-Maxime}, allow us to show the following key step. Let $\Disk_d^{=k,\Spin}$ be the category consisting of a single object $\sqcup_k \mathbb{R}^d$, where the morphism space is $\Emb^{\Spin}_{\pi_0\text{-inj}}(\sqcup_k \mathbb{R}^d,\sqcup_k \mathbb{R}^d)$ of those embeddings that induce a bijection on $\pi_0$. We observe that $\Disk_{d}^{=k,\Spin}$ is equivalent to the category $BG_k$, where $G_k \defeq \mathfrak{S}_k \ltimes \Spin(d)^k$, the action given by permutation of factors. There is an inclusion functor $\mathcal{j}:\Disk_{d}^{=k,\Spin} \to \Disk_d^{\leq k,\Spin}$, which induces a restriction functor $\mathcal{j}^\ast \colon \Psh(\Disk_d^{\leq k,\Spin}) \to \Psh(\Disk_{d}^{=k,\Spin})$.
\begin{theorem}\label{thm:Reedy-pullback}
    We have the following pullback diagram of $\infty$-categories
    \begin{center}
        \begin{tikzcd}
            \Psh(\Disk_{d}^{\leq k,\Spin}) \arrow{r}{\circled{1}} \arrow{d}[swap]{\iota_{k-1}^\ast} & \Fun(\Delta^2,\Psh(\Disk_d^{=k,\Spin})) \arrow{d}{\circled{2}}
            \\
            \Psh(\Disk_{d}^{\leq k-1,\Spin}) \arrow{r}[swap]{\circled{3}} & \Fun(\Delta^1,\Psh(\Disk_d^{=k,\Spin}))
        \end{tikzcd}
    \end{center}
    where the three other functors are the following
    \begin{itemize}
        \item $\circled{1} \colon \Psh(\Disk_{d}^{\leq k,\Spin}) \to \Fun(\Delta^2,\Psh(\Disk_{d}^{=k,\Spin}))$ maps $X \in \Psh(\Disk_d^{\leq k,\Spin})$ to the triangle in $\Psh(\Disk_{d}^{=k,\Spin})$
        \begin{center}
            \begin{tikzcd}
                \mathcal{j}^\ast (\iota_{k-1})_{!}\iota_{k-1}^\ast X \arrow{r} \arrow{d} & \mathcal{j}^\ast X \arrow{ld}
                \\
                \mathcal{j}^\ast (\iota_{k-1})_\ast \iota_{k-1}^\ast X
            \end{tikzcd}
        \end{center}

        \item The functor $\circled{2}\colon\Fun(\Delta^2,\Psh(\Disk_d^{=k,\Spin})) \to \Fun(\Delta^1,\Psh(\Disk_d^{=k,\Spin}))$ is given by precomposition with the map $\Delta^1 \to \Delta^2$, $0 \mapsto 0$ and $1 \mapsto 2$.
        \item The functor $\circled{3}\colon\Psh(\Disk_{d}^{\leq k-1,\Spin}) \to \Fun(\Delta^1,\Psh(\Disk_{d}^{=k,\Spin}))$ maps $X$ to the arrow
        $$j^{\ast}(\iota_{k-1})_! X \to j^{\ast}(\iota_{k-1})_\ast X$$
        
    \end{itemize}
\end{theorem}

\begin{remark}
    We emphasize that the above theorem works in general for all tangential structures $\Theta \in \Psh(BO(d))$.
\end{remark}

We now combine the previous pullback square with \cref{Construction:restricted-Postnikov}. For $X \in \Psh(\Disk_d^{\leq k,\Spin})$, let $\{\widetilde{\tau_{\leq n}}X\}$ be its restricted Postnikov decomposition. In our aim of studying self-maps of $X$ in $\Psh(\Disk_{d}^{\leq k,\Spin})$, we observe that
\begin{align*}
    \Map_{\Psh(\Disk_d^{\leq k,\Spin})}(X,X) &\simeq \Map_{\Psh(\Disk_d^{\leq k,\Spin})}(X,\lim_{n}\widetilde{\tau_{\leq n}}X)
    \\
    &\simeq \lim_{n}\Map_{\Psh(\Disk_d^{\leq k,\Spin})}(X,\widetilde{\tau_{\leq n}}X)
\end{align*}
It would be ideal for our purposes if the consecutive stages of the above limit are governed, as for usual mapping spaces, by some cohomology group; the following result is a step in this direction, and follows readily from \cref{thm:Reedy-pullback}.

\begin{proposition}\label{Prop:Reedy-Lifts}
    For any $X \in \Psh(\Disk_{d}^{\leq k,\Spin})$, evaluation at $D_k \defeq \sqcup_k \mathbb{R}^d$ yields an equivalence
    $$\mathrm{Lift}_{\Psh(\Disk_{d}^{\leq k,\Spin})}\left( \begin{tikzcd}
        & \widetilde{\tau_{\leq n+1}}X \arrow{d}
        \\
        X \arrow[dashed]{ur} \arrow{r} & \widetilde{\tau_{\leq n}}X
    \end{tikzcd} \right) \xrightarrow[]{\simeq} \mathrm{Lift}_{G_k}\left( \begin{tikzcd}
        \mathcal{j}^\ast (\iota_1)_!\iota_1^\ast X(D_k) \arrow{d} \arrow{r} & \mathcal{j}^\ast\widetilde{\tau_{\leq n+1}}X(D_k) \arrow{d}
        \\
        \mathcal{j}^\ast X(D_k) \arrow[dashed]{ur}\arrow{r} & \mathcal{j}^\ast\widetilde{\tau_{\leq n}}X(D_k)
    \end{tikzcd} \right)$$
    where $G_k \defeq \mathfrak{S}_k\ltimes \Spin(d)^k \simeq \Emb^{\Spin}_{\pi_0\text{-inj}}( \sqcup_k \mathbb{R}^d , \sqcup_k \mathbb{R}^d)$, and where the target space denotes the space of $G_k$-equivariant lifts.
\end{proposition}

We now specialize to the case of $X=E_M^{\Spin}$ for a 2-connected closed manifold. We first observe the following. For a closed manifold, we let $\overline{\Conf}_k (M)$ denote the Fulton-MacPherson compactification of $\Conf_k (M)$; this is a compact manifold with corners containing $\Conf_k (M)$ as its interior. Let $\partial\overline{\Conf}_k(M)$ denote its boundary.

\begin{lemma}\label{lemma:leftKan-is-boundary}
    For $M$ a 2-connected closed manifold, and for $D_k \defeq \sqcup_k \mathbb{R}^d$, $\mathcal{j}^{\ast}(\iota_{k-1})_!\iota_{k-1}^\ast E_M^{\Spin}(D_k)$ sits in the following pullback diagram
    \begin{center}
        \begin{tikzcd}
            \mathcal{j}^\ast (\iota_{k-1})_! \iota_{k-1}^\ast E_M^{\Spin}(D_k) \arrow{d} \arrow{r} & \Emb^{\Spin}(D_k,M) \arrow{d}{\mathrm{ev_0}}
            \\
            \partial\overline{\Conf}_k(M) \arrow{r}[swap]{\partial} & \Conf_k (M)
        \end{tikzcd}
    \end{center}
    where the bottom arrow is defined after choice of a boundary collar on $\overline{\Conf}_k (M) $.
\end{lemma}
\begin{proof}
    This follows from \cite[Proposition 5.12]{Manuel-Sander}.
\end{proof}

It follows in particular that the pair $(\mathcal{j}^\ast E_M^{\Spin}(D_k),\mathcal{j}^\ast (\iota_{k-1})_! \iota_{k-1}^\ast E_M^{\Spin}(D_k))$, for $D_k=\sqcup_k \mathbb{R}^d$, is equivalent to a finite, free $G_k$-CW pair, which, for simplicity, we henceforth denote by
$$(E_M^{\Spin}(D_k),\partial E_M^{\Spin}(D_k))$$
\\

We now start with a warm-up to the proof of \cref{thm:injective-pi_0}. The main advantage of \cref{Prop:Reedy-Lifts} is that it enables us to study a space of lifts of maps in a complicated category in terms of obstruction theory of equivariant spaces. We now carry out this analysis for $X = E_M^{\Spin}$. For ease of notation, we let $D_k$ denote the disjoint union of $k$-many discs, and we again let $G_k \defeq \mathfrak{S}_k\ltimes\Spin(d)^k$. We first observe that $\pi_1 E_M^{\Spin}(D_k) \cong 1$, so that the fibration $\widetilde{\tau_{\leq n+1}}E_M^{\Spin}(D_k) \to \widetilde{\tau_{\leq n}}E_M^{\Spin}(D_k)$ may be delooped to a pullback diagram in $\Psh(BG_k)$
\begin{center}
    \begin{tikzcd}
        \widetilde{\tau_{\leq n+1}}E_M^{\Spin}(D_k) \arrow{r} \arrow{d}  & \ast \arrow{d}
        \\
        \widetilde{\tau_{\leq n}}E_M^{\Spin}(D_k) \arrow{r}[swap]{\kappa_n}& K(\pi_{n+1}F,n+2)
    \end{tikzcd}
\end{center}
where $F$ is the fibre of the map $E_M^{\Spin}(D_k) \to (\iota_{k-1})_\ast \iota_{k-1}^\ast E_M^{\Spin}(D_k)$, on which $G_k$ acts. Observe now that $\pi_{n+1}F$ is a finitely generated group, which admits a potentially non-trivial action of $G_k$, as $G_k$ is not connected. We now fix a pair of maps $(\varphi,\partial\varphi):(E_M^{\Spin}(D_k), \partial E_M^{\Spin}(D_k)) \to (\widetilde{\tau_{\leq n}}E_M^{\Spin}(D_k), \widetilde{\tau_{\leq n+1}}E_M^{\Spin}(D_k))$, where $\partial E_M^{\Spin}(D_k)$ is to be understood as in \cref{lemma:leftKan-is-boundary}. We are thus looking at the following lifting problem in $\Psh(BG_k)$
\begin{center}
    \begin{tikzcd}
        \partial E_M^{\Spin}(D_k) \arrow{r}{\partial\varphi} \arrow{d} &  \widetilde{\tau_{\leq n+1}}E_M^{\Spin}(D_k) \arrow{d} \arrow{r} & \ast \arrow{d}
        \\
        E_M^{\Spin}(D_k) \arrow{r}[swap]{\varphi} \arrow[dashed]{ur}{f} &  \widetilde{\tau_{\leq n}}E_M^{\Spin}(D_k) \arrow{r}[swap]{\kappa_n}& K(\pi_{n+1}F,n+2)
    \end{tikzcd}
\end{center}
As the right hand square is a pullback square, a lift $f$ is therefore equivalent to a nullhomotopy of $\kappa_n \circ \varphi$ relative to $\partial E_M^{\Spin}(D_k)$, or equivalently, a map from the relative cone of $(E_M^{\Spin}(D_k),\partial E_M^{\Spin}(D_k))$ to $K(\pi_{n+1}F,n+2)$. Given two lifts $f,g$, we may glue the two maps from the cone to obtain a map from the suspension, which by adjunction yields a relative map $\Delta(f,g)\colon(E_M^{\Spin}(D_k),\partial E_M^{\Spin}(D_k)) \to (K(\pi_{n+1}F,n+1),\ast)$. Considered up to homotopy, this gives a class $[\Delta(f,g)] \in H^{n+1}_{G_k}(E_M^{\Spin}(D_k),\partial E_M^{\Spin}(D_k);\pi_{n+1}F)$ in the relative $G_k$-equivariant cohomology group with coefficients in the $\mathbb{Z}G_k$-module $\pi_{n+1}F$. We note that the above cohomology group is a finitely generated abelian group, as the pair is equivalent to a finite, free $G_k$-CW pair, and $\pi_{n+1}F$ is finitely generated. The class $\Delta(f,g)$ satisfies, as usual, the following property
\begin{center}
    $[\Delta(f,g)]=0 \iff f\simeq g$ through equivariant lifts
\end{center}
As in the previous setting, we are after a subgroup $K$ of $ H^{n+1}_{G_k}(E_M^{\Spin}(D_k),\partial E_M^{\Spin}(D_k);\pi_{n+1}F)$, with the further property that
\begin{center}
    $f\simeq g \iff [\Delta(f,g)]\in K$
\end{center}
It can be readily seen that such a subgroup is given by the image of
$$\pi_0\Omega_{(\varphi,\partial\varphi)}\Map_{\mathrm{Ar}(\Psh(BG_k))}\left((E_M^{\Spin}(D_k),\partial E_M^{\Spin}(D_k)),(\widetilde{\tau_{\leq n}}E_M^{\Spin}(D_k),\widetilde{\tau_{\leq n+1}}E_M^{\Spin}(D_k))\right)$$
in the equivariant cohomology group by the obvious composition morphism. We can repeat the above after finite completion on the target. We first record some technicalities that prove useful.

We now treat the case of relative mapping spaces. 
\begin{lemma}\label{lemma:Profinite-after-postcomposition}
    The maps 
    \begin{center}
       \small $\Phi^s\Map^{G_k}\left(\begin{tikzcd}
            \partial E_M^{\Spin}(D_k) \arrow{d}
            \\
            E_M^{\Spin}(D_k)
        \end{tikzcd},\begin{tikzcd}
            \widetilde{\tau_{\leq n+1}}E_M^{\Spin}(D_k) \arrow{d}
            \\
            \widetilde{\tau_{\leq n+1}}E_M^{\Spin}
        \end{tikzcd}\right)\to \Map^{G_k}\left(\begin{tikzcd}
            \partial E_M^{\Spin}(D_k) \arrow{d}
            \\
            E_M^{\Spin}(D_k)
        \end{tikzcd},\begin{tikzcd}
            \Phi^s\widetilde{\tau_{\leq n+1}}E_M^{\Spin}(D_k) \arrow{d}
            \\
            \Phi^s\widetilde{\tau_{\leq n+1}}E_M^{\Spin}
        \end{tikzcd}\right)$
    \end{center}
    and
    \begin{center}
       \small $\Phi^s\Map^{G_k}\left(\begin{tikzcd}
            \partial E_M^{\Spin}(D_k) \arrow{d}
            \\
            E_M^{\Spin}(D_k)
        \end{tikzcd},\begin{tikzcd}
            \ast \arrow{d}
            \\
            K(\pi_{n+1}F,n+2)
        \end{tikzcd}\right)\to \Map^{G_k}\left(\begin{tikzcd}
            \partial E_M^{\Spin}(D_k) \arrow{d}
            \\
            E_M^{\Spin}(D_k)
        \end{tikzcd},\begin{tikzcd}
            \ast \arrow{d}
            \\
            \Phi^s K(\pi_{n+1}F,n+2)
        \end{tikzcd}\right)$
    \end{center}
    induce isomorphisms on all $\pi_n$, $n \geq 1$, on basepoints coming from the image of the corresponding spaces before finite completions. In the above, the mapping spaces are taken in the category $\Fun(\Delta^1,\Psh(BG_k))$, the category of arrows in spaces with an action of $G_k$.
\end{lemma}
\begin{proof}
    We prove it for one of the above spaces, as the next one works in the exact same way. By the description of mapping spaces in arrow categories, the mapping space in consideration can be seen as the pullback of the following cospan
    \begin{center}
        \begin{tikzcd}
             &  \Map^{G_k}(\partial E_M^{\Spin}(D_k), \Phi^s\widetilde{\tau_{\leq n+1}}E_M^{\Spin}(D_k)) \arrow{d}
             \\
             \Map^{G_k}(E_M^{\Spin}(D_k),\Phi^s\widetilde{\tau_{\leq n}}E_M^{\Spin}(D_k)) \arrow{r} & \Map^{G_k}(\partial E_M^{\Spin}(D_k),\Phi^s\widetilde{\tau_{\leq n}}E_M^{\Spin}(D_k))
        \end{tikzcd}
    \end{center}
    where the maps are the obvious post/pre-composition, and where we observe that all the spaces in the above diagram can be written as finite limits of the finite completion of finite type, simply connected spaces. The result follows from iterated application of \cref{thm:finite-limit-finite-completion}.
\end{proof}

We now proceed to show \cref{thm:injective-pi_0}.
\begin{proof}[Proof of \cref{thm:injective-pi_0}]
   The proof proceeds by induction, and follows a strategy similar to that of Sullivan's proof. Let $n \in \mathbb{N}$, and consider the following lifting problem
   \begin{center}
       \begin{tikzcd}
            & \widetilde{\tau_{\leq n+1}}E_M^{\Spin} \arrow{d}
            \\
            E_M^{\Spin} \arrow{r}[swap]{\varphi} \arrow[dashed]{ur} & \widetilde{\tau_{\leq n}}E_M^{\Spin}
       \end{tikzcd}
   \end{center}
   By \cref{Prop:Reedy-Lifts} and \cref{lemma:leftKan-is-boundary}, this translates to a lifting problem in $\Psh(BG_k)$, where $G_k \defeq \mathfrak{S}_k \ltimes \Spin(d)^k$:
   \begin{center}
       \begin{tikzcd}
           \partial E_M^{\Spin}(D_k) \arrow{r}{\partial\varphi}\arrow{d} & \widetilde{\tau_{\leq n+1}}E_M^{\Spin}(D_k) \arrow{r}\arrow{d} & \ast \arrow{d}
           \\
           E_M^{\Spin}(D_k) \arrow{r}[swap]{\varphi} \arrow[dashed]{ur} & \widetilde{\tau_{\leq n}}E_M^{\Spin}(D_k) \arrow{r} & K(\pi_{n+1}F,n+2)
       \end{tikzcd}
   \end{center}
   Given two lifts $f,g$, the above analysis yields a cohomology class in the $G_k$-equivariant cohomology group
   $$[\Delta(f,g)] \in H^{n+1}_{G_k}(E_M^{\Spin}(D_k),\partial E_M^{\Spin}(D_k);\pi_{n+1}F)$$
   whose vanishing implies the equivalence of $f$ and $g$ as lifts (we refer to the discussion after \cref{lemma:leftKan-is-boundary} for more details). For the sake of ease of notation, we denote the above equivariant cohomology group by $H$ throughout the proof. We also consider the group
   $$G \defeq \pi_0\Omega_{(\varphi,\partial\varphi)}\Map_{\mathrm{Ar}(\Psh(BG_k))}\left((E_M^{\Spin}(D_k),\partial E_M^{\Spin}(D_k)),(\widetilde{\tau_{\leq n}}E_M^{\Spin}(D_k),\widetilde{\tau_{\leq n+1}}E_M^{\Spin}(D_k))\right)$$
   which maps to $H$ by post-composition with the right hand square of the above diagram; we call the group morphism $\alpha \colon G \to H$. We observe that
   $$f \simeq g \iff [\Delta(f,g)] \in \mathrm{Im}(G \to H) $$
   We may repeat the same analysis after composing with finite completion on the target. Namely, given the above lifts $f,g$, we get lifts $\eta \circ f, \eta \circ g$ fitting in the following diagram
   
   \begin{center}
       \begin{tikzcd}
           \partial E_M^{\Spin}(D_k) \arrow{r}{\partial\eta\circ\varphi}\arrow{d} & \Phi^s \widetilde{\tau_{\leq n+1}}E_M^{\Spin}(D_k) \arrow{r}\arrow{d} & \ast \arrow{d}
           \\
           E_M^{\Spin}(D_k) \arrow{r}[swap]{\Phi\varphi} \arrow[dashed]{ur} & \Phi^s\widetilde{\tau_{\leq n}}E_M^{\Spin}(D_k) \arrow{r} & \Phi^sK(\pi_{n+1}F,n+2)
       \end{tikzcd}
   \end{center}
   The rightmost square is again a pullback square, and since $\pi_{n+1}F$ is a finitely generated abelian group, we furthermore obtain an equivalence 
   $$\Phi^s K(\pi_{n+1}F,n+2) \simeq K((\Phi^g\pi_{n+1}F),n+2)$$
   We thus obtain a cohomology class 
   $$[\Delta(\eta \circ f, \eta \circ g)] \in H^{n+1}_{G_k}\left(E_M^{\Spin}(D_k),\partial E_M^{\Spin}(D_k);\Phi^g\pi_{n+1}F)\right)$$
   whose vanishing implies that $\eta \circ f,\eta \circ g$ are equivalent as equivariant lifts. Again for ease of notation, we denote the above cohomology group by $H'$. We define $G'$ analogously to $G$, by finite completion on the target:
   $$G' \defeq \pi_0\Omega_{(\eta\varphi,\partial\eta\varphi)}\Map_{\mathrm{Ar}(\Psh(BG_k))}\left((E_M^{\Spin}(D_k),\partial E_M^{\Spin}(D_k)),(\Phi^s \widetilde{\tau_{\leq n}}E_M^{\Spin}(D_k),\Phi^s \widetilde{\tau_{\leq n+1}}E_M^{\Spin}(D_k))\right)$$
   The above group maps to $H'$ by composition with the pullback square on the target, and we call the group homomorphism $\alpha'$. Again, obstruction theory tells us that
   $$\Phi^sf \simeq \Phi^sg \iff [\Delta(\eta \circ f,\eta \circ g)] \in \mathrm{Im}(G' \to H') $$

   The theorem follows by induction once we show that if $f \not\simeq g$, then $\eta \circ f \not\simeq \eta \circ g$. To that end, we observe that we have a commutative diagram
   \begin{equation}\label{diagram5}
       \begin{tikzcd}
           G \arrow{r}\arrow{d}[swap]{\alpha} & G' \arrow{d}{\alpha'}
           \\
           H \arrow{r} \arrow{d}[swap]{p} & H' \arrow{d}{p'}
           \\
           \coker(\alpha) \arrow{r} & \coker(\alpha')
       \end{tikzcd}
   \end{equation}
   such that $[\Delta(f,g)]$ is mapped to $[\Delta(\eta \circ f,\eta \circ g)]$ via the middle horizontal map. The claim follows if we can show that the lower horizontal map is injective. Indeed, $f \not \simeq g$ is equivalent to $p([\Delta(f,g)]) \neq 0$ in $\coker(\alpha)$, and hence if the map $\coker(\alpha) \to \coker(\alpha')$ is injective, we obtain that $p'([\Delta(\eta \circ f, \eta \circ g)])\neq 0$ in $\coker(\alpha')$, which in turn is equivalent to $\eta\circ f \not \simeq \eta \circ g$. As $H$ is a finitely generated abelian group, so is $\coker(\alpha)$. In particular, as finitely generated abelian groups are residually finite, the map $\coker(\alpha) \to \coker(\alpha')$ would be injective if we can show that the arrow $\coker(\alpha) \to \coker(\alpha')$ is isomorphic in the category $\Fun(\Delta^1,\Grp)$ of arrows in groups, to the arrow $\coker(\alpha) \xrightarrow[]{\eta}\Phi^g \coker(\alpha)$, which is injective by residual finiteness of finitely generated abelian groups. To that end, we show that the top square of diagram \ref{diagram5} is isomorphic, in the category $\mathrm{Sq}(\Grp)$ of squares in groups, to the square
   \begin{center}
       \begin{tikzcd}
           G \arrow{d}[swap]{\alpha}  \arrow{r}{\eta_G}  & \Phi^g G \arrow{d}{\Phi^g \alpha}
           \\
           H \arrow{r}[swap]{\eta_H} & \Phi^g H
       \end{tikzcd}
   \end{center}

   By \cref{lemma:Profinite-after-postcomposition}, we observe that $G'$ and $H'$ are the fundamental groups (based at a point coming from $G$ and $H$, through the canonical map) of a finite limit of the finite completion of some (componentwise) finite type, nilpotent space. Then, \cref{cor:finite-limits-finite-completion} implies that the map $G \to G'$ and $H \to H'$ are identified with the group profinite completion maps.  Consequently, the remainder of the proof of \cref{thm:Sullivan3.2} carries through, and we obtain the desired result.

\end{proof}

\bigskip
Having finished the technical heart of the argument, we can now assemble the pieces and show that $\pi_0 T_k \Diff^{\Spin}(M)$ is a residually finite group, for every $k$.
\bigskip

\begin{proof}[Proof of \cref{thm:spin-T_k-residual-finite}]
    Consider the following diagram
    \begin{center}
    \begin{tikzcd}
        \pi_0 \Map^{\simeq}_{\Psh(\Disk_{d}^{\Spin,\leq k})}(\iota_k^{\ast}E_{M}^{\Spin},\iota_k^{\ast}E_{M}^{\Spin}) \arrow{r} \arrow{d} \arrow{dr} & \Phi^g \pi_0 \Map^{\simeq}_{\Psh(\Disk_{d}^{\Spin,\leq k})}(\iota_k^{\ast}E_{M}^{\Spin},\iota_k^{\ast}E_{M}^{\Spin})
        \arrow[dashed]{d}{\exists !}
        \\
         \pi_0 \Map^{\simeq}_{\Psh(\Disk_{d}^{\Spin,\leq k})}(\Phi^s\iota_k^{\ast}E_{M}^{\Spin},\Phi^s\iota_k^{\ast}E_{M}^{\Spin})  & \arrow{l} \pi_0 \Map^{\simeq}_{\Fun(\Disk_{d}^{\Spin,\leq k},\operatorname{Pro}(\mathcal{S}_\pi))}( \iota_k^{\ast}\widehat E_{M}^{\Spin}, \iota_k^{\ast}\widehat E_{M}^{\Spin})
    \end{tikzcd}
\end{center}
The existence (and uniqueness) of the right vertical map follows from \cref{thm:profinite}, and the universal property of profinite completion of groups. Consequently, the injectivity of the top horizontal morphism follows from the injectivity of the leftmost vertical morphism. Furthermore, we observe that from the canonical morphism $E_{M}^{\Spin} \to \Phi^sE_{M}^{\Spin}$, we obtain the following commuting diagram

\begin{center}
    \begin{tikzcd}
        \Map_{\Psh(\Disk_{d}^{\Spin,\leq k})}(\iota_k^{\ast}E_{M}^{\Spin},\iota_k^{\ast}E_{M}^{\Spin}) \arrow{r} \arrow{dr} &  \Map_{\Psh(\Disk_{d}^{\Spin,\leq k})}(\Phi\iota_k^{\ast} E_{M}^{\Spin},\Phi\iota_k^{\ast} E_{M}^{\Spin}) \arrow{d}
        \\
        & \Map_{\Psh(\Disk_{d}^{\Spin,\leq k})}(\iota_k^{\ast}E_{M}^{\Spin},\Phi\iota_k^{\ast} E_{M}^{\Spin})
    \end{tikzcd}
\end{center}
We deduce the injectivity of the horizontal map on $\pi_0$ from the injectivity of the diagonal one, as given by \cref{thm:injective-pi_0}. Finally, the injectivty on the invertible parts of $T_k \Emb^{\Spin}(M,M)$ follows from the following commutative diagram

\begin{center}
    \begin{tikzcd}
        \pi_0 \Map_{\Psh(\Disk_{d}^{\Spin,\leq k})}^{\simeq}(\iota_k^{\ast}E_{M}^{\Spin},\iota_{k}^{\ast}E_{M}^{\Spin}) \arrow{r} \arrow[hook]{d}  &   \pi_{0}\Map_{\Psh(\Disk_{d}^{\Spin,\leq k})}^{\simeq}(\Phi^s\iota_{k}^{\ast}E_{M}^{\Spin},\Phi^s\iota_{k}^{\ast}E_{M}^{\Spin}) \arrow[hook]{d}
        \\
        \pi_0 \Map_{\Psh(\Disk_{d}^{\Spin,\leq k})}(\iota_k^{\ast}E_{M}^{\Spin},\iota_{k}^{\ast}E_{M}^{\Spin}) \arrow[hook]{r}  &   \pi_{0}\Map_{\Psh(\Disk_{d}^{\Spin,\leq k})}(\Phi^s\iota_{k}^{\ast}E_{M}^{\Spin},\Phi^s\iota_{k}^{\ast}E_{M}^{\Spin})
    \end{tikzcd}
\end{center}

\end{proof}

\subsection{Main result: residual finiteness}
We are now ready to prove the main theorem of the section.

\begin{theorem}\label{thm:main-result-resfinite}
    Let $M$ be a smooth, 2-connected closed manifold of dimension $d\geq5$. Then, $\pi_0 T_k \Diff^\theta (M)$ is a residually finite group, for $\theta$ the spin, oriented or non-oriented tangential structure.
\end{theorem}
\begin{proof}
We fix $\mathfrak{s}$ a $\Spin$-structure on $M$, and let $\ell$ be the associated $\Theta$-structure. By \cref{thm:pullback-moduli-spaces},  the fibre $F$ of the map
$$BT_k\Diff^{\Spin}(M,\mathfrak{s}) \to BT_k \Diff^\theta(M,\ell)$$
over the identity is a collection of components of the space of lifts
$$\mathrm{Lift}\left(\begin{tikzcd}
    & B\Spin(d) \arrow{d}
    \\
    M \arrow{r}[swap]{\ell} \arrow[dashed]{ur} & B\Theta(d)
\end{tikzcd}\right) $$
and it is fairly straigthforward to see that both $\pi_0$ and $\pi_1$ (at each basepoint) of the above space are finite. We thus have an exact sequence
$$\pi_1 F \to \pi_1 BT_k \Diff^\Spin (M) \xrightarrow[]{f} \pi_1 BT_k \Diff^{\Theta(d)}(M) \to \pi_0 F$$
By \cref{thm:spin-T_k-residual-finite} and \cref{lemma:quotient-resfinite}, we know that $\pi_1 BT_k \Diff^{\Theta(d)} / \mathrm{Im}(f)$ is a residually finite group whenever $M$ is 2-connected, and can furthermore be identified with a finite index subgroup of $\pi_1 BT_k\Diff^{\Theta(d)}(M)$. Thus, by \cref{lemma:finite-index-resfin}, it also follows that $\pi_1 BT_k\Diff^{\Theta(d)}(M)$ is residually finite. The latter is isomorphic to $\pi_0 T_k \Diff^{\Theta(d)}(M)$, and the result follows.

\begin{remark}
    Via smoothing theory for embedding calculus, as in \cite[Theorem 4.15]{Manuel-Sander}, one concludes that $\pi_0 T_k \mathrm{Homeo}(M)$ is residually finite, under the same assumptions on $M$, where $T_k \Homeo(M)$ is the space of automorphisms of $\iota_k^\ast E_M$ in the category of $\Psh(\Disk_d^{\mathrm{Top},\leq k})$.
\end{remark}

\end{proof}

\subsection{Remarks on mapping class groups of $W_g^n$ and disc-structure spaces}
In \cite{discstructurespace}, Krannich and Kupers study the fibre of the embedding calculus map
$$\mathcal{S}^{\Disk}(M) \defeq \fib_{E_M}\left( \Manf_d^\cong \to \Psh(\Disk_d)^{\cong}\right)$$
where $M$ is a smooth, closed manifold, and where $\Manf_d^\cong$ denotes the groupoid core of the $\infty$-category of smooth, $d$-dimensional manifold with empty boundary and embeddings as spaces of 1-morphisms. The fibre over $E_M$ is called the \emph{disc-structure space of $M$}. Two important properties of the above space are
\begin{itemize}
    \item $\mathcal{S}^{\Disk}(M)$ is not contractible, when $M$ admits a spin structure, and $\dim M \geq 8$.
    \item $\mathcal{S}^{\Disk}(M)$ only depends on the tangential 2-type of $M$.
\end{itemize}

In particular, in the case $M = W_g^n \defeq \#_g(S^n \times S^n)$ where $n>3$, we know that $\Emb(W_g^n,W_g^n) \to T_\infty \Emb(W_g^n,W_g^n)$ is not an equivalence. Using \cref{thm:main-result-resfinite}, we observe that the non-equivalence can be already detected on $\pi_0$.

\begin{proposition}
    For $n=5\pmod8$ and $g \geq 5$, the group homomorphism $\pi_0\Diff^\theta (W^{n}_g) \to \pi_0 T_\infty \Diff^\theta (W^{n}_g)$ induced by the embedding calculus tower is not an isomorphism, for $\theta$ either the unoriented or oriented tangential structure.
\end{proposition}

\begin{proof}
We begin by showing that $\pi_1 T_k \Diff^\theta (M)$ is finitely generated for all $k\in\N$, where the basepoint here is taken to be the identity. Indeed, we know that $T_1 \Emb^\theta (M,M) \simeq \mathrm{Bun}^\theta (TM,TM)$, which sits in a fibration
$$\Map(M,G_\theta (d)) \to \mathrm{Bun}^\theta (TM,TM) \to \Map(M,M)$$
where the fibre is taken over the identity, and where $G_\theta (d)$ is $O(d)$ in the unoriented case, $SO(d)$ in the oriented case. By \cite[Lemma 2.21]{Sander-finiteness}, the base and the fibre of the above fibration are have finitely generated homotopy groups on each component (we regard the mapping space as a space of sections of the trivial bundle). Consequently, the fundamental group of the total space is also finitely generated. Furthermore, we have a fibration 
$$L_k \Emb^\theta (M,M) \to T_k \Emb^\theta (M,M) \to T_{k-1}\Emb^\theta (M,M)$$
where $L_k$ is the layer of the embedding tower, obtained as a space of relative sections of a fibration over $\Conf_k (M)$ relative the fat diagonal; its homotopy groups can be seen to be finitely generated, as follows again from \cite[Lemma 2.21]{Sander-finiteness}. By finite induction, we conclude that $\pi_1 T_k \Emb^\theta (M,M)$ is indeed finitely generated, and hence so is $\pi_1 T_k \Diff^\theta (M)$. The same holds for the other path components, as $T_k \Diff^{\theta} (M)$ is a grouplike $E_1$-space. Now, we observe that the Milnor exact sequence
$$1 \to \mathrm{lim}^1 _k \pi_1 T_k\Diff^\Theta (M) \to \pi_0 T_\infty \Diff^\theta (M) \to \lim_k \pi_0 T_k \Diff^\theta (M) \to 1 $$
which is a priori an exact sequence of pointed sets, is indeed an exact sequence of groups. Indeed, as $T_\infty \Diff^\theta (M)$ and $\lim_k T_k \Diff^\theta (M)$ are grouplike $E_1$-spaces, and the map $T_\infty \Diff^\theta (M) \to \lim_k T_k \Diff^\theta(M)$ is an $E_1$-map, the above exact sequence occurs as the Milnor exact sequence computing the fundamental groups of their deloopings $BT_\infty \Diff^\theta (M)$ and $BT_k \Diff^\theta (M)$, based at the identity; furthermore, the $\lim^1$ term now consists of abelian groups. By \cite[thm 2.3.3]{MayPonto}, $\lim_{k}\nolimits^1 \pi_1 T_k \Diff^\theta (M)$ is either trivial or uncountable. In the case it is trivial, then $\pi_0 T_{\infty} \Diff^\theta (M) \cong \lim_k \pi_0 T_k \Diff^\theta (M)$. By \cref{thm:main-result-resfinite}, the latter is again residually finite. Observe that $\Emb^\theta (W^{n}_g)=\Diff^\theta (M)$, as any self-embedding of a closed manifold is necessarily surjective. If the map $\Emb^\theta (W^{n}_g) \to T_\infty \Emb^{\theta} (W^{n}_g)$ was indeed an equivalence, it induces an equivalence $\Diff^\theta (M) \to T_\infty \Diff^\theta (M)$ which in turn induces an isomorphism of groups on $\pi_0$. In particular, this would imply that $\pi_0 \Diff^\theta (W^{n}_g)$ is residually finite, which is not the case as follows from \cite{Manuel-Oscar-arithmetic}.
\\

In the case $\lim_k \nolimits^1 \pi_1 T_k \Diff^\theta (M)$ is non-trivial, it follows from the Milnor exact sequence that $\pi_0 T_\infty \Diff^\theta (M)$ is uncountable, and hence cannot be isomorphic to $\pi_0 \Diff^\theta (M)$, a finitely generated group, as can be seen from \cite[thm 13.3]{Sullivan-infinitesimal}.
\end{proof}

We observe that in the above proof, $\pi_1 T_k \Diff^\theta (M)$ is shown to be a finitely generated group, and by \cite[thm 2.3.3]{MayPonto}, it follows that $\lim^1 \pi_1 T_k \Diff^\theta (M)$ is a divisible group. In particular, we obtain the following

\begin{corollary}
    In the above setting, $\lim^1 \pi_1 T_k \Diff^\theta (M)$ embeds, via the Milnor exact sequence, into the finite residual of $\pi_0 T_\infty \Diff^\theta (M)$.
\end{corollary}
In the case the above $\lim^1$ vanishes, one obtains an interesting application to the disc-structure space $\mathcal{S}^{\Disk}_\partial(D^{2n})$. We restrict ourselves to the case $n=5\pmod{8}$, and consider the following diagram
\begin{center}
    \begin{tikzcd}
        && \mathrm{bP}_{2n+2} \arrow[hook]{d} \arrow{ddr}{0} \arrow[dashed]{ddl}[swap]{\exists}
        \\
        && \pi_0 \Diff(W_g^n) \arrow{d}
        \\
       1 \arrow{r} & \lim^1 \pi_1 T_k \Diff(W_g^n) \arrow{r} & \pi_0 T_\infty \Diff(W_g^n) \arrow{r} & \lim_k \pi_0 T_k \Diff(W_g^n) \arrow{r} & 1
    \end{tikzcd}
\end{center}
We observe that the map $\mathrm{bP}_{2n+2} \to \lim_k \pi_0 T_k \Diff(M)$, obtained by restricting the embedding calculus map, is trivial. This follows from the fact that $\mathrm{bP}_{2n+2}$ is embedded as a subgroup of the finite residual of $\pi_0 \Diff(W_g^n)$. In case $\mathrm{bP}_{2n+2} \to \lim_k \pi_0 T_k \Diff(W_g^n)$ is non-trivial, and since $\lim_k \pi_0 T_k \Diff(W_g^n)$ is itself residually finite, then some non-trivial $\mathrm{bP}$-sphere will be detected by some map $\pi_0 \Diff(W_g^n)  \to F$ where $F$ is a finite group, contradicting the fact that $\mathrm{bP}_{2n+2}$ is indeed in the finite residual of $\pi_0 \Diff(W_g^n)$. Thus, we obtain a lift $\mathrm{bP}_{2n+2} \to \lim^1 \pi_1 T_k \Diff(W_g^n)$; if the latter is trivial, then $\mathrm{bP}_{2n+2} \leq \ker(\pi_0 \Diff(W_g^n) \to \pi_0 T_{\infty} \Diff(W_g^n)) = \mathrm{Im} (\pi_1 \mathcal{S}^{\Disk}(W_g^n) \to \pi_0 \Diff(W_g^n))$, i.e. every $\mathrm{bP}$-sphere can be hit by an element in $\pi_1 \mathcal{S}^{\Disk}(W_g^n)$, and in particular, $\mathcal{S}^{\Disk}(W_g^n) \simeq \mathcal{S}^{\Disk}_\partial (D^{2n})$ is not simply connected.

\bigskip

\section{Topological mapping class groups}
Combining the Weiss fibre sequence and convergence results for embedding calculus relative half the boundary, \cref{thm:main-result-resfinite} implies that $\pi_0 \Homeo(M)$ is residually finite, whenever $M$ is a smoothable closed manifold of dimension $d \geq 6$. This is the content of the current section. We begin by setting up the Wess fibre sequence, both in the smooth and the topological categories; embedding calculus allows us to infer some consequences on embeddings relative boundary conditions. We then use smoothing theory in order to describe the difference between the smooth and topological Weiss fibre sequences.

\subsection{The Weiss fibre sequence}\label{section:Weiss-fibre-sequence}

Let $M$ be a smooth, closed 2-connected manifold of dimension $d \geq 6$, and fix a smooth embedding $D^d \hookrightarrow M$. In what follows, we denote by $M^{\circ}$ as the manifold $M \setminus \Int(D)$, with boundary diffeomorphic to $S^{d-1}$. We decompose the boundary sphere into upper and lower hemispheres as
$$S^{d-1} = D^{d-1}_{+} \cup D^{d-1}_{-}$$
We fix a collar $D^{d-1}_+ \times I \hookrightarrow M^{\circ}$, and denote $P \defeq M^{\circ} \setminus (\mathrm{int}(D^{d-1}_+ \times I))$. We furtermore denote $D^{d-1}_{-}$ by $\partial/2$. For $\Cat \in \{O,\mathrm{Top}\}$, the parametrized isotopy extension theorem yields a fibre sequence
$$\Aut^{\Cat}_{\partial}(M^{\circ} \setminus \mathrm{int}(P)) \to \Aut^{\Cat}_{\partial}(M^\circ) \to \Emb_{\partial/2}^{\Cat,\cong}(P,M^{\circ})$$
where the base is the space of $\Cat$-embeddings of $P$ into $M$, relative $\partial/2$, which extend to a diffeomorphism of $M^\circ$, and where $\Aut^\Cat$ denotes the space of automorphisms in the corresponding category. We observe that $M^\circ \setminus \mathrm{int}(P) \cong D^{d-1}_{-} \times I \cong D^d$. Finally, another application of parametrized isotopy extension and contractibility of spaces of neighborhood collars implies that the map
$$\Emb_{\partial/2}^{\cong,\Cat}(M^\circ,M^\circ) \to \Emb_{\partial/2}^{\cong,\Cat}(P,M^\circ)$$
is an equivalence, where the target is the space of $\Cat$-embeddings, relative half the boundary, that extend to an automorphism of $M^\circ$. We thus obtain a fibre sequence of $E_1$-spaces
$$\Aut_\partial^{\Cat}(D^d) \to \Aut_\partial^\Cat(M^\circ) \to \Emb_{\partial/2}^{\cong,\Cat}(M^\circ, M^\circ)$$
which upon delooping, yields the \emph{Weiss fibre sequence}. It is perhaps noteworthy to mention that the Weiss fibre sequence can be delooped once, as in \cite{Sander-finiteness}; this will however not be necessary for our purposes. Forgetting smooth structures yields a map of fibre sequences
\begin{center}
    \begin{tikzcd}
        B\Diff_\partial (D^d) \arrow{r} \arrow{d} & B\Diff_\partial (M^{\circ}) \arrow{r}\arrow{d} & B\Emb_{\partial/2}^{\cong}(M^\circ,M^\circ) \arrow{d}{\Psi}
        \\
         \ast \simeq B\Homeo_\partial (D^d) \arrow{r} & B\Homeo_\partial (M^{\circ}) \arrow{r}{\simeq} & B\Emb_{\partial/2}^{\cong,\mathrm{Top}}(M^\circ,M^\circ)
    \end{tikzcd}
\end{center}
By the Alexander trick, we observe that the map $B\Homeo_\partial (M^{\circ}) \xrightarrow[]{\simeq} B\Emb_{\partial/2}^{\cong,\mathrm{Top}}(M^{\circ},M^\circ)$ is a homotopy equivalence, which is a crucial observation to deduce results on homeomorphism groups via embedding calculus. We next study the fibre of the map $\Psi$; this is the subject of parametrized smoothing theory, relative half the boundary.

\subsection{Parameterized smoothing theory rel. half boundary} Smoothing theory is the study of the space of smooth structures on a topological manifold. Given a smoothable manifold $M$ with empty boundary, smoothing theory describes the fibre (over the identity) of the map $B\Diff(M) \to B\Homeo(M)$ as a space of sections of a bundle over $M$ with fibres $\mathrm{Top}(d)/\mathrm{O}(d)$, and this space is amenable to computations using homotopical methods. The proof of the following can be extracted from the proof of \cite[Corollary 5.15]{Sander-finiteness}.

\begin{lemma}\label{lemma:smoothing-theory-rel-bdry}
    Let $W$ be a smooth compact manifold with boundary $\partial W \cong S^{d-1}$, and let $D^{d-1} \subset S^{d-1}$ denote the lower hemisphere. Let $\varphi:B\Emb_{\partial/2}^{\cong}(W) \to B\Emb_{\partial/2}^{\cong,\mathrm{Top}}(W)$ be the forgetful map. Then, the fibre over the identity is given by
    $$\fib_{\mathrm{id}}(\varphi) \simeq \Gamma_{\partial/2}(E \to W) $$
    where the right hand side is a space of sections of a bundle over $W$ with fibres equivalent to $\mathrm{Top}(d)/\mathrm{O}(d)$, relative a fixed section on the lower hemisphere.
\end{lemma}
\begin{proof}
    By \cite[Corollary 2]{Lashof}, the natural inclusion maps of $\Cat$-embeddings spaces into $\Cat$-immersion spaces, for $\Cat \in\{\mathrm{Top},O\}$, induce an equivalence on fibres
    $$\fib_{\mathrm{id}}\left(\Emb^{\cong}_{\partial/2}(M^\circ, M^\circ) \to \Emb^{\cong,\mathrm{Top}}_{\partial/2}(M^\circ, M^\circ)\right)  \simeq \fib_{\mathrm{id}}\left(\mathrm{Imm}^{\cong}_{\partial/2}(M^\circ,M^\circ) \to \mathrm{Imm}^{\cong,\mathrm{Top}}_{\partial/2}(M^\circ,M^\circ) \right)$$
    where the fibre is taken over the identity. By immersion theory, the fibre on the right hand side can be identified with the fibre of a map
    $$\Gamma_{\partial/2}(E^O \to M^\circ) \to \Gamma_{\partial/2}(E^{\mathrm{Top}} \to M^\circ)$$
    where $E^\Cat \to M^\circ$ is a bundle with fibres $\Cat(d)$, over some fixed section in the base. The result thus readily follows.
\end{proof}

As observed previously, one can show that the homotopy groups of these section spaces are finitely generated, at all choices of basepoints. However, one further thing can be said about its set of connected components and the fundamental group over each of them: they are indeed finite. First, we recall that by \cite[p.246]{Kirby-Siebenmann}, we have the following isomorphisms
\begin{equation*}
    \pi_k \left(\mathrm{Top}/\mathrm{O}\right) \cong \begin{cases}
    \mathbb{Z}/2\mathbb{Z} & k=3
    \\
    \Theta_k & \text{else}
    \end{cases}
\end{equation*}
where $\Theta_k$ are the groups of exotic spheres. We also note that the map
$$\mathrm{Top}(d)/\mathrm{O}(d) \to \mathrm{Top}/\mathrm{O}$$
is $(d+2)$-connected, for $d \geq 5$, again by \cite[p.246]{Kirby-Siebenmann}. In particular, $\pi_k \mathrm{Top}(d)/\mathrm{O}(d)$ is finite for $k \leq d+1$. By obstruction theory, we have the following.
\begin{proposition}\label{prop:section-space-finite-pi1}
    Let $W$ be a compact $d$-manifold with boundary $\partial W \cong S^{d-1}$, and let $\partial/2 \defeq D^{d-1}_{-} \subset S^{d-1}$ be the lower hemisphere. Then, the space $\Gamma_{\partial/2}(E \to W)$ has finitely many path components, and for any choice of basepoint, its fundamental group is finite.
\end{proposition}
\begin{proof}
    The claim regarding path components follows readily from obstruction theory. For $\pi_1$, we fix a section $\mathcal{s} \in \Gamma_{\partial/2}(E \to W)$ as a basepoint. We begin by observing that $(W,\partial/2)$ is equivalent, as a pair, to a finite, $d$-dimensional CW-pair, i.e. a finite relative CW-complex where all cells have dimension $\leq d$. Given $(X,\partial/2)$ a skeleton of $(W,\partial/2)$, attaching an $n$-cell yields the following fibre sequence by restriction
    $$\Gamma_{\partial D^n}(E|D^n \to D^n) \to \Gamma_{\partial/2}(E|X\cup D^n \to X \cup D^n) \xrightarrow[]{res} \Gamma_{\partial/2}(E|X \to X)$$
    where the fibre is taken over the restriction of the section $\mathcal{s}$ to $ \Gamma_{\partial/2}(E|X \to X)$. The fibre of the above fibre sequence can be identified with $\Omega^n (\mathrm{Top}(d)/O(d))$, where $n \leq d$; thus, its fundamental group is again finite, on all possible choices of basepoints. The statement thus follows from finite induction, as on fundamental groups, the above fibre sequence exhibits $\pi_1 \Gamma_{\partial/2}(E|X\cup D^n \to X \cup D^n)$ as an extension of a finite group by a finite group.
\end{proof}

\subsection{Embedding calculus rel. half boundary}
We set up embedding calculus to study embedding spaces of manifolds, fixing boundary conditions. This section follows closely the account of \cite[§3.3.2]{Sander-finiteness}.
\\

Fix $P$, a closed $(d-1)$-dimensional smooth manifold. We consider the topological category $\Manf_{d,P}$, consisting of smooth $d$-dimensional manifolds (not necessarily compact) with boundary diffeomorphic to $P$, and where morphisms are taken to be embeddings preserving the boundary identifications; we also view $\Manf_{d,P}$ as an $\infty$-category via coherent nerve. Inside $\Manf_{d,P}$, we can consider the collection of full subcategories $\Disk_{d,P}^{\leq k}$, for $k \in \mathbb{N} \cup \{\infty\}$, on objects of the form $(P \times [0,\infty)) \sqcup (\sqcup_{\ell}\mathbb{R}^d)$, for $\ell \leq k$. Yoneda followed by restriction defines a tower of functors
$$\Manf_{d,P} \to \Psh(\Disk_{d,P}^{\leq k})$$
which, on morphism spaces, yields the \emph{rel $\partial$ embedding tower}, $\Emb_{\partial}(-,-) \to T_k \Emb_{\partial}(-,-)$. The same convergence theorem relative boundary also holds with the same codimension estimates, \cite[Chapter 5]{Goodwillie-Weiss} and \cite[§3.3.2]{Sander-finiteness}.
\\

The above setting is in fact well-suited for our purposes. Let $M$ be a closed, smooth $d$-dimensional manifold, and as above, denote $M^\circ$ to be the manifold obtained from $M$ by removing the interior of some embedded disc. We note that $\partial M^\circ \cong S^{d-1}$, which decomposes into $S^{d-1} = D_{-}^{d-1} \cup D_{+}^{d-1}$, where we let $\partial/2$ stand for the lower hemisphere. The space we are interested in is precisely $\Emb_{\partial/2}(M^\circ , M^\circ)$ of those self embeddings that are the identity on $\partial/2$. We let $M^\ast \coloneqq M^\circ \setminus D^{d-1}_+$.

\begin{lemma}
    In the above setting, $M^\circ$ is isotopy equivalent to $M^\ast$, relative $\partial/2$. We consequently obtain an equivalence $$\Emb_{\partial/2}(M^\circ,M^\circ) \simeq \Emb_{\partial}(M^\ast, M^\ast)$$
\end{lemma}

We observe that with the above equivalence, we have a method of attacking the embedding space by means of embedding calculus relative to the full boundary of $M^\ast$. We consequently obtain a tower $\{T_k \Emb_{\partial/2}(M^\circ,M^\circ)\}_{k \in \mathbb{N}}$ receiving maps from $\Emb_{\partial/2}(M^\circ,M^\circ)$.

\begin{proposition}\label{prop:emb-calc-convergence-rel-bdry}
    If $M$ is a 2-connected closed manifold, then the embedding calculus tower for $\Emb_{\partial/2}(M^\circ,M^\circ)$ converges.
\end{proposition}
This is proved as follows: using a Morse theoretic argument, one can show that $M$ admits a handle decomposition of dimension $\leq d-3$, relative $\partial/2$ \cite[Lemma 3.14]{Sander-finiteness}. Convergence of embedding calculus for embeddings relative half the boundary can be deduced from \cite[thm 6.3]{Manuel-Sander}, and the increasing connectivity of the layers of the tower can be deduced from Remark 6.11 (ii), \textit{loc. cit.}
\\

With the above observation in mind, we work towards showing \cref{prop:resfinite-boundary} below.
\begin{proposition}\label{prop:resfinite-boundary}
    For $M$ a closed, 2-connected smooth manifold, and $k \in \mathbb{N}$, the group
    $$\pi_0 T_k \Emb^{\cong}_{\partial/2}(M^\circ, M^\circ)$$
    is residually finite.
\end{proposition}

\begin{proof}
The following proof follows from the proofs of \cref{thm:injective-pi_0} and \cref{thm:main-result-resfinite} above, with the modifications that we present here. We recall that $T_k \Emb_{\partial/2}(M^{\circ},M^{\circ})$ was defined as the space $T_k \Emb_{\partial}(M^\ast, M^\ast)$, where $M^\ast$ is the manifold $M \setminus D^{d-1}_+$ with boundary $S^{d-1} \setminus D^{d-1}_+$, i.e. the endomorphism space in the category $\Psh(\Disk_{d,\partial}^{\leq k})$ of the object $\iota_k^\ast E_{M^\ast}$. Following the proof of \cref{thm:main-result-resfinite}, it suffices to consider residual finiteness in the case where we remember spin structures. In that setting, we fix once and for all a spin structure on a collar of $S^{d-1} \setminus D^{d-1}_+$, and endow $M^\ast$ a spin structure extending the fixed spin structure. We now claim that pointwise, $\iota_k^\ast E_{M^\ast}^{\Spin}$ is a 1-connected finite space. To see this, we recall that objects of $\Disk_{d,\partial}^{\leq k}$ are of the form $(\partial \times [0,\infty)) \sqcup (\sqcup_{\ell}\mathbb{R}^d)$, where $\ell \leq k$. The restriction map
$$\Emb_{\partial}^{\Spin}(\partial\times[0,\infty) \sqcup (\sqcup_{\ell}\mathbb{R}^d),M^\ast) \to \Emb^{\Spin}(\sqcup_{\ell}\mathbb{R}^d, \Int(M^\ast))$$
is an equivalence, by the contractibility of the space of boundary collars, and the space on the right hand side is a simply connected, finite space as follows from \cref{lemma:pointwisefinite}. Furthermore, we consider the following diagram
\begin{equation}\label{diagram:boundary}
    \begin{tikzcd}
        \Psh(\Disk_{d,\partial}^{\Spin,\leq k}) \arrow{d}[swap]{\iota_{k-1}^\ast} \arrow{r}{\circled{1}} & \Fun([2],\Psh(\Disk_{d,\partial}^{\Spin,=k})) \arrow{d}{\circled{3}}
        \\
        \Psh(\Disk_{d,\partial}^{\Spin,\leq k-1}) \arrow{r}[swap]{\circled{2}} & \Fun([1],\Psh(\Disk_{d,\partial}^{\Spin,=k}))
    \end{tikzcd}
\end{equation}
where
\begin{itemize}
    \item $\iota_{k-1}^\ast$ is induced by the inclusion $\iota_{k-1}\colon \Disk_{d,\partial}^{\Spin,\leq k-1} \hookrightarrow \Disk_{d,\partial}^{\leq k}$;
    \item The category $\Disk_{d,\partial}^{\Spin,=k}$ consits of a single object $D_k^{\partial} \defeq (\partial \times [0,\infty)) \sqcup(\sqcup_k \mathbb{R}^d)$, and where the space of morphisms is $\Emb_\partial^{\Spin,\pi_0\text{-inj}}(D_k^\partial,D_k^\partial)$, which is equivalent to the group $\mathfrak{S}_k \ltimes \Spin(d)^k$.
    \item The functor $\circled{1}$ sends $X \in \Psh(\Disk_{d,\partial}^{\Spin,\leq k})$ to the commutative triangle
    \begin{center}
        \begin{tikzcd}
            \mathcal{j}^\ast(\iota_{k-1})_! \iota_{k-1}^\ast X(D_k^\partial) \arrow{r} \arrow{dr}  &   \mathcal{j}^\ast X(D_k^\partial) \arrow{d}
            \\
            & \mathcal{j}^\ast (\iota_{k-1})_\ast \iota_{k-1}^\ast X(D_k^\partial)
        \end{tikzcd}
    \end{center}
    and where $\mathcal{j} \colon \Disk_{d,\partial}^{\Spin=k} \to \Disk_{d,\partial}^{\Spin,\leq k}$ is the obvious inclusion functor.
    \item The functor $\circled{2}$ sends $X \in \Psh(\Disk_{d,\partial}^{\Spin,\leq k-1})$ to the arrow 
    $$\mathcal{j}^\ast (\iota_{k-1})_! X(D_k^\partial) \to \mathcal{j}^\ast (\iota_{k-1})_\ast X(D_k^\partial)$$
    \item The funtor $\circled{3}$ is induced by the map $[1] \to [2]$, $0 \mapsto 0$ and $1 \mapsto 2$.
\end{itemize}
By \cite[thm 4.15]{Manuel-Sander}, we conclude that diagram (\ref{diagram:boundary}) is a pullback square. Furthermore, for $n \in \mathbb{N}$, we consider the restricted Postnikov truncation of $X \in \Psh(\Disk_{d,\partial}^{\Spin,\leq k})$, as constructed in \cref{Construction:restricted-Postnikov}, namely as the composite
$$\widetilde{\tau_{\leq n}}X \colon \left(\Disk_{d,\partial}^{\Spin,\leq k}\right)^{\op} \xrightarrow[]{X/T_{k-1}X} \mathrm{Ar}(\mathcal{S}) \xrightarrow[]{\mathrm{MP}_n} \mathcal{S}$$
where the first functor sends $D$ to the arrow $X(D) \to (\iota_{k-1})_\ast \iota_{k-1}^\ast X(D)$, and where the second functor is the $n$-th stage Moore-Postnikov truncation. As per the proof of \cref{thm:main-result-resfinite}, residual finiteness of the above group follows from the $\pi_0$-injectivity of the map
$$\Map_{\Psh(\Disk_{d,\partial}^{\Spin, \leq k})}(E_{M^\ast},E_{M^\ast}) \to \Map_{\Psh(\Disk_{d,\partial}^{\Spin, \leq k})}(E_{M^\ast},\Phi^s \circ E_{M^\ast})$$
induced by composition with the map $\eta \colon E_{M^\ast} \to \Phi^s \circ E_{M^\ast}$. By diagram (\ref{diagram:boundary}), we have an equivalence
$$\mathrm{Lift}_{\Psh(\Disk_{d,\partial}^{\Spin,\leq k})}\left(\begin{tikzcd}
    & \widetilde{\tau_{n+1}}E_{M^\ast} \arrow{d}
    \\
    E_{M^\ast} \arrow{r} \arrow[dashed]{ur} & \widetilde{\tau_{\leq n}}E_{M^\ast}
\end{tikzcd}\right)  \xrightarrow[]{\simeq} \Map^{G_k}\left(\begin{tikzcd}
    \mathcal{j}^\ast (\iota_{k-1})_! \iota_{k-1}^\ast E_{M^\ast} \arrow{d} \arrow[dashed]{r} & \mathcal{j}^\ast \widetilde{\tau_{\leq n+1}}E_{M^\ast} \arrow{d}
    \\
    \mathcal{j}^\ast E_{M^\ast} \arrow[dashed]{r} & \mathcal{j}^\ast \widehat{\tau_{\leq n}}E_{M^\ast}
    \end{tikzcd} \right)$$
    where $G_k \defeq \mathfrak{S}_k \ltimes \Spin(d)^k$, and where the mapping space on the right hand side is the mapping space in the arrow category $\Fun(\Delta^1,\Psh(BG_k))$ of spaces with an action of $G_k$. We observe that the pair $(\mathcal{j}^\ast E_{M^\ast}, \mathcal{j}^\ast (\iota_{k-1})_! \iota_{k-1}^\ast E_{M^\ast})$ is a finite, free $G_k$-CW pair, as can be seen from \cite[Proposition 5.15]{Manuel-Sander}. The remainder of the obstruction theoretic proof of \cref{thm:injective-pi_0} carries through. 

\end{proof}

As a consequence of \cref{prop:resfinite-boundary}, we obtain the following residual finiteness result

\begin{proposition}\label{prop:embedding-rel-bdry-resfinite}
    The group $\pi_0 \Emb_{\partial/2}^{\cong}(M^{\circ},M^\circ)$ is residually finite.
\end{proposition}
\begin{proof}
    By \cref{prop:emb-calc-convergence-rel-bdry}, the embedding calculus map is an equivalence, and in particular induces a group isomorphism
    $$\pi_0 \Emb^{\cong}_{\partial/2}(M^{\circ},M^{\circ}) \xrightarrow[]{\cong} \pi_0 T_{\infty}\Emb^{\cong}_{\partial/2}(M^{\circ},M^{\circ})$$
    We identify $T_\infty \Emb^{\cong}_{\partial/2}(M^\circ, M^\circ)$ with $\lim_{k\in\mathbb{N}}T_{k} \Emb^{\cong}_{\partial/2}(M^\circ, M^\circ)$. By the Milnor exact sequence, we obtain a short exact sequence of groups
    $$1 \to \mathrm{lim}_{k}^1 \pi_1 T_k \Emb^{\cong}_{\partial/2}(M^\circ, M^\circ) \to \pi_0 T_{\infty} \Emb^{\cong}_{\partial/2}(M^\circ, M^\circ) \to \lim_{k\in\mathbb{N}}\pi_0 T_k  \Emb^{\cong}_{\partial/2}(M^\circ, M^\circ) \to 1$$
    By the proof of the convergence of the embedding tower, we know that the connectivity of the fibres $L_k$ over the identity of the maps strictly increases with $k$ (see for instance, \cite[Remark 6.11 (ii)]{Manuel-Sander}). Consequently, the system $\{\pi_1 T_k \Emb^{\cong}(M^\circ,M^\circ)\}_{k \in \mathbb{N}}$ eventually consists of isomorphisms, so that in particular it satisfies the Mittag-Leffler condition. As a consequence, $\lim^1$ vanishes, and the above short exact sequence yields an isomorphism
    $$\pi_0 T_\infty  \Emb^{\cong}_{\partial/2}(M^\circ, M^\circ) \xrightarrow[]{\cong} \lim_{k \in \mathbb{N}}T_k  \Emb^{\cong}_{\partial/2}(M^\circ, M^\circ)$$
    As the class of residually finite groups is closed under inverse limits, the claim follows.
\end{proof}
\begin{remark}
    The above proof generalizes to the situation where we consider a smooth, compact 2-connected manifold $W$ of dimension $d \geq 5$, where $\partial W \neq \varnothing$, with the following technical condition. Let $\partial_0 W \subset \partial W$ be a compact, codimension 0 submanifold of the boundary, such that $\partial_1 W \hookrightarrow W$ is 2-connected, where $\partial_1 W \defeq \partial W \setminus \partial_0 W$. In that case, the same proof of \cref{prop:embedding-rel-bdry-resfinite} implies that $\pi_0 T_k \Emb_{\partial_0}^{\cong}(W,W)$ is residually finite, and \cite[Remark 6.11 (ii)]{Manuel-Sander} yields convergence of embedding calculus rel $\partial_1$. In particular, we conclude that $\pi_0 \Emb_{\partial_0}^{\cong}(W,W)$ is residually finite. This is developped in more details in the proof of \cref{thm:relbdryresfinite} below. 
\end{remark}

\subsection{Residual finiteness of topological mapping class groups}
We now combine the three previous sections to show that topological mapping class groups are indeed residually finite.

\begin{theorem}
    
    For any 2-connected smoothable closed topological manifold $M$ of dimension $d \geq 6$, the group $\pi_0 \Homeo_\partial (M^{\circ}) \cong \pi_1 B\Emb_{\partial/2}^{\cong,\mathrm{Top}}(M^{\circ})$ is residually finite.
\end{theorem}

\begin{proof}
    We consider the smooth and topological Weiss fibre sequences, and the forgetful map between them
    \begin{center}
        \begin{tikzcd}
            & & \fib(\Psi) \simeq \Gamma_{\partial/2}(\xi_M \to M) \arrow{d}{\iota}
            \\
            B\Diff_{\partial}(D^d) \arrow{r} \arrow{d} & B\Diff_{\partial}(M^{\circ}) \arrow{d} \arrow{r} & B\Emb^{\cong}_{\partial/2}(M^{\circ},M^{\circ}) \arrow{d}{\Psi}
            \\
            \ast \simeq B\Homeo_{\partial}(D^d) \arrow{r} & B\Homeo_{\partial}(M^\circ) \arrow{r}[swap]{\simeq} & B\Emb^{\cong,\mathrm{Top}}_{\partial/2}(M^\circ,M^\circ)
        \end{tikzcd}
    \end{center}
    By \cref{lemma:smoothing-theory-rel-bdry} and \cref{prop:section-space-finite-pi1}, the fibre of the map $\Psi$ has a finite set of components, and on each of them, $\pi_1$ is a finite group. The long exact sequence on homotopy groups thus yields an exact sequence
    $$\pi_1 \Gamma_{\partial}(\xi_M) \xrightarrow[]{\iota_\ast} \pi_1 B\Emb^{\cong}_{\partial/2}(M^\circ , M^\circ) \xrightarrow[]{\Psi_\ast} \pi_1 B\Emb^{\cong,\mathrm{Top}}_{\partial/2}(M^\circ,M^\circ) \xrightarrow[]{\partial} \pi_0 \Gamma_{\partial/2}(\xi_M)$$
    where exactness on $\partial$ is exactness as pointed sets. We note that $\pi_1 B\Emb^{\cong}_{\partial/2}(M^\circ,M^\circ)$ is residually finite by \cref{prop:embedding-rel-bdry-resfinite}. Furthermore, $\mathrm{Im}(\iota_\ast) = \ker \Psi_\ast$ is a finite, normal subgroup of $\pi_1 B\Emb^{\cong}_{\partial/2}(M^\circ,M^\circ)$. Thus, by \cref{lemma:quotient-resfinite}, it follows that
    $$\faktor{\pi_1 B\Emb^{\cong}_{\partial/2}(M^\circ,M^\circ)}{\mathrm{Im}(\iota_\ast)}$$
    is residually finite. $\Psi_\ast$ embedds this group into a finite index subgroup of $\pi_1 B\Emb^{\cong,\mathrm{Top}}_{\partial/2}(M^\circ,M^\circ)$. Thus, by \cref{lemma:finite-index-resfin}, it follows that $\pi_1 B \Emb_{\partial/2}^{\cong,\mathrm{Top}} (M^\circ,M^\circ)$ is residually finite. The Alexander trick yields an equivalence 
    $$B\Homeo_{\partial}(M^\circ) \xrightarrow[]{\simeq} B\Emb^{\cong,\mathrm{Top}}_{\partial/2}(M^\circ,M^\circ)$$
    and thus identifies their fundamental groups, implying the desired result.
\end{proof}

In order to now study $\pi_0 \Homeo(M)$ from $\pi_0 \Homeo_{\partial}(M^\circ)$, we use isotopy extension to obtain the fibre sequence
$$\Homeo_\partial (M^{\circ}) \to \Homeo(M)  \to \Emb^{\mathrm{Top}}(D^d,M)$$
We note that $\Emb^{\mathrm{Top}}(D^d,M) \simeq \Fr^{\mathrm{Top}}(M)$; consequently, $\pi_0 \Emb^{\mathrm{Top}}(D^d,M) \cong \mathbb{Z}/2\mathbb{Z}$, and on each basepoint, $\pi_1 \Emb(D^d,M) \cong \mathbb{Z}/2\mathbb{Z}$. Running the long exact sequence on homotopy groups, we obtain an exact sequence
$$\mathbb{Z}/2\mathbb{Z} \to \pi_0 \Homeo_{\partial}(M^{\circ}) \to \pi_0 \Homeo(M) \to \mathbb{Z}/2\mathbb{Z}$$

As a corollary of \cref{lemma:finite-index-resfin,lemma:quotient-resfinite}, we obtain the following

\begin{theorem}\label{thm:homeo-res-finite}
    For $M$ a 2-connected closed, smooth manifold, the group $\pi_0 \Homeo(M)$ is residually finite.
\end{theorem}
\begin{remark}
    It readily follows that the same holds for $\pi_0 \Homeo^+(M)$.
\end{remark}

\subsection{Arithmeticity of topological mapping class groups}\label{section:arithmeticity} We follow \cite[§2]{Sander-boundary} for all conventions regarding arithmetic groups. We consider two equivalence relations on the class of groups. The first is generated by
\begin{enumerate}[(i)]
    \item isomorphisms;
    \item passing to finite index subgroups;
    \item taking quotients by finite, normal subgroups
\end{enumerate}
Two groups in the same equivalence class of the above relation are said to be \emph{commensurable up to finite kernel}. Dropping condition (iii), the equivalence relation thus obtained is the classical notion of \emph{commensurability}.  In \cite[thm 13.3]{Sullivan-infinitesimal}, Sullivan shows that smooth mapping class groups of closed, simply connected smooth manifolds of dimension $d\geq 5$ are commensurable up to finite kernel to an arithmetic group.

\begin{theorem}[Sullivan]
    For a closed, smooth simply connected manifold $M$ of dimension $d \geq 5$, the group $\pi_0 \Diff(M)$ is commensurable up to finite kernel to an arithmetic group.
\end{theorem}

To go from the smooth category to the topological, we use smoothing theory as follows

\begin{lemma}
    For a smoothable, closed topological manifold $M$ of dimension $d \geq 5$, $\pi_0 \Homeo(M)$ and $\pi_0 \Diff(M)$ are commensurable up to finite kernel.
\end{lemma}
\begin{proof}
    Using Kirby-Siebenmann's bundle theorem \cite[Essay V, §3]{Kirby-Siebenmann}, we have a fibre sequence
    $$F \to B\Diff(M) \to B\Homeo(M)$$
    where the fibre $F$ is given as a collection of components of the space $\Gamma(\xi_M)$ of sections of a bundle over $M$ with fibre $\mathrm{Top}(d)/O(d)$. By obstruction theory, and using the fact that $\pi_k \left(\mathrm{Top}(d)/O(d)\right)$ is a finite group for $k \leq d+1$, it follows that $\pi_0 \Gamma(\xi_M)$ is finite, and for all basepoints, $\pi_1 \Gamma(\xi_M)$ is a finite group. Let $\Homeo^{\mathrm{sm}}(M)$ be the subgroup of $\Homeo(M)$ of those homeomorphisms isotopic to a diffeomorphism. Using the long exact sequence on homotopy groups associated to the above fibre sequence, we have a short exact sequence
    $$1 \to \ker p \to \pi_0 \Diff(M) \xrightarrow{p} \pi_0 \Homeo^{\mathrm{sm}}(M) \to 1$$
    where $\ker p \leq \pi_1 F$ is a finite normal subgroup of $\pi_0 \Diff(M)$. Thus,
    $$\pi_0 \Homeo^{\mathrm{sm}}(M) \cong \faktor{\pi_0 \Diff(M)}{\ker p}$$
    and consequently, $\pi_0 \Homeo^{\mathrm{sm}}(M)$ is commensurable up to finite kernel to $\pi_0 \Diff(M)$. Again using the same long exact sequence as above, we notice that the quotient
    $$\faktor{\pi_0 \Homeo(M)}{\pi_0 \Homeo^{\mathrm{sm}}(M)} \subset \pi_0 \Gamma(\xi_M)$$
    is a finite set; thus, $\pi_0 \Homeo^{\mathrm{sm}}(M)$ is a finite index subgroup of $\pi_0 \Homeo(M)$. We thus conclude that $\pi_0 \Diff(M)$ and $\pi_0 \Homeo(M)$ are commensurable up to finite kernel.
\end{proof}

As a corollary, we obtain

\begin{corollary}\label{cor:homeo-commensurable-arithmetic}
    Let $M$ be a smoothable closed manifold of dimension $d \geq 5$. Then, $\pi_0 \Homeo(M)$ is commensurable up to finite kernel to an arithmetic group.
\end{corollary}

However, as observed in \cite[Remark, p.470]{Manuel-Oscar-arithmetic}, commensurability and commensurability up to finite kernel agree among residually finite groups. Indeed, if $G$ is a residually finite group, and $K$ is a finite, normal subgroup, then $G$ and $G/K$ are commensurable. To see this, we observe that for any $g \neq 1$ in $G$, there exists a finite group $F_g$ and a group morphism $\varphi_g \colon G \to F_g$ such that $\varphi_g(g) \neq 1 \in F_g$. In particular, for a finite normal subgroup $K \leq G$, we can find a finite group $F$ and a group morphism $\varphi\colon G \to F$ such that the product map $p \times \varphi \colon G \to (G/K) \times F$ is injective. The image of the above group morphism is isomorphic to $G$, and is a finite index subgroup of $(G/K) \times F$. Finally, $G/K$ is evidently a finite index subgroup of $(G/K) \times F$, and as a consequence, $G$ is commensurable to $G/K$. Combining \cref{thm:homeo-res-finite} and \cref{cor:homeo-commensurable-arithmetic}, we obtain the following

\begin{theorem}\label{thm:homeo-arithmetic}
    For $M$ a smoothable, 2-connected closed topological manifold of dimension $d \geq 6$, $\pi_0 \Homeo(M)$ is an arithmetic group.
\end{theorem}

\begin{remark}
    Similar considerations also show that $\pi_0 \Homeo^+ (M)$ is arithmetic.
\end{remark}

\subsection{Compact manifolds with boundary}
All of our results have so far focused on the case of closed manifolds: the strategy was to delete the interior of a disc inside a closed manifold $M$, and study embedding calculus relative boundary for this manifold. We study here generalizations to the case of compact manifolds with non-empty boundaries. Let $W$ be a compact, $d$-dimensional manifold of dimension $d \geq 5$, such that $\partial W \neq \varnothing$. We furthermore assume that $W$ is 2-connected. The aim of this remark is to show the following
\begin{theorem}\label{thm:relbdryresfinite}
    Let $W$ be a smoothable, 2-connected compact manifold of dimension $d \geq 5$, with $\partial W \neq \varnothing$. Then, $\pi_0 \Homeo_\partial(W)$ is residually finite.
\end{theorem}
\begin{proof}
    We fix a smooth embedding $D^{d-1} \hookrightarrow \partial W$, and denote $\partial_1 W \defeq \partial W \setminus D^{d-1}$. The aim is to show that $\pi_0 \Emb_{\partial_1}^{\cong}(W,W)$, the group of path components of the space of embeddings fixing $\partial_1$ that are isotopic to diffeormophisms, is residually finite. Indeed, the smooth and topological Weiss fibre sequences yields a commutative diagram
    \begin{center}
        \begin{tikzcd}
            B\Diff_\partial (D^d) \arrow{r}\arrow{d} & B\Diff_{\partial}(W) \arrow{r}\arrow{d} & B\Emb_{\partial_1}^{\cong}(W,W) \arrow{d}
            \\
            B\Homeo_\partial (D^d)\simeq\ast \arrow{r} & B\Homeo_{\partial}(W) \arrow{r}[swap]{\simeq} & B\Emb_{\partial_1}^{\cong,\mathrm{Top}}(W,W)
        \end{tikzcd}
    \end{center}
    The fibre of the rightmost vertical map is understood by smoothing theory, and has finite $\pi_0$, and finite fundamental groups over each component (\cref{lemma:smoothing-theory-rel-bdry,prop:section-space-finite-pi1}). Thus, residual finiteness of $\pi_0 \Emb_{\partial_1}^{\cong}(W,W)$ implies the residual finiteness of $\pi_0 \Emb_{\partial_1}^{\cong,\mathrm{Top}}(W,W) \cong \pi_0 \Homeo_\partial (W)$, by \cref{lemma:quotient-resfinite,lemma:finite-index-resfin}.
    \\
    
    We now focus on showing that $\pi_0 \Emb_{\partial_1}^{\cong}(W,W)$ is residually finite. This follows the same embedding calculus strategy, and we first consider the case where we remember spin structures. Fix a $\Spin(d)$ structure on a collar of $\partial_1$, and consider embedding calculus rel. boundary $\Psh(\Disk_{d,\partial_1}^{\Spin,\leq k})$ where $\Disk_{d,\partial_1}^{\Spin,\leq k}$ is the category consisiting of objects of the form $\partial_1 \times [0,\infty) \times \sqcup_{\ell}\mathbb{R}^d$, for $\ell \leq k$, along with choices of spin structures extending the fixed spin structure on $\partial_1 \times [0,\infty)$, and embeddings rel. $\partial_1$ respecting the spin structures. By \cite[thm 4.20]{Manuel-Sander}, we obtain a pullback square of $\infty$-categories
    \begin{equation}\label{diagram:boundary2}
    \begin{tikzcd}
        \Psh(\Disk_{d,\partial_1}^{\Spin,\leq k}) \arrow{d}[swap]{\iota_{k-1}^\ast} \arrow{r}{\circled{1}} & \Fun([2],\Psh(\Disk_{d,\partial_1}^{\Spin,=k})) \arrow{d}{\circled{3}}
        \\
        \Psh(\Disk_{d,\partial_1}^{\Spin,\leq k-1}) \arrow{r}[swap]{\circled{2}} & \Fun([1],\Psh(\Disk_{d,\partial_1}^{\Spin,=k}))
    \end{tikzcd}
\end{equation}
where $\circled{1},\circled{2},\circled{3}$ are as in the proof of \cref{prop:resfinite-boundary}. We denote $T_k\Emb_{\partial_1}^{\Spin,\cong}(W,W)$ as the space of automorphisms of $\iota_k^\ast E_{W^\ast}$ in the category $\Psh(\Disk_{d,\partial_1}^{\Spin,\leq k})$, where $W^\ast \defeq W \setminus D^{d-1}$. Pointwise, this presheaf has the homotopy type of a simply connected, finite CW-complex (\cref{lemma:pointwisefinite}). In order to run the same obstruction theoretic proof as in \cref{thm:injective-pi_0}, we observe that diagram (\ref{diagram:boundary2}) yields an equivalence
$$\mathrm{Lift}_{\Psh(\Disk_{d,\partial}^{\Spin,\leq k})}\left(\begin{tikzcd}
    & \widetilde{\tau_{n+1}}E_{W^\ast} \arrow{d}
    \\
    E_{W^\ast} \arrow{r} \arrow[dashed]{ur} & \widetilde{\tau_{\leq n}}E_{W^\ast}
\end{tikzcd}\right)  \xrightarrow[]{\simeq} \Map^{G_k}\left(\begin{tikzcd}
    \mathcal{j}^\ast (\iota_{k-1})_! \iota_{k-1}^\ast E_{W^\ast} \arrow{d} \arrow[dashed]{r} & \mathcal{j}^\ast \widetilde{\tau_{\leq n+1}}E_{W^\ast} \arrow{d}
    \\
    \mathcal{j}^\ast E_{W^\ast} \arrow[dashed]{r} & \mathcal{j}^\ast \widehat{\tau_{\leq n}}E_{W^\ast}
    \end{tikzcd} \right)$$
$G_k$ in the above denotes the group $\mathfrak{S}_k \ltimes \Spin(d)$, and the space on the right hand side is the space of morphism in the arrow category $\Fun(\Delta^1,\Psh(BG_k))$ of spaces with an action of $G_k$, where $\mathcal{j}^\ast$ is the restriction functor along the inclusion $BG_k \to \Disk_{d,\partial_1}^{\Spin,\leq k}$. The pair $(\mathcal{j}^\ast E_{W^\ast},\mathcal{j}^\ast (\iota_{k-1})E_{W^\ast})$ is equivalent to a finite, $G_k$-CW pair, as follows from \cite[Proposition 5.15]{Manuel-Sander}. The remainder of the obstruction theoretic proof of \cref{thm:injective-pi_0} follows, and implies that $\pi_0 T_k\Emb_{\partial_1}^{\Spin,\cong}(W^\ast,W^\ast)$ is residually finite, for all $k \in \mathbb{N}$. Convergence of embedding calculus, along with increasing connectivity of the layer (\cite[Remark 6.11(ii)]{Manuel-Sander}) implies the same for $\pi_0 \Emb_{\partial_1}^{\Spin}(W^\ast,W^\ast)$. The proof of \cref{thm:main-result-resfinite} allows us to forget the spin structures, and implies that $\pi_0 \Emb_{\partial_1}^{\cong}(W^\ast,W^\ast)$ is residually finite, as required.
\end{proof}

\subsection{Finite residuals of smooth mapping class groups}
As mentioned previously, the smooth mapping class group $\pi_0 \Diff(M)$ of a smooth manifold $M$ may fail to be residually finite, as was shown in \cite{Manuel-Oscar-arithmetic}. In the situation of the counterexample in \emph{loc. cit.}, we fix an embedding $D^{2n} \hookrightarrow W_g^n \defeq \#_g (S^n \times S^n)$, where $n = 5\pmod{8}$, for $g \geq 5$. Extension by the identity yields a morphism
$$\pi_0 \Diff_\partial (D^{2n}) \to \pi_0 \Diff(W_g^n)$$
which embeds the subgroup $\mathrm{bP}_{2n+2} \leq \Theta_{2n+1}\cong\pi_0 \Diff_\partial(D^{2n})$ (non-trivial for these values of $n$ by \cite{Kervaire-Milnor}) to the finite residual of $\pi_0 \Diff(W_g^n)$. The following two results imply that this is the only phenomenon preventing residual finiteness. We split it into two statements, one regarding smooth mapping class groups relative boundary, and one for smooth mapping class groups of closed manifolds.
\begin{lemma}\label{lemma:fin-res-smooth}
    Let $M$ be a closed, 2-connected smooth manifold of dimension $d \geq 6$, and denote $M^{\circ}\defeq M \setminus\Int(D^d)$, for some disc $D^d \subset M$. Fix another embedding $D^d \hookrightarrow \Int(M^{\circ})$. Then, $\mathrm{fr}(\pi_0 \Diff_{\partial}(M^{\circ})) \subseteq \mathrm{Im}(\pi_0 \Diff_\partial(D^d) \to \pi_0 \Diff_{\partial}(M^{\circ}))$, where $\mathrm{fr}(\pi_0 \Diff_{\partial}(M^{\circ}))$ is the finite residual, and where the map
    $\pi_0 \Diff_{\partial}(D^d) \to \pi_0 \Diff_{\partial}(M^{\circ})$ is extension by identity.
\end{lemma}
\begin{proof}
Let $M^{\circ}$ denote $M \setminus \Int(D^d)$, where $D^d \hookrightarrow M$ is a fixed embedded disc. We also fix an embedding $D^d \hookrightarrow \Int(M^\circ)$. The proof boils down to showing that the quotient group
$$G \defeq \faktor{\pi_0 \Diff_\partial (M^{\circ})}{\left(\mathrm{Im}(\pi_0\Diff_\partial (D^d)\to \pi_0 \Diff_\partial(M^{\circ})\right)}$$
is a residually finite group: indeed, if that is the case, we take $g \in \pi_0 \Diff_{\partial}(M^{\circ})$ which does not lie in the image of $\pi_0 \Diff_{\partial}(D^d)$. This $g$ maps to a non-trivial element in the quotient group, and if the quotient group $G$ is residually finite, this element can be detected by a finite group, hence is not in $\mathrm{fr}(\pi_0 \Diff_{\partial}(M^{\circ}))$. To show that $G$ is residually finite, we consider the smooth Weiss fibre sequence
$$B\Diff_{\partial}(D^d) \to B\Diff_\partial(M^{\circ}) \to B\Emb_{\partial/2}^{\cong}(M^{\circ},M^{\circ})$$
which, on homotopy groups, yields an exact sequence
$$\pi_0 \Diff_{\partial}(D^d) \to \pi_0\Diff_{\partial}(M^{\circ}) \to \pi_0 \Emb_{\partial/2}^{\cong}(M^{\circ},M^{\circ})$$
As a consequence, we observe that the image of the morphism $\pi_0 \Diff_{\partial}(D^d) \to \pi_0 \Diff_{\partial}(M^{\circ})$ is a normal subgroup, and hence the quotient is indeed a group. By \cref{prop:embedding-rel-bdry-resfinite}, we know that $\pi_0 \Emb_{\partial/2}^{\cong}(M^{\circ},M^{\circ})$ is residually finite. Thus, the quotient group $G$, which is identified as a subgroup of $\pi_0 \Emb_{\partial/2}^{\cong}(M^{\circ},M^{\circ})$, is consequently itself a residually finite group, as subgroups of residually finite groups are residually finite.
\end{proof}
We now study the case of the mapping class groups of $M$, $\pi_0 \Diff(M)$.

\begin{theorem}\label{thm:fin-res-smooth}
    Let $M$ be a smooth, closed 2-connected manifold of dimension $d \geq 6$, and fix an embedded disc $D^d \subset M$. Then,
    \begin{enumerate}[(i)]
        \item $\mathrm{fr}(\pi_0 \Diff^+ (M)) \subseteq \mathrm{Im}(\pi_0 \Diff_\partial(D^d) \to \pi_0 \Diff^+(M))$, where the map $\pi_0 \Diff_\partial (D^d) \to \pi_0 \Diff^+(M)$ is given by extension by the identity;
        \item $\mathrm{fr}(\pi_0 \Diff(M)) \subseteq \mathrm{Im}(\pi_0 \Diff_\partial(D^d) \to \pi_0 \Diff(M))$, where the map $\pi_0 \Diff_\partial (D^d) \to \pi_0 \Diff(M)$ is given by extension by the identity
    \end{enumerate}
\end{theorem}
\begin{proof}
    We begin by showing (i). For the proof, we make use of the following three fibre sequences below, where $M$ is a smooth, 2-connected closed manifold, and where $M^{\circ}$ denotes, as usual, $M$ with the interior of an embedded disc deleted; this new disc is taken to be disjoint from the fixed disc $D^d \subset M$. These fibre sequences are the following:
    \begin{equation}\label{eq:fibseq1}
        B\Diff_{\partial}(D^d) \to B\Diff_\partial (M^{\circ}) \to B\Emb_{\partial/2}^{\cong}(M^{\circ},M^{\circ})
    \end{equation}
    \begin{equation}\label{eq:fibseq2}
        B\Diff_{\partial}(M^{\circ}) \to B\Diff^+(M,\ast) \to BSO(d)
    \end{equation}
    \begin{equation}\label{eq:fibseq3}
        M \to B\Diff^+(M,\ast) \to B\Diff^+(M)
    \end{equation}

The fibre sequence (\ref{eq:fibseq1}) is the Weiss fibre sequence, and fibre sequences (\ref{eq:fibseq2}) and (\ref{eq:fibseq3}) are as in \cite[equation (6) and (10)]{ManuelFibrations}; $\Diff^+(M,\ast)$ denotes the group of orientation preserving diffeomorphisms that fix a basepoint of $M$, and the morphism $\Diff_\partial(M^{\circ}) \to \Diff^+(M,\ast)$ is given by extension by the identity. As $BSO(d)$ is 1-connected, and $M$ is 2-connected, it follows that the morphism $B\Diff_\partial(M) \to B\Diff^+(M)$ is surjective on $\pi_1$; therefore, the morphism $\pi_0 \Diff_\partial (M^{\circ}) \to \pi_0 \Diff^+ (M)$ is surjective. Additionally, we observe that the image of $\pi_0 \Diff_\partial(D^d) \cong \Theta_{d+1}$ in $\pi_0 \Diff_\partial (M^{\circ})$ is normal, as can be seen from the exact sequence on $\pi_1$ for the fibre sequence (\ref{eq:fibseq1}). Thus, the image of $\pi_0 \Diff_\partial(D^d)$ in $\pi_0 \Diff^+(M)$ is also normal, as surjective group homormophisms send normal subgroups to normal subgroups. We thus obtain a group morphism
$$\frac{\pi_0 \Diff_\partial(M^{\circ})}{\Theta_{d+1}} \to \frac{\pi_0 \Diff^+(M)}{\Theta_{d+1}}$$
where by the quotient by $\Theta_{d+1}$, we mean the quotient by the image of $\Theta_{d+1}$ in each of the respective groups. By \cref{lemma:fin-res-smooth}, we know that the group on the left hand side is residually finite. We claim that the above morphism is an isomorphism. Surjectivity follows from the fact that the group morphism $\pi_0 \Diff_\partial (M^{\circ}) \to \pi_0 \Diff^+(M)$ is surjective prior to taking quotients. For injectivity, we consider an element $[g] \in \frac{\pi_0 \Diff_{\partial}(M^{\circ})}{\Theta_{d+1}}$, represented by an isotopy class of diffeomorphisms in $\pi_0 \Diff_\partial (M^{\circ})$, that is mapped to the trivial element in $\frac{\pi_0 \Diff^+(M)}{\Theta_{d+1}}$. Thus, the diffeomorphism $g \cup_{\partial}\mathrm{id}$, given by extension by the identity applied to $g$, is in the image of $\Theta_{d+1}$. However, we have a commutative diagram
\begin{center}
    \begin{tikzcd}
        \Theta_{d+1}\cong \pi_0 \Diff_{\partial}(D^d) \arrow{dr}\arrow{r} & \pi_0\Diff_\partial(M^{\circ}) \arrow{d}
        \\
        & \pi_0 \Diff^+(M)
    \end{tikzcd}
\end{center}
where all morphisms are given by extensions by the identity. Thus, $g$ is itself in the image of $\Theta_{d+1}$ in $\pi_0 \Diff_{\partial}(M^{\circ})$, so that $[g]=0$ in the quotient. Consequently, $\frac{\pi_0 \Diff^+(M)}{\Theta_{d+1}}$ is residually finite, and (i) follows.
\\

We now show (ii). The non-oriented analogue of fibre sequence (\ref{eq:fibseq3}) holds, namely we have a fibre sequence
\begin{equation}\label{eq:fibseq4}
    M \to B\Diff(M,\ast) \to B\Diff(M)
\end{equation}
where $\Diff(M,\ast)$ is the group of diffeomorphisms preserving a basepoint $\ast \in M$. As $M$ is assumed 2-connected, the group morphism $\pi_0 \Diff(M,\ast) \to \pi_0 \Diff(M)$ is an isomorphism. As a consequence, we observe that $\mathrm{Im}(\Theta_{d+1} \to \pi_0 \Diff(M))$ is a normal subgroup of $\pi_0 \Diff(M)$: indeed, any diffeomorphism of $M$ is isotopic to a diffeomorphism that fixes a point, and hence fixes a neighborhood disc around that point setwise. As a consequence, we have a group morphism
$$\frac{\pi_0 \Diff_\partial (M^{\circ})}{\Theta_{d+1}} \to \frac{\pi_0 \Diff(M)}{\Theta_{d+1}}$$
where we use the same convention regarding the quotient by $\Theta_{d+1}$ as in the proof of (i). The above group morphism is injective, for the same argument above; however, it may fail to be surjective. Nonetheless, it embedds $\frac{\pi_0 \Diff_{\partial}(M^{\circ})}{\Theta_{d+1}}$ as a finite index subgroup of $\frac{\pi_0 \Diff(M)}{\Theta_{d+1}}$. As the domain of the map is residually finite, \cref{lemma:finite-index-resfin} implies the target is also residually finite, and (ii) follows. 
\end{proof}

As a corollary, it follows that $\pi_0 \Diff(M)$ and $\pi_0 \Diff^+(M)$ are residually finite, for $M$ a closed smooth 2-connected manifold of dimension $d \geq 5$, whenever $\Theta_{d+1}$ is the trivial group. By \cite[Corollary 1.3]{trivialexotic}, this holds for $d\in \{5,11,55,60\}$, for instance.
\bigskip
\newpage
\appendix
\section{Diagrams of finite completion}\label{appendix}
The content of this appendix builds up to a proof of \cref{thm:profinite} and \cref{cor:profinite}. We wish to thank Thomas Blom for the main ideas presented here.

\begin{definition}[Tensoring/cotensoring]
    We consider the following two functors
    \begin{align*}
        \mathcal{S} \times \Pro(\mathcal{S}_\pi) &  \to \Pro(\mathcal{S}_\pi)
        \\
        (K,X) & \mapsto K\otimes X \coloneqq \colim_K X
    \end{align*}
    and
    \begin{align*}
        \mathcal{S}^{\op} \times \Pro(\mathcal{S}_\pi) &\to \Pro(\mathcal{S}_\pi)
        \\
        (K,X) & \mapsto X^K \coloneqq \lim_K X
    \end{align*}
    both of which are taken over the constant diagrams.
\end{definition}

The following couple of lemmas consist of showing that the above tensoring-cotensoring define a $\Pro(\mathcal{S}_\pi)$-enriched mapping space between $\widehat X$ and $\widehat Y$, whenever $X$ and $Y$ are finite, simply connected spaces, which recovers $\Map_{\Pro(\mathcal{S}_\pi)}(\widehat X, \widehat Y)$ upon materialisation.

\begin{lemma}\label{lemma:cotensoring-mapping-space}
    Let $K$ be a space, and $X$ a profinite space. Then,
    $$\Map_{\Pro(\mathcal{S}_\pi)}(\widehat{K},X) \simeq \Mat(X^K)$$
\end{lemma}
\begin{proof}
    By adjunction, we have an equivalence
    $$\Map_{\Pro(\mathcal{S_\pi})}(\widehat{K},X) \simeq \Map_\mathcal{S}(K,\Mat(X))$$
    In spaces, we may write $K$ as a colimit over the constant diagram
    $$K \simeq \colim_K \ast$$
    which in turns gives an equivalence
    $$\Map_\mathcal{S}(K,\Mat(X)) \simeq \lim_K \Map_{\mathcal{S}}(\ast,\Mat(X)) \simeq \lim_K \Mat(X)$$
    As materialisation is a right adjoint, it preserves limits, and we consequently obtain the desired equivalence
    $$\Map_{\Pro(\mathcal{S_\pi})}(\widehat{K},X) \simeq \Mat(X^K)$$
\end{proof}

\begin{lemma}\label{lemma:tensoring-finite-spaces}
    Let $X \in \Pro(\mathcal{S}_\pi)$, and fix $K$ a simply connected, finite space. Then, $K \otimes X \simeq \widehat{K} \times X$, naturally in $K$ and $X$.
\end{lemma}
\begin{proof}
    By definition, $X \coloneqq \lim_i X_i$, where $X_i$ is a cofiltered system of $\pi$-finite spaces, and where the limit is meant as a formal limit. Then, since cofiltered limits commute with finite colimits, we get the equivalence
    \begin{align*}
        K \otimes X & \simeq \colim_K \lim_i X_i
        \\
        & \simeq \lim_i \colim_K X_i \simeq \lim_i K \otimes X_i
    \end{align*}
    We now check that for a fixed $\pi$-finite space $X_i$, $K \otimes X_i  \simeq X_i \times \widehat K $ where we view $X_i$ as a constant $\Pro$-object in $\mathcal{S}_\pi$. By the Yoneda lemma, it suffices to show that for all $Y \in \Pro(\mathcal{S}_\pi)$, we have a natural equivalence
    $$\Map_{\Pro(\mathcal{S}_\pi)}(K \otimes X_i , Y) \simeq \Map_{\Pro(\mathcal{S}_\pi)}(\widehat K \times X_i, Y)$$
    which we may further restrict to the case where $Y$ is a $\pi$-finite space, viewed as a constant $\Pro$-object in $\mathcal{S}_\pi$. We have the following chain of natural equivalences
    \begin{align*}
        \Map_{\Pro(\mathcal{S}_\pi)}(K \otimes X_i ,Y) & \stackrel{\circled{1}}{\simeq} \Map_{\Pro(\mathcal{S}_\pi)}(X_i, Y^K)
        \\
        &\stackrel{\circled{2}}{\simeq}\Map_{\mathcal{S}}(X_i,\Mat(Y^K))
        \\
        &\stackrel{\circled{3}}{\simeq} \Map_{\mathcal{S}}(X_i,\Map_{\Pro(\mathcal{S}_\pi)}(\widehat K, Y))
        \\
        &\stackrel{\circled{4}}{\simeq}\Map_{\mathcal{S}}(X_i,\Map_{\mathcal{S}}( K, Y))
        \\
        &\stackrel{\circled{5}}{\simeq} \Map_\mathcal{S}(X_i \times K, Y)
        \\
        &\stackrel{\circled{6}}{\simeq} \Map_{\Pro(\mathcal{S}_\pi)}(\widehat{K \times X_i},Y)
    \end{align*}
    where
    \begin{itemize}
        \item $\circled{1}$ is given by the tensoring-cotensoring adjunction;
        \item $\circled{2}$ follows from the fact that $X_i = \widehat{X_i}$, as $X_i$ is a $\pi$-finite space;
        \item $\circled{3}$ follows from  \cref{lemma:cotensoring-mapping-space};
        \item $\circled{4}$ is obtained from the adjunction $\widehat{(-)} \dashv \Mat$ adjunction, along with the fact that $\Mat(\widehat{Y}) \simeq Y$, since $Y$ is $\pi$-finite;
        \item $\circled{5}$ follows from the usual tensor-hom adjunction in $\mathcal{S}$;
        \item $\circled{6}$ is again obtained from the adjunction $\widehat{(-)} \dashv \Mat$
    \end{itemize}

    As a consequence, we obtain an equivalence $K \otimes X_i \simeq \widehat{K \times X_i}$. We conclude by showing that the natural map
    $$\widehat{K \times X_i} \to \widehat{K} \times \widehat{X_i}=\widehat{K} \times X_i$$
    is an equivalence. First, observe that $\pi_0(K \times X_i) \cong \pi_0 X_i$ is a finite set, as $K$ is a connected space. As (space) finite completion induces the set profinite completion on $\pi_0$, it follows that $\pi_0 (\widehat {K \times X_i}) \cong \pi_0 (\widehat{K} \times X_i) \cong \pi_0 X_i$. As the functor $\Mat\colon \Pro(\mathcal{S}_\pi) \to \mathcal{S}$ is conservative \cite[thm E.3.1.6]{SAG}, it follows that the above map is an equivalence of profinite spaces if it induces an isomorphism on all homotopy groups, and all choices of basepoints. Now, we note that $K \times X_i$ satisfies Sullivan's conditions for \cite[thm 3.1]{Sullivan-genetics} (\cref{rem:Sullivan3.1}), and we see that (omitting basepoint notation) the map induces the morphism
    $$\Phi^g (\pi_n (K \times X_i)) \to \Phi^g \pi_n K \times \Phi^g \pi_n X_i$$
    Since group profinite completion preserves finite products, it follows that the above is an isomorphism, for all $n \in \mathbb{N}$, and thus that $\widehat{K \times X_i} \simeq \widehat{K} \times X_i$. Combining with the above, we conclude that $K \otimes X_i \simeq \widehat{K} \times X_i$, and thus that $K \otimes X \simeq \widehat{K} \times X$, as desired.
\end{proof}
While showing the above lemma, we also show a separate result of interest; we record it as a separate lemma and reprove it, and point out the similarity between the proofs.
\begin{lemma}\label{lemma:commute_tensors_hats}
    Let $X, Y \in \mathcal{S}$ be arbitrary spaces. Then, there are equivalences
    $$X \otimes \widehat{Y} \simeq \widehat{X \times Y} \simeq Y \otimes \widehat{X}$$
    that are natural in both variables.
\end{lemma}
\begin{proof}
    It suffices to show that we have a natural equivalence $X \otimes \widehat{Y} \simeq \widehat{X \times Y}$. Using the Yoneda lemma, we restrict ourselves to showing that for all $Z \in \Pro(\mathcal{S}_\pi)$, there are natural equivalences
    $$\Map_{\Pro(\mathcal{S}_\pi)}(X \otimes \widehat{Y},Z) \simeq \Map_{\Pro(\mathcal{S}_\pi)}(\widehat{X \times Y},Z)$$
    As $Z$ can be written as a limit of $\pi$-finite spaces $\{Z_\alpha\}$, and as we are mapping into a limit, we may further restrict to the case $Z \in \mathcal{S}_\pi$. We consider the following chain of equivalences
    \begin{align*}
        \Map_{\Pro(\mathcal{S}_\pi)}(X \otimes \widehat{Y}, Z) & \stackrel{\circled{1}}{\simeq} \Map_{\Pro(\mathcal{S}_\pi)}(\widehat{Y},Z^X)
        \\
        & \stackrel{\circled{2}}{\simeq} \Map_{\mathcal{S}}(Y,\Mat(Z^X))
        \\
        &\stackrel{\circled{3}}{\simeq} \Map_{\mathcal{S}}(Y,\Map_{\mathcal{S}}(X,Z))
        \\
        & \stackrel{\circled{4}}{\simeq} \Map_{\mathcal{S}}(X \times Y, Z)
        \\
        & \stackrel{\circled{5}}{\simeq} \Map_{\Pro(\mathcal{S}_\pi)}(\widehat{X \times Y},Z)
    \end{align*}
    where
    \begin{itemize}
        \item $\circled{1}$ follows from the tensoring-cotensoring adjunction;
        \item $\circled{2}$ follows from the adjunction $\widehat{(-)} \dashv \Mat$;
        \item $\circled{3}$ follows from \cref{lemma:cotensoring-mapping-space};
        \item $\circled{4}$ follows from the tensor-hom adjunction in $\mathcal{S}$;
        \item $\circled{5}$ follows from the adjunction $\widehat{(-)} \dashv \Mat$, along with the fact that $Z$ is $\pi$-finite
    \end{itemize}
\end{proof}

We combine the previous two lemmas to obtain an interesting result; it will not play an important role in what follows, but is worth mentioning.
\begin{lemma}
    Let $X \in \mathcal{S}$ be an arbitrary space, and let $Y \in \mathcal{S}_{>1}^{\fin}$ be a finite, 1-connected space. The comparison map
    $$\widehat{X \times Y} \to \widehat{X} \times \widehat{Y}$$
    is an equivalence of profinite spaces.
\end{lemma}
\begin{proof}
    By \cref{lemma:commute_tensors_hats}, we have an equivalence
    $$Y \otimes \widehat{X} \simeq \widehat{X \times Y}$$
    By \cref{lemma:tensoring-finite-spaces}, we have an equivalence
    $$Y \otimes \widehat{X} \simeq \widehat{Y} \times \widehat{X}$$
    and by chasing through the proofs of the above two lemmas, we see that the equivalence is indeed given by the natural map
    $$\widehat{X \times Y} \to \widehat{X} \times \widehat{Y}$$
\end{proof}

Using \cref{lemma:tensoring-finite-spaces}, one can construct a composition morphism on the cotensoring.

\begin{construction}
    Let $X$, $Y$ and $Z$ be simply connected, finite spaces. We recall that we have a composition morphism
    $$-\circ- \colon \Map_{\Pro(\mathcal{S}_\pi)}(\widehat X, \widehat Y) \times \Map_{\Pro(\mathcal{S}_\pi)}(\widehat Y, \widehat Z) \to \Map_{\Pro(\mathcal{S}_\pi)}(\widehat X, \widehat Z)$$
    Similarly, we know that $\Map_{\Pro(\mathcal{S}_\pi)}(\widehat X, \widehat Y) \simeq \Mat(\widehat{Y}^X)$, where $\widehat{Y}^X$ is the cotensoring. We construct a map
    $$-\widehat{\circ}- \colon \widehat{Y}^X \times \widehat{Z}^Y \to \widehat{Z}^X$$
    as follows. We first observe that we have an equivalence
    $$\Map_{\Pro(\mathcal{S}_\pi)}(\widehat{Y}^X \times \widehat{Z}^Y, \widehat{Z}^X) \simeq \Map_{\Pro(\mathcal{S}_\pi)}(X \otimes(\widehat{Y}^X \times \widehat{Z}^Y),\widehat{Z})$$
    and using this equivalence, we construct the desired composition as an element of the mapping space on the right hand side. This is done as follows
    \begin{align*}
        X \otimes (\widehat{Y}^X \times \widehat{Z}^Y) & \stackrel{\circled{1}}{\simeq}\widehat X \times \widehat{Y}^X \times \widehat{Z}^Y
        \\
        &\xrightarrow[]{\circled{2}} \widehat Y \times \widehat{Z}^Y 
        \\
        &\xrightarrow[]{\circled{3}}\widehat Z
    \end{align*}
    where $\circled{1}$ follows from \cref{lemma:tensoring-finite-spaces}, and where $\circled{2}$ and $\circled{3}$ are given by the corresponding image of the identity morphism via the following equivalence
    $$\Map_{\Pro(\mathcal{S}_\pi)}(\widehat{Y}^X,\widehat{Y}^X) \simeq \Map_{\Pro(\mathcal{S}_\pi)}(X \otimes \widehat{Y}^X,\widehat Y) \simeq \Map_{\Pro(\mathcal{S}_\pi)}(\widehat X \times \widehat{Y}^X,\widehat Y)$$
\end{construction}
We have thus constructed, for simply connected finite spaces $X$, $Y$ and $Z$, a map
$$- \widehat{\circ} -\colon \widehat{Y}^X \times \widehat{Z}^Y \to \widehat{Z}^X$$
We now aim at generalizing the above composition morphism to the case of profinite completion of space-valued functors. This will be done using the universal properties of endomorphism objects of Lurie \cite[§4.7.1]{HA}. The setup is as follows: for $\mathcal{C}$ an arbitrary $\infty$-category, and a functor $\mathcal{F} \colon \mathcal{C} \to \mathcal{S}_{>1}^{\fin}$, we denote $\widehat{\mathcal{F}}$ to be the composite
$$\widehat{\mathcal{F}} \colon \mathcal{C} \xrightarrow[]{\mathcal{F}} \mathcal{S}_{>1}^{\fin} \xrightarrow[]{(\widehat{-})} \Pro(\mathcal{S}_\pi)$$
We then show that there is an $E_1$-algebra $\widehat{\mathcal{F}}^{\mathcal{F}}$ in $\Pro(\mathcal{S}_\pi)$ which, upon materialisation, is equivalent to the space of endomorphisms
$$\Map_{\Fun(\mathcal{C},\Pro(\mathcal{S}_\pi))}(\widehat{\mathcal{F}},\widehat{\mathcal{F}})$$
as spaces, and such that both compositions agree.

We begin by constructing the above object, first as an object of $\Pro(\mathcal{S}_\pi)$, and show that its materialisation is equivalent to the mapping space. We later use Lurie's theory of endomorphism objects to check that the object defined promotes to an $E_1$-algebra in $\Pro(\mathcal{S}_\pi)$, and compare the multiplications of $\Mat(\widehat{\mathcal{F}}^{\mathcal{F}})$ to that of $\mathrm{End}(\widehat{\mathcal{F}})$.

\begin{corollary}\label{cor:diagram-finite-completion}
    Let $\mathcal{C}$ be an arbitrary $\infty$-category, and let $\mathcal{F} \colon \mathcal{C} \to \mathcal{S}$, $\mathcal{G} \colon \mathcal{C} \to \Pro(\mathcal{S}_\pi)$ be two functors. Then, there exists an object in $\Pro(\mathcal{S}_\pi)$, denoted $\mathcal{G}^{\mathcal{F}}$, such that
    $$\Map_{\Fun(\mathcal{C},\Pro(\mathcal{S}_\pi))}(\widehat{\mathcal{F}}, \mathcal{G}) \simeq \Mat(\mathcal{G}^{\mathcal{F}})$$
\end{corollary}
\begin{proof}
    We rewrite the mapping space in the category $\Fun(\mathcal{C},\Pro(\mathcal{S}_\pi))$ as a limit over the twisted arrow category of $\mathcal{C}$ (\cite[Prop. 5.1]{twistedarrowcategory})
    \begin{align*}
        \Map_{\Fun(\mathcal{C},\Pro(\mathcal{S}_\pi))}(\widehat{\mathcal{F}}, \mathcal{G}) & \simeq \lim_{(a \to b) \in \mathrm{Tw}(\mathcal{C})}\Map_{\Pro(\mathcal{S}_\pi)}(\widehat{\mathcal{F}(a)},\mathcal{G}(b))
        \\
        &\simeq \lim_{(a \to b) \in \mathrm{Tw}(\mathcal{C})} \Mat(\mathcal{G}(b)^{\mathcal{F}(a)})
    \end{align*}
    Defining $\mathcal{G}^{\mathcal{F}} \in \Pro(\mathcal{S}_\pi)$ as
    $$\mathcal{G}^{\mathcal{F}} \coloneqq \lim_{(a \to b) \in \mathrm{Tw}(\mathcal{C})} \mathcal{G}(b)^{\mathcal{F}(a)}$$
    the claim follows from the fact that $\Mat$ is a right adjoint, hence preserves all limits.
\end{proof}

We now proceed as follows: we fix $\mathcal{F} \colon \mathcal{C} \to \mathcal{S}^{\mathrm{fin}}_{>1}$, a functor from an arbitrary category $\mathcal{C}$ to the category of finite, simply connected spaces; we denote by $\widehat{\mathcal{F}}$ the composite
$$\mathcal{C} \xrightarrow[]{\mathcal{F}} \mathcal{S}^{\mathrm{fin}}_{>1} \xrightarrow[]{\widehat{(-)}}\Pro(\mathcal{S}_\pi)$$
We consider $\widehat{\mathcal{F}}^{\mathcal{F}}$ as the object constructed in the above proof, namely
$$\widehat{\mathcal{F}}^{\mathcal{F}} \coloneqq \lim_{(a \to b) \in \mathrm{Tw}(\mathcal{C})} \widehat{\mathcal{F}(b)}^{\mathcal{F}(a)}$$
By \cref{cor:diagram-finite-completion}, we already know that
$$\Map_{\Fun(\mathcal{C},\Pro(\mathcal{S}_\pi))}(\widehat {\mathcal{F}}, \widehat {\mathcal{F}}) \simeq \Mat(\widehat{\mathcal{F}}^\mathcal{F})$$
We furthermore observe that the space on the left hand side admits a structure of $E_1$-algebra in $\mathcal{S}$ (with respect to the cartesian monoidal structure), where the multiplication is given by composition of natural transformations $\widehat{\mathcal{F}} \Rightarrow \widehat{\mathcal{F}}$. The aim in what follows is to first show that $\widehat{{\mathcal{F}}}^{\mathcal{F}}$ promotes to an $E_1$-algebra in $\Pro(\mathcal{S}_\pi)$ (with respect to the cartesian monoidal structure), and then show that the multiplication map $\widehat{\mu}\colon\widehat{\mathcal{F}}^{\mathcal{F}} \times \widehat{\mathcal{F}}^{\mathcal{F}} \to \widehat{\mathcal{F}}^{\mathcal{F}}$, obtained as part of the datum of the $E_1$-algebra structure on $\widehat{\mathcal{F}}^{\mathcal{F}}$, agrees with the usual composition multiplication on $\mathrm{End}(\mathcal{F})$ upon materialisation; note that $\Mat$, being a right adjoint, in particular preserves finite products. This requires the technology of \cite[§4.7.1]{HA}. In what follows, we give a brief summary of the required results.
\\
Let $\mathcal{C}$ be a monoidal $\infty$-category, and let $\mathcal{M}$ be an $\infty$-category left-tensored over $\mathcal{C}$. Given an object $M\in \mathcal{M}$, an action of an object $C \in \mathcal{C}$ on $M$ is precisely the datum of a map 
$$C \otimes M \to M$$
in $\mathcal{M}$. The endomorphism object of $\mathrm{End}(M)$ is then defined to be universal among objects acting on $M$ in the following sense: if for any $C \in \mathcal{C}$, we have an equivalence
$$\Map_{\mathcal{C}}(C,\mathrm{End}(M)) \simeq \Map_{\mathcal{M}}(C \otimes M,M)$$
which is natural in $C$. Should such an object exits, Lurie shows it admits the structure of an $E_1$-algebra in $\mathcal{C}$ (\cite[Corollary 4.7.1.40]{HA}).
\\

In our setting, we consider the monoidal category $(\Pro(\mathcal{S}_{\pi}),\times)$, and view $\Fun(\mathcal{C},\Pro(\mathcal{S}_\pi))$ as left-tensored over $\Pro(\mathcal{S}_\pi)$, via
    \begin{center}
        $(X,F) \mapsto X\times F$, $X \in \Pro(\mathcal{S}_\pi)$, $F \in \Fun(\mathcal{C},\Pro(\mathcal{S}_\pi))$
    \end{center}

\begin{lemma}\label{lemma:endomorphism_object}
    Let $\mathcal{C}$ be an arbitrary $\infty$-category, and fix a functor $\mathcal{F} \colon \mathcal{C} \to \mathcal{S}_{>1}^{\fin}$ landing in the full subcategory of finite, simply connected spaces. Then, $\widehat{\mathcal{F}}^{\mathcal{F}}$ is an endomorphism object for $\widehat{\mathcal{F}}$ in $\Pro(\mathcal{S}_\pi)$.
\end{lemma}
\begin{proof}
    We have to check the universal property for endomorphism objects. Namely, we need to check that for all $X \in \Pro(\mathcal{S}_\pi)$, we have equivalences
    $$\Map_{\Fun(\mathcal{C},\Pro(\mathcal{S}_\pi))}(X \times \widehat{\mathcal{F}},\widehat{\mathcal{F}}) \simeq \Map_{\Pro(\mathcal{S}_\pi)}(X,\widehat{\mathcal{F}}^{\mathcal{F}})$$
    that are natural in $X$. For this, we consider the following list of equivalences
    \begin{align*}
        \Map_{\Fun(\mathcal{C},\Pro(\mathcal{S}_\pi))}(X \times \widehat{\mathcal{F}},\widehat{\mathcal{F}}) & \stackrel{\circled{1}}{\simeq} \lim_{(a\to b)\mathrm{Tw}(\mathcal{C})} \Map_{\Pro(\mathcal{S}_\pi)}(X \times \widehat{\mathcal{F}(a)},\widehat{\mathcal{F}(b)})
        \\
        & \stackrel{\circled{2}}{\simeq}\lim_{(a\to b)\mathrm{Tw}(\mathcal{C})} \Map_{\Pro(\mathcal{S}_\pi)}(X \otimes \mathcal{F}(a),\widehat{\mathcal{F}(b)})
        \\
        &\stackrel{\circled{3}}{\simeq}\lim_{(a\to b)\mathrm{Tw}(\mathcal{C})} \Map_{\Pro(\mathcal{S}_\pi)}(X,\widehat{\mathcal{F}(b)}^{\mathcal{F}(a)})
        \\
        &\stackrel{\circled{4}}{\simeq}\Map_{\Pro(\mathcal{S}_\pi)}(X,\lim_{(a\to b)\mathrm{Tw}(\mathcal{C})}\widehat{\mathcal{F}(b)}^{\mathcal{F}(a)})
        \\
        &=\Map_{\Pro(\mathcal{S}_\pi)}(X,\widehat{\mathcal{F}}^{\mathcal{F}})
    \end{align*}
    where
    \begin{itemize}
        \item $\circled{1}$ follows from the limit description of mapping spaces in functor categories;
        \item $\circled{2}$ follows from \cref{lemma:tensoring-finite-spaces};
        \item \circled{3} follows from the tensoring-cotensoring adjunction;
        \item \circled{4} follows after pulling the limit into the mapping space
    \end{itemize}
    
    As all the above equivalences are natural in $X$, it follows that $\widehat{\mathcal{F}}^{\mathcal{F}}$ is an endomorphism object for $\widehat{\mathcal{F}}$, and in particular promotes to an $E_1$-algebra in $(\Pro(\mathcal{S}_\pi),\times)$.
\end{proof}
We now stick to the situation of \cref{lemma:endomorphism_object}. We fix a functor $\mathcal{F} \colon \mathcal{C} \to \mathcal{S}_{>1}^{\fin}$ from an $\infty$-category $\mathcal{C}$ to the $\infty$-category of finite, 1-connected spaces. Using \cref{lemma:endomorphism_object}, we obtain an evaluation morphism
$$\mathrm{ev} \colon \widehat{\mathcal{F}}^{\mathcal{F}} \times \widehat{\mathcal{F}} \to \widehat{\mathcal{F}}$$
obtained explicitly as the adjoint of the identity map via the equivalence described in the proof, namely
$$\Map_{\Fun(\mathcal{C},\Pro(\mathcal{S}_\pi))}(\widehat{\mathcal{F}}^{\mathcal{F}} \times \widehat{\mathcal{F}},\widehat{\mathcal{F}}) \simeq \Map_{\Pro(\mathcal{S}_\pi)}(\widehat{\mathcal{F}}^{\mathcal{F}},\widehat{\mathcal{F}}^{\mathcal{F}})$$
The multiplication morphism of the $E_1$-algebra $\widehat{\mathcal{F}}^{\mathcal{F}}$ is explicitly described as the adjoint of the composition
$$ \widehat{\mathcal{F}}^{\mathcal{F}} \times \widehat{\mathcal{F}}^{\mathcal{F}} \times \widehat{\mathcal{F}} \xrightarrow[]{\mathrm{id} \times \mathrm{ev}} \widehat{\mathcal{F}}^{\mathcal{F}} \times \widehat{\mathcal{F}} \xrightarrow[]{\mathrm{ev}} \widehat{\mathcal{F}}$$
which we henceforth denote
$$\widehat{\mu} \colon \widehat{\mathcal{F}}^{\mathcal{F}} \times \widehat{\mathcal{F}}^{\mathcal{F}} \to \widehat{\mathcal{F}}^{\mathcal{F}}$$
We denote $\mathrm{End}(\widehat{\mathcal{F}})$ as the space of endomorphisms
$$\mathrm{End}(\widehat{\mathcal{F}}) \defeq \Map_{\Fun(\mathcal{C},\Pro(\mathcal{S}_\pi))}(\widehat{\mathcal{F}},\widehat{\mathcal{F}})$$
We note that the above is indeed an endomorphism object of $\widehat{\mathcal{F}}$, with respect to the following tensoring: we consider the monoidal $\infty$-category $(\mathcal{S},\times)$, and view $\Fun(\mathcal{C},\Pro(\mathcal{S}_\pi))$ as left-tensored over $\mathcal{S}$ via:
\begin{center}\label{eq:univ-ppty-endo}
    $(X,\mathcal{G}) \mapsto X \otimes \mathcal{G}$, $X \in \mathcal{S}$, $\mathcal{G} \in \Fun(\mathcal{C},\Pro(\mathcal{S}_\pi))$
\end{center}
Similarly to the proof of \cref{lemma:endomorphism_object}, we have natural equivalences
\begin{equation} \label{eq:diagram}
    \Map_{\Fun(\mathcal{C},\Pro(\mathcal{S}_\pi))}(X \otimes \widehat{\mathcal{F}},\widehat{\mathcal{F}}) \simeq \Map_{\mathcal{S}}(X,\mathrm{End}(\widehat{\mathcal{F}}))
\end{equation}
exhibiting $\mathrm{End}(\widehat{\mathcal{F}})$ as the endomorphism object of $\widehat{\mathcal{F}}$ in $\mathcal{S}$. The composition
$$-\circ- \colon \mathrm{End}(\widehat{\mathcal{F}}) \times \mathrm{End}(\widehat{\mathcal{F}}) \to \mathrm{End}(\widehat{\mathcal{F}})$$
thus occurs as the adjoint of applying twice the action map
$$\mathrm{End}(\widehat{\mathcal{F}}) \otimes \widehat{\mathcal{F}} \to \widehat{\mathcal{F}}$$
given by the adjoint of the identity $\mathrm{id} \colon \widehat{\mathcal{F}} \to \widehat{\mathcal{F}}$ in equation \ref{eq:diagram}. Consequently, to compare the multiplication maps, it suffices to compare the action maps. The counit of the adjunction $\widehat{(-)} \dashv \Mat$ yields a map
$$\mathrm{counit}\colon \widehat{\Mat(\widehat{\mathcal{F}}^{\mathcal{F}})} \to \widehat{\mathcal{F}}^{\mathcal{F}}$$
We consider the following commutative diagram

\begin{center}
    \begin{tikzcd}
        \Map_{\Fun(\mathcal{C},\Pro(\mathcal{S}_\pi))}(\widehat{\mathcal{F}}^{\mathcal{F}} \times \widehat{\mathcal{F}},\widehat{\mathcal{F}}) \arrow[no head]{r}{\simeq}[swap]{\circled{1}} \arrow{d}[swap]{-\circ(\mathrm{counit}\times \mathrm{id})} & \Map_{\Pro(\mathcal{S}_{\pi})}(\widehat{\mathcal{F}}^{\mathcal{F}},\widehat{\mathcal{F}}^{\mathcal{F}}) \arrow{d}[swap]{-\circ(\mathrm{counit})} \arrow[bend left=90]{dd}{\Mat}
        \\
        \Map_{\Fun(\mathcal{C},\Pro(\mathcal{S}_\pi))}(\widehat{\Mat(\widehat{\mathcal{F}}^{\mathcal{F}})} \times \widehat{\mathcal{F}},\widehat{\mathcal{F}}) \arrow[no head]{r}{\circled{2}}[swap]{\simeq} \arrow[no head]{d}{\circled{4}}[swap]{\simeq} & \Map_{\Pro(\mathcal{S}_{\pi})}(\widehat{\Mat(\widehat{\mathcal{F}}^{\mathcal{F}})},\widehat{\mathcal{F}}^{\mathcal{F}}) \arrow[no head]{d}{\circled{3}}[swap]{\simeq}
        \\
        \Map_{\Fun(\mathcal{C},\Pro(\mathcal{S}_\pi))}(\mathrm{End}(\widehat{\mathcal{F}}) \otimes \widehat{\mathcal{F}},\widehat{\mathcal{F}}) \arrow[no head]{r}{\simeq}[swap]{\circled{5}} & \Map_{\mathcal{S}}(\mathrm{End}(\widehat{\mathcal{F}}),\mathrm{End}(\widehat{\mathcal{F}}))
    \end{tikzcd}
\end{center}
where
\begin{itemize}
    \item $\circled{1}$ and $\circled{2}$ are the natural equivalences constructed in the proof of \cref{lemma:endomorphism_object};
    \item $\circled{3}$ follows from the adjunction $\widehat{(-)} \dashv \Mat$, and \cref{cor:diagram-finite-completion};
    \item $\circled{4}$ is obtained by combining the equivalences
    $$\mathrm{End}(\widehat{\mathcal{F}}) \otimes \widehat{\mathcal{F}} \simeq \mathcal{F} \otimes \widehat{\mathrm{End}(\widehat{\mathcal{F}})} \simeq \widehat{\Mat(\widehat{\mathcal{F}}^{\mathcal{F}})} \times \widehat{\mathcal{F}}$$
    which follow from \cref{lemma:tensoring-finite-spaces,lemma:commute_tensors_hats}
    \item $\circled{5}$ follows from equation \ref{eq:diagram}
\end{itemize}

With the above diagram at hand, we see that the action map $\mathrm{ev} \colon \widehat{\mathcal{F}}^{\mathcal{F}} \times \widehat{\mathcal{F}} \to \widehat{\mathcal{F}}$ exhibiting $\widehat{\mathcal{F}}^{\mathcal{F}}$ as an endomorphism object of $\widehat{\mathcal{F}}$ in $\Pro(\mathcal{S}_\pi)$, and which is adjoint to the identity in $\Map_{\Pro(\mathcal{S}_\pi)}(\widehat{\mathcal{F}}^{\mathcal{F}},\widehat{\mathcal{F}}^{\mathcal{F}})$, is mapped by materialisation to the action map $\mathrm{End}(\widehat{\mathcal{F}}) \otimes \widehat{\mathcal{F}} \to \widehat{\mathcal{F}}$, exhibiting $\mathrm{End}(\widehat{\mathcal{F}})$ as an endomorphism object of $\widehat{\mathcal{F}}$ in $\mathcal{S}$, and which is adjoint to the identity in $\Map_{\mathcal{S}}(\mathrm{End}(\widehat{\mathcal{F}}),\widehat{\mathcal{F}})$. As a consequence, the two multiplication morphisms on $\mathrm{End}(\widehat{\mathcal{F}}) \simeq \Mat(\widehat{\mathcal{F}}^\mathcal{F})$, namely the usual composition $-\circ-$ and $\Mat(\widehat{\mu})$ agree.

Combining the results of the section, we have shown the following
\begin{lemma}\label{lemma:composition}
    Let $\mathcal{F} \colon \mathcal{C} \to \mathcal{S}_{>1}^{\fin}$. Then, we have a commutative diagram
    \begin{center}
        \begin{tikzcd}
            \Mat(\widehat{\mathcal{F}}^{\mathcal{F}}) \times \Mat(\widehat{\mathcal{F}}^{\mathcal{F}}) \arrow{d}[swap]{\Mat(\widehat{\mu})} \arrow[no head]{r}{\simeq}  &  \mathrm{End}(\widehat{\mathcal{F}}) \times \mathrm{End}(\widehat{\mathcal{F}}) \arrow{d}{-\circ-}
            \\
            \Mat(\widehat{\mathcal{F}}^{\mathcal{F}}) \arrow[no head]{r}[swap]{\simeq} &  \mathrm{End}(\widehat{\mathcal{F}})
        \end{tikzcd}
    \end{center}
\end{lemma}

We observe that the techinques described above are better aimed at showing that $\pi_0 \mathrm{End}(\widehat{\mathcal{F}})$ admits the structure of a monoid in profinite spaces. The following lemma serves as a bridge to conclude that $\pi_0 \mathrm{Aut}(\widehat{\mathcal{F}})$, i.e. the group of units of the above monoid, admits the structure of a profinite group compatible with that of its monoid.
\begin{lemma}\label{lemma:group-of-units-profinite}
    Let $M$ be a monoid in $\Pro(\mathcal{S}_\pi)$, i.e. a topological monoid whose underlying topological space is a Stone space. Then, $\mathcal{u}(M)$, its group of units, admits the structure of a profinite group via the subspace topology.
\end{lemma}
\begin{proof}
    We begin by observing that the category of profinite monoids, i.e. pro-objects in the category $\mathrm{Mon}^{\fin}$, is equivalent to the category $\mathrm{Prof}(\mathrm{Mon})$, the category of topological monoids with underlying topological space a compact, Hausdorff, totally disconnected space; this follows from \cite[Chapter VI, §2.9]{johnstone}. More precisely, the functor $\Pro(\mathrm{Mon}^\fin) \to \mathrm{Prof}(\mathrm{Mon})$, sending a cofiltered system of finite monoids to its limit, where we endow the limit with the profinite topology, is an equivalence of categories. The anologous statement in the case of profinite groups is classical, but a proof can also be found in \textit{loc.cit}. Consider the adjunction $\iota \dashv \mathcal{u}$:
    \begin{center}
        \begin{tikzcd}
            \Grp \arrow[hook]{r}{\iota} & \arrow[bend left]{l}{\mathcal{u}} \mathrm{Mon}
        \end{tikzcd}
    \end{center}
    where $\Grp$ is the category of discrete groups and group homomorphisms, $\mathrm{Mon}$ is the category of discrete monoids and monoid homomorphisms, $\iota$ is the inclusion functor and $\mathcal{u}$ is the group of units functor. The above adjunction restricts to the analogous one on finite groups and monoids. Applying $\Pro\colon \mathrm{Cat}_\infty \to \mathrm{Cat}_\infty$, we obtain a diagram
    \begin{center}
        \begin{tikzcd}
            \Pro(\Grp) \arrow{r}{\Pro(\iota)} & \arrow[bend left]{l}{\Pro(\mathcal{u})} \Pro(\mathrm{Mon})
        \end{tikzcd}
    \end{center}
    which can also be seen to be an adjunction, as $\iota$ itself preserves all limits. As a consequence, we obtain a commutative diagram
    \begin{center}
        \begin{tikzcd}
            \mathrm{Prof}(\mathrm{Mon}) \simeq\Pro(\mathrm{Mon}^\fin) \arrow{r}{\Pro(\mathcal{u})} \arrow{d} & \Pro(\Grp^\fin) \simeq \mathrm{Prof}(\Grp) \arrow{d}
            \\
            \mathrm{Mon} \arrow{r}[swap]{\mathcal{u}} & \Grp
        \end{tikzcd}
    \end{center}
    which in particular implies that given a profinite monoid $M$, the group of units of its underlying discrete monoid lifts to profinite groups via the above diagram.
\end{proof}
We are now ready to show \cref{thm:profinite}.

\begin{proof}[Proof of \cref{thm:profinite}]
    The goal is to show that 
    $$G \defeq \pi_0 \Map^{\simeq}_{\mathrm{Fun}(\mathcal{C},\mathrm{Pro}(\mathcal{S}_\pi))}(\widehat{\mathcal{F}}, \widehat{\mathcal{F}})$$
    admits the structure of a profinite group. Let $M \defeq \pi_0 \Map_{\Fun(\mathcal{C},\Pro(\mathcal{S}_\pi))}(\widehat{\mathcal{F}},\widehat{\mathcal{F}})$, and observe that $G$ is the group of units of $M$. \cref{lemma:group-of-units-profinite} gives therefore a reduction to showing that $M$ is a profinite monoid. 
    The latter fact follows readily from \cref{lemma:composition}.
\end{proof}

\nocite{Goodwillie-Weiss}
\nocite{HA}
\nocite{Hilton-Roitberg}
\nocite{HTT}
\nocite{Manuel-Oscar-arithmetic}
\nocite{Manuel-Sander-Copenhagen}
\nocite{Nikolov-Segal}
\nocite{Piotr}
\nocite{SAG}
\nocite{Sander-finiteness}
\nocite{Serre-arithmetic}
\nocite{Sullivan-genetics}
\nocite{Sullivan-infinitesimal}
\nocite{Weiss-Dalian}
\nocite{discstructurespace}
\nocite{Kirby-Siebenmann}
\nocite{Lashof}
\nocite{Deligne}
\nocite{Bousfield-Kan}
\nocite{Schneebeli}
\nocite{paratopological-group}
\nocite{MayPonto}
\nocite{Weiss}
\nocite{Jan-Maxime}
\nocite{Sander-boundary}
\nocite{Kervaire-Milnor}
\nocite{twistedarrowcategory}
\nocite{ManuelFibrations}
\nocite{trivialexotic}
\nocite{johnstone}

\vspace{-0.1cm}
\bibliographystyle{amsalpha}
\bibliography{literature}

\end{document}